\documentclass[11pt]{amsart} 
\usepackage{amsaddr} 

\usepackage[OT1]{fontenc}
\usepackage{amsthm,amsmath}
\usepackage[round]{natbib}
\usepackage[colorlinks,citecolor=blue,urlcolor=blue,linkcolor=blue,breaklinks=true]{hyperref}
\usepackage[numbered]{bookmark} 

\usepackage{amssymb}
\usepackage{tikz}
\usepackage{cancel}
\usepackage{grffile}
\usepackage{breakcites}

\usepackage{enumerate}
\usepackage[shortlabels]{enumitem}
\usepackage{soul}

\numberwithin{equation}{section}
\theoremstyle{plain}
\newtheorem{thm}{Theorem}[section]

\usepackage{enumitem}
\usepackage{highdimdag}

\title[Learning Directed Acyclic Graphs]{Learning Directed Acyclic Graphs with Penalized Neighbourhood Regression}
\author[Aragam, Amini, and Zhou]{Bryon Aragam,$^1$ Arash A. Amini,$^2$ \and Qing Zhou$^3$}
\address{University of California, Los Angeles}
\email{$^1$naragam@cs.cmu.edu, $^2$aaamini@stat.ucla.edu, $^3$zhou@stat.ucla.edu}

\begin{document}
\maketitle

\begin{abstract} 

We study a family of regularized score-based estimators for learning the structure of a directed acyclic graph (DAG) for a multivariate normal distribution from high-dimensional data with $p\gg n$. Our main results establish support recovery guarantees and deviation bounds for a family of penalized least-squares estimators under concave regularization without assuming prior knowledge of a variable ordering. These results apply to a variety of practical situations that allow for arbitrary nondegenerate covariance structures as well as many popular regularizers including the MCP, SCAD, $\ell_{0}$ and $\ell_{1}$. The proof relies on interpreting a DAG as a recursive linear structural equation model, which reduces the estimation problem to a series of neighbourhood regressions. We provide a novel statistical analysis of these neighbourhood problems, establishing uniform control over the superexponential family of neighbourhoods associated with a Gaussian distribution. We then apply these results to study the statistical properties of score-based DAG estimators, learning causal DAGs, and inferring conditional independence relations via graphical models. Our results yield---for the first time---finite-sample guarantees for structure learning of Gaussian DAGs in high-dimensions via score-based estimation.
\end{abstract}


\section{Introduction}
\label{sec:intro}
The problem of learning directed acyclic graphs from observational data has received increased attention recently due to its many applications in causal inference, genomics, machine learning, and theoretical computer science. While there are several competing strategies for solving this problem, a well-established line of work starting with \citet{heckerman1995} has illustrated that so-called \emph{score-based learning} outperforms existing approaches and scales to high-dimensional datasets with thousands of variables \citep{aragam2015,ramsey2016}. The basic idea behind score-based learning is to formulate the problem as a statistical $M$-estimation problem of the form
\begin{align}
\label{eq:intro:mainestimator}
\dagadjest
\,\in\, \argmin_{\gendag\,\in\,\dagspace}\plsobj(\gendag),
\end{align}

\noindent
where $\dagspace$ is the set of $p\times p$ matrices representing the weighted adjacency matrix of a directed acyclic graph (DAG) and $\plsobj$ is a \emph{score function}, usually combining a data-dependent loss function, based on a sample of size $n$, and a model complexity regularizer. 

Unfortunately, due to the nonconvex and combinatorial nature of the program \eqref{eq:intro:mainestimator}, very little is known about the statistical properties of score-based estimators in the high-dimensional regime with $p\gg n$. In contrast, the development and application of algorithms for high-dimensional score-based learning has surged in recent years. This has created a gap in the literature, whereby score-based methods based on \eqref{eq:intro:mainestimator} are commonly used by practitioners despite an absence of guarantees regarding the quality of such estimators. A goal of this paper is to study the statistical properties of $\dagadjest$, motivated by applications to the theory of high-dimensional $M$-estimation, causal inference, and graphical modeling. Our main results establish---for the first time---structure learning and estimation guarantees for the score-based estimator $\dagadjest$ in high-dimensions that avoid commonly used but uncheckable assumptions found in the literature, giving a sense of what happens when a score-based method is na\"ively applied to high-dimensional datasets.

\subsection{Background}
\label{subsec:intro:background}
Our approach is based on the well-known structural equation model (SEM) interpretation of Gaussian DAGs \citep{wright1934,drton2011}. Suppose $\dagvec=(\dagvec_{1},\ldots,\dagvec_{p})$  is a random vector satisfying
\begin{align}
\label{eq:intro:streqnmodel}
\dagvec = \dagnopi^{T}\dagvec + \dagerr,
\quad \dagerr\sim \normalN_{p}(0,\varnopi),
\end{align}

\noindent
where $\dagnopi\in\dagspace$ and $\varnopi$ is a $p\times p$ positive diagonal matrix of variances. One can interpret $\dagnopi$ as the weighted adjacency matrix of a graph. We will say that $\dagnopi$ is, or represents, a DAG if this graph is directed and acyclic. Given an $n\times p$ random matrix $\sampledagmat$ whose rows are i.i.d. drawn according to the model \eqref{eq:intro:streqnmodel}, define a penalized least-squares (PLS) score function by
\begin{align}
\label{eq:def:plsobj}
\plsobj(\gendag)
= \frac{1}{2n}\Norm{\sampledagmat-\sampledagmat\gendag}_{\frob}^{2} + \pl(\gendag),
\end{align}
where $\Norm{\cdot}_{\frob}$ is the matrix Frobenius norm and $\pl$ belongs to a class of regularizers, parametrized by the \emph{regularization parameter} $\lambda$. The corresponding score-based estimator, defined by \eqref{eq:intro:mainestimator}, is the main focus of this article. In the sequel, $\dagadjest$ always refers to the estimator defined by using \eqref{eq:def:plsobj} in \eqref{eq:intro:mainestimator}.

The objective function \eqref{eq:def:plsobj} is not new---there is a long line of previous work based on the same objective function \citep[e.g.][]{schmidt2007,shojaie2010,xiang2013,han2016}. Unfortunately, these works are mainly empirical and leave many open questions regarding the statistical properties of $\dagadjest$ unanswered. Given its popularity in practice, our goal here is to analyze the statistical properties of $\dagadjest$.

\subsection{Motivation}
\label{subsec:intro:motivation}

The decomposition \eqref{eq:intro:streqnmodel} is often referred to as a linear \emph{structural equation model} for $\dagvec$, and extends well beyond the Gaussian setting considered here. Structural equations have a long history in the statistical literature, dating all the way back to \citet{wright1921,wright1934}. More recently, structural equations have become an essential quantitative analysis tool in the social sciences \citep{yuan2007sem}. Structural equations have also played an important role in machine learning and causal inference \citep{buhlmann2014, mooij2014} as well as covariance selection in statistics \citep{pourahmadi2013,wermuth1980}. These models have been applied recently in biology \citep{cai2013}, network modeling \citep{horvath2011}, and psychology \citep{guay2015,morin2015}.

Motivated by these applications and recent interest in developing fast algorithms for learning DAGs \citep{schmidt2007,teyssier2012,fu2013,xiang2013,aragam2015,gu2016,ramsey2016}, we seek to address three fundamental questions regarding score-based estimators:
\begin{enumerate}
\item \emph{Score-based learning.} It is well-known that DAG models are nonidentifiable (Section~\ref{subsec:prelim:perms})---nonetheless, is it possible to obtain statistical guarantees for $\dagadjest$ when $p\gg n$? 
\item \emph{Causal DAG learning.} In some situations, there exists a unique causal DAG that represents the true, underlying distribution. In such situations, can $\dagadjest$ recover the structure of the causal DAG?
\item \emph{Conditional independence learning.} What can be said about the conditional independence (CI) relations inferred by $\dagadjest$, particularly when there is no faithful DAG that represents the joint distribution?
\end{enumerate}
Crucially, in all three problems, we are interested in \emph{structure learning}  (also known as \emph{support recovery} and \emph{model selection consistency}), i.e. recovering the correct graph. In graphical modeling, stucture learning is of fundamental importance since it is a critical step for learning the conditional independence structure of the underlying distribution; $\ell_{2}$-consistency is insufficient for this purpose. Typically, the \emph{faithfulness} assumption is needed (Section~\ref{subsec:app:ci}), which is a strong assumption that rarely holds in practice \citep{lin2014,uhler2013}. Our analysis avoids this assumption entirely, and to the best of our knowledge, provides the first structure learning guarantees for DAGs \emph{without faithfulness} in high-dimensions.

The organization of the rest of this paper is as follows: In the rest of this section, we review previous work and outline our main contributions. Section~\ref{sec:prelim} provides some intuition behind the estimator \eqref{eq:intro:mainestimator} and covers necessary preliminaries.  Section~\ref{sec:nhbd} sets up a framework for neighbourhood regression that is fundamental to our main results. In Section~\ref{sec:results}, we present our main results, establishing uniform nonasymptotic bounds on the probability of false selection and the estimation error of a family of penalized least-squares estimators. In Section~\ref{sec:app}, we apply these results to answer each of the three questions outlined above. Finally, we compare our results to the literature in Section~\ref{subsec:app:comparison} before concluding with some extensions and generalizations in Section~\ref{sec:conclusion}. As the proofs are quite involved and invoke novel technical machinery, we outline the main ideas in Section~\ref{sec:proofs} and postpone detailed proofs of the various technical results to the Appendix. 

\subsection{Previous work}
\label{subsec:intro:prevwork}

Despite substantial methodological and algorithmic progress towards solving \eqref{eq:intro:mainestimator} when $p$ is large, theoretical progress in high-dimensions has been slower.
\citet{chickering2003} showed that in a low-dimen\-sional setting with $p$ fixed, score-based learning is sufficient to learn a so-called \emph{faithful} DAG for $X$ in the asymptotic limit $n\to\infty$, where $n$ is the number of i.i.d. samples. The first truly high-dimensional results for DAG learning were for the PC algorithm \citep{kalisch2007}, however, this algorithm is not based on \eqref{eq:intro:mainestimator} and requires strong assumptions such as faithfulness. \cite{geer2013} proved the first such results for a score-based estimator, establishing bounds on the $\ell_{2}$-error and the number of edges in $\hB$ for a thresholded $\ell_{0}$-penalized maximum likelihood estimator. Notably, their results fall short of providing consistency in structure learning. One of the main contributions of the current work is a unified regression framework for analyzing score-based estimators that provides guarantees on support recovery, parameter estimation, and sparsity. Furthermore, our work covers a wide spectrum of nonconvex regularizers including both $\ell_0$ and $\ell_1$ as boundary cases. A more detailed comparison to the aforementioned results is provided in Section~\ref{subsec:app:comparison}.

Finally, we note that---perhaps surprisingly---there has been much more theoretical progress for nonlinear and non-Gaussian models \citep{peters2012nonlinear,peters2014,mooij2014}. There has also been progress using multi-stage methods that separate the learning procedure into several decoupled steps \citep{loh2014,buhlmann2014}. These results are largely due to various assumptions that guarantee identifiability \citep[e.g.][]{shimizu2006,hoyer2009,zhang2009}. Under such assumptions, the analysis becomes substantially easier. By contrast, the present work studies a fully nonidentifiable model as in \cite{geer2013}, and as such the main complication in the theoretical analysis is dealing with the existence of equivalent DAGs as discussed in Section~\ref{subsec:prelim:perms}. Furthermore, our approach via score-based learning consists of a single learning step that effectively combines neighbourhood search, order search, and regression by minimizing a single penalized loss function.

\subsection{Contributions}
\label{subsec:intro:contrib}

Our main results will be organized into two sections: In Section~\ref{sec:results}, we develop our technical results for neighbourhood regression problems, and then in Section~\ref{sec:app} we apply these results to address the three main questions raised in Section~\ref{subsec:intro:motivation}.

On the technical side, we provide simultaneous, uniform control over the class of DAGs satisfying \eqref{eq:intro:streqnmodel} for a joint Gaussian distribution, including both support recovery guarantees (Section~\ref{subsec:results:supp}) and deviation bounds (Section~\ref{subsec:results:dev}), by leveraging a neighbourhood regression interpretation of the program \eqref{eq:intro:mainestimator}. 
As the cardinality of the class of such DAGs is $O(p!)$, it is very challenging to obtain any uniform control result. We overcome this technical difficulty by introducing a novel notion of monotonicity among the neighbourhood regressions associated with a DAG. This proof technique is completely different from existing work and makes an explicit connection between graphical models and existing concepts from the regression literature. In this way, it sheds new light on the interplay between the covariance matrix $\trueCov$ and its regression coefficients through the intuitive concept of neighbourhood regression.

Besides the technical novelty, the significance of this work is seen from its direct contributions to the following three statistical applications which motivate our analysis:

\begin{enumerate}[leftmargin=0cm,itemindent=.5cm, itemsep=.5em]
\item \emph{Score-based learning} (Section~\ref{subsec:app:scorebased}). Without any identifiability assumptions,
we show that $\dagadjest$ consistently estimates the structure of a sparse DAG representation of the joint Gaussian distribution via \eqref{eq:intro:streqnmodel}, thus providing useful theoretical guarantees for the practical application of score-based DAG learning methods. Our result is intriguing more generally: Whereas much of the literature on high-dimensional statistics has focused on identifiable models, our results indicate that much of this theory carries over to the nonidentifiable setting, and more importantly, it is possible to \emph{adaptively} estimate a well-behaved parameter instead of an arbitrary parameter.

\item \emph{Causal DAG learning} (Section~\ref{subsec:app:causal}).  Under an identifiability assumption, we establish that $\dagadjest$ can be used to learn the ``true'' causal DAG.
We generalize existing results on estimating causal DAGs with equal error variances to what we call \emph{minimum-trace} DAGs. Furthermore, we strengthen existing results to include support recovery in the high-dimensional setting. 

\item \emph{Conditional independence learning} (Section~\ref{subsec:app:ci}). 
We show that it is possible to use our framework to consistently estimate the full set of CI relations in a Gaussian distribution, even in the absence of a faithful DAG representation. Detecting CIs among a set of random variables from data is a key motivation for structure learning of graphical models, including DAGs. Achieving this task without the faithfulness assumption distinguishes our work from the existing literature on graphical models.
\end{enumerate}

\noindent
Notwithstanding these interesting applications to causal inference and graphical models, the first application is interesting in and of itself from a purely statistical standpoint: $\dagadjest$ is a popular, useful estimator and yet we have an extremely limited theoretical understanding of its properties. In this paper we attempt to provide a comprehensive portrait of the behaviour of $\dagadjest$, providing much needed justification---and caution---for its use in applications.

\subsection*{Notation and terminology} 
Universal constants which may not be the same from line to line will be denoted by $c_{1}$, $c_{2}$, etc. The standard $\ell_q$ norms for $q\in[0,\infty]$ will be denoted by $\norm{\cdot}_q$. For a matrix $A \in \R^{n \times p}$,  $\norm{A}_q$ is its $\ell_q$ norm, viewing $A$ as a vector in $\R^{np}$, and so, in particular, the Frobenius norm on matrices is denoted by $\norm{\cdot}_{\frob}$. The maximum and minimum eigenvalues of a matrix $A$ are denoted by $\maxeval(A)$ and $\mineval(A)$, respectively. The support of $A=(a_{ij})$ is defined by $\supp(A):=\{(i,j):a_{ij}\ne0\}$ and the cardinality of a set by $|\cdot|$. We also make use of the column notation $A=[\,a_{1}\|\cdots\|a_{p}\,]\in\R^{n\times p}$ to denote the columns of $A$. Given a permutation $\perm$, $\permmat_\perm$ denotes the associated permutation operator on matrices: For any matrix $A$, $\permmat_\perm A$ is the matrix obtained by permuting the rows and columns of $A$ according to $\perm$, so that $(\permmat_\perm A)_{ij}=a_{\perm(i)\perm(j)}$. For any integer $m$, we define $[m]=\{1,\ldots,m\}$ and $[m]_j = [m] - \{j\}$. For a vector $\genv\in\R^m$ and a subset $S\subset[m]$, we let $\select{v}{S}\in\R^{|S|}$ denote the restriction of $\genv$ to the components in $S$. For a matrix $A\in\R^{n\times m}$, $\select{A}{S}\in\R^{n\times|S|}$ denotes \emph{column-wise} restriction to the columns in $S$. Given two quantities $X$ and $Y$, which may depend on $n$ and $\trueCov$, we will write $X\les Y$ to mean there exists a constant $a>0$---independent of $n$ and $\trueCov$---such that $X\le aY$, and analogously for $X\ges Y$. A positive definite matrix is denoted as $\trueCov \posdef 0$. 

We write $\sampledagmat\iid\normalN_{p}(0,\trueCov)$ for a random matrix $\sampledagmat \in \R^{n \times p}$ to mean that the \emph{rows} are i.i.d. draws from $\normalN_{p}(0,\trueCov)$. Boldface type is reserved for \emph{random} quantities that depend on the sample size $n$; for example a random matrix $\sampledagmat\in\R^{n\times p}$ or the response $\sampley\in\R^{n}$ in a linear model. Events defined on some probability space will usually be denoted by script font, e.g. $\mE$, $\mathcal{F}$, etc.

\section{Preliminaries}
\label{sec:prelim}

In this section we discuss some preliminary notions and develop some intuition for the problem we study. 
Throughout this section, it is important to keep in mind our main goal: To study the consistency of the estimator $\dagadjest$ by decomposing the estimation problem into a series of simpler, neighbourhood regression problems. The present section sets the stage, while the technical details of neighbourhood regression are developed in the next section (Section~\ref{sec:nhbd}). The setup outlined here, including much of the notation, is the same as~\citet{geer2013}. In other words, we work with the same statistical model as in~\citet{geer2013}, however, our results are more general and our proof technique is completely different; for a comparison, see Section~\ref{subsec:app:comparison}.

\subsection{Nonidentifiability and permutation equivalence}
\label{subsec:prelim:perms}

A fundamental property of the model \eqref{eq:intro:streqnmodel} is that the parameters $(\dagnopi,\varnopi)$ are statistically nonidentifiable. It follows from \eqref{eq:intro:streqnmodel} that $X\sim\normalN_{p}(0,\trueCov(\dagnopi,\varnopi))$, where 
\begin{align}
\label{eq:def:sigmafcn}
\trueCov(\dagnopi,\varnopi) 
:= (I-\dagnopi)^{-T}\varnopi(I-\dagnopi)^{-1}.
\end{align}

\noindent
This yields a parametrization of a multivariate normal distribution in terms of a DAG $\dagnopi$ and its conditional variances $\varnopi$. As it turns out, the mapping $(\dagnopi,\varnopi)\mapsto \trueCov(\dagnopi,\varnopi)$ is not one-to-one. In fact, given $\Sigma\posdef 0$, we can show that for each order of the variables, there is a DAG $(\dagnopi,\varnopi)$ such that $\trueCov(\dagnopi,\varnopi)=\Sigma$. This is most easily illustrated with an example:

\begin{example}
\label{ex:mainexample}
Consider the following covariance matrix:
\begin{align}
\label{eq:ex:covmat}
\trueCov = \begin{pmatrix}
6 & 4 & -6 & -30 \\
4 & 4 & -4 & -20 \\
-6 & -4 & 7 & 39 \\
-30 & -20 & 39 & 234
\end{pmatrix}.
\end{align}

\noindent
Given $\trueCov$, we can construct matrices $(\dagnopi, \varnopi)$ that satisfy \eqref{eq:intro:streqnmodel} by the following procedure: First, project $\dagvec_{1}$ onto $\dagvec_{2}$, then project $\dagvec_{3}$ onto $\dagvec_{2}$ and $\dagvec_{1}$, and finally project $\dagvec_{4}$ onto $\dagvec_{2}$, $\dagvec_{1}$, and $\dagvec_{3}$. This induces an ordering $\prec$ on the variables given by $\dagvec_{4}\prec\dagvec_{3}\prec\dagvec_{1}\prec\dagvec_{2}$, wherein each $\dagvec_j$ is projected onto the nodes after it under $\prec$. This ordering induces a permutation $\perm_1$ defined by 
$\perm_{1}(1) = 4$, 
$\perm_{1}(2) = 3$, 
$\perm_{1}(3) = 1$, 
$\perm_{1}(4) = 2$.
The coefficients of each linear projection and the resulting residual variances lead to a set of structural equations as in \eqref{eq:intro:streqnmodel}. Denoting the parameters obtained in this way by $\dagnopi=\tB{\perm_{1}}$ and $\varnopi=\tOm{\perm_{1}}$, we obtain the following parameters for \eqref{eq:ex:covmat}:
\begin{align}
\label{eq:param:pi1}
\tB{\perm_{1}}
&= \begin{pmatrix}
0 & 0 & -1 & 4 \\
1 & 0 & 0 & 0 \\
0 & 0 & 0 & 9 \\
0 & 0 & 0 & 0
\end{pmatrix}, 
\quad
\tOm{\perm_{1}}
= \begin{pmatrix}
2 & 0 & 0 & 0 \\
0 & 4 & 0 & 0 \\
0 & 0 & 1 & 0 \\
0 & 0 & 0 & 3
\end{pmatrix}.
\end{align}

\noindent
This DAG is depicted on the left in Figure~\ref{fig:equivdag}.

\begin{figure}[t]

\centering
\begin{tikzpicture}[scale=1]
    \tikzset{
    mainnode/.style = {draw=black, circle, fill=white, text=black},
    transform shape, thick,
    every node/.style = {minimum size = 5mm},
    level 1/.style = {sibling distance=1.5cm, scale=0.9},
    level distance = 2cm
}

  \node[mainnode](X1) at (1,4) {$X_{1}$};
  \node[mainnode](X2) at (4,4) {$X_{2}$};
  \node[mainnode](X3) at (4,1) {$X_{3}$};
  \node[mainnode](X4) at (1,1) {$X_{4}$};
  \draw[->] (X1) -- (X4)  node[midway, fill = white, inner sep=5pt, minimum size=5pt]{4};
  \draw[->, very thick, black] (X1) -- (X3)  node[midway, fill = white, inner sep=5pt, minimum size=5pt]{\textbf{-1}};
  \draw[->, very thick, black] (X2) -- (X1)  node[midway, very thick, fill = white, inner sep=5pt, minimum size=5pt]{\textbf{1}};
  \draw[->] (X3) -- (X4)  node[midway, fill = white, inner sep=5pt, minimum size=5pt]{9};
\end{tikzpicture}%
\hspace{5em}%
\begin{tikzpicture}[scale=1]
    \tikzset{
    mainnode/.style = {draw=black, circle, fill=white, text=black},
    transform shape, thick,
    every node/.style = {minimum size = 5mm},
    level 1/.style = {sibling distance=1.5cm, scale=0.9},
    level distance = 2cm
}

  \node[mainnode](X1) at (1,4) {$X_{1}$};
  \node[mainnode](X2) at (4,4) {$X_{2}$};
  \node[mainnode](X3) at (4,1) {$X_{3}$};
  \node[mainnode](X4) at (1,1) {$X_{4}$};
  \draw[->] (X1) -- (X4)  node[midway, fill = white, inner sep=5pt, minimum size=5pt]{4};
  \draw[->, very thick, black] (X2) -- (X1)  node[midway, fill = white, inner sep=5pt, minimum size=5pt]{$\mathbf{\frac13}$};
  \draw[->, very thick, black] (X3) -- (X1)  node[midway, fill = white, inner sep=5pt, minimum size=5pt]{$\mathbf{-\frac23}$};
  \draw[->, very thick, black] (X3) -- (X2)  node[midway, fill = white, inner sep=5pt, minimum size=5pt]{$\mathbf{-\frac47}$};
  \draw[->] (X3) -- (X4)  node[midway, fill = white, inner sep=5pt, minimum size=5pt]{9};
\end{tikzpicture}%
\caption{Two equivalent DAGs with different edge weights, as indicated by the edges highlighted in bold. \label{fig:equivdag}}
\end{figure}
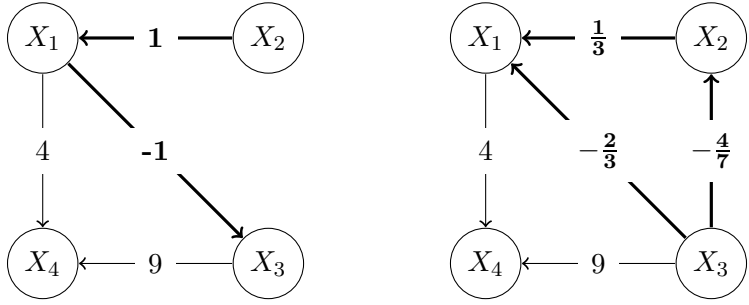

If we choose a different ordering, say $\dagvec_{4}\prec\dagvec_{1}\prec\dagvec_{2}\prec\dagvec_{3}$, we obtain a different set of parameters:
\begin{align}
\label{eq:param:pi2}
\tB{\perm_{2}}
&= \begin{pmatrix}
0 & 0 & 0 & 4 \\
\frac13 & 0 & 0 & 0 \\
-\frac23 & -\frac47 & 0 & 9 \\
0 & 0 & 0 & 0
\end{pmatrix}, 
\quad
\tOm{\perm_{2}}
= \begin{pmatrix}
\frac23 & 0 & 0 & 0 \\
0 & \frac{12}{7} & 0 & 0 \\
0 & 0 & 7 & 0 \\
0 & 0 & 0 & 3
\end{pmatrix},
\end{align}

\noindent
where $\pi_2$ is the permutation induced by this ordering. This DAG is depicted on the right in Figure~\ref{fig:equivdag}.  The fact that the incoming edges for $\dagvec_{4}$ are the same in both cases is consistent with the observation that under both orderings, we are projecting $\dagvec_{4}$ onto $\{\dagvec_{1},\dagvec_{2},\dagvec_{3}\}$.
\end{example}

Both \eqref{eq:param:pi1} and \eqref{eq:param:pi2} lead to a linear model as in \eqref{eq:intro:streqnmodel} with $\cov(\dagvec)=\trueCov$, but with distinct numbers of edges. Two other key points from this example are worth emphasizing:
\begin{itemize}
\item Both DAGs $\tB{\perm_{j}}$ ($j=1,2$) are sparse, as measured by the number of nonzero entries in $\tB{\perm_{j}}$.
\item The conditional variances given by $\tOm{\perm_{j}}$ ($j=1,2$) are different in general. Moreover, in this example there is no~\emph{equivariance} DAG; that is, one whose conditional variance matrix $\varnopi$ is a constant multiple of the identity (Section~\ref{subsec:app:causal}).
\end{itemize}

\noindent
These observations motivate many of the developments in the sequel and lead us to make the following definition:

\begin{defn}
\label{defn:eqclass}
Let $\diagmatrices$ denote the space of $p\times p$ diagonal matrices with positive entries on the diagonal. Given a covariance matrix $\trueCov \posdef 0$, define the \emph{equivalence class} of $\trueCov$ to be
\begin{align*}
\eqclass(\trueCov)
&= \big\{\dagnopi\in\dagspace: \text{$\trueCov=\trueCov(\dagnopi,\varnopi)$ for some $\varnopi\in\diagmatrices$}\big\},
\end{align*}
where $\dagspace\subset\R^{p\times p}$ is the space of DAGs and $\trueCov(\dagnopi,\varnopi)$ is defined by \eqref{eq:def:sigmafcn}. The DAGs in $\eqclass(\trueCov)$ will be called \emph{equivalent}.
\end{defn}

\noindent
Equivalently, $\dagnopi$ is the set of DAGs that satisfy \eqref{eq:intro:streqnmodel} for some $\varnopi\in\diagmatrices$. For the reader familiar with graphical models, note that $\eqclass(\trueCov)$ is \emph{not} the same as the Markov equivalence class, and in fact is quite different.

We will now show that $\eqclass(\trueCov)$ can be constructed explicitly. Let $\allperms$ denote the class of permutations on $p$ elements and $\perm\in\allperms$ be a fixed permutation. Write $\trueInv:=\trueCov^{-1}$ and use the (modified) Cholesky decomposition to write $\permmat_\perm\trueInv$ uniquely as $\permmat_{\perm}\trueInv = (I-L)D^{-1}(I-L)^T$ where $L$ is strictly lower triangular and $D\in\diagmatrices$. Define
\begin{align}
\label{eq:def:tB}
\tB{\perm} 
:= \permmat_{\perm^{-1}}L, \quad 
\tOm{\perm}
:= \permmat_{\perm^{-1}}D.
\end{align}

\noindent
Since an adjacency matrix represents a DAG if and only if it is permutation-similar to a strictly lower triangular matrix, $\tB{\perm}$ is indeed a DAG. We refer to $\tB{\perm}$ as the DAG associated with $\pi$.

\begin{lemma}
\label{lem:eqclass}
For any $\trueCov\posdef 0$, the equivalence class of $\trueCov$ is given by $\eqclass(\trueCov) = \{\tB{\perm} : \;\perm\in\allperms\}$.
\end{lemma}

\noindent
This characterization of the equivalence class confirms that $\eqclass(\trueCov)$ depends only on $\trueCov$ and not on the sample $\sampledagmat$, and will play an important role in the rest of the paper. Therefore, without further qualification, we shall always write an arbitrary element of $\eqclass(\trueCov)$ as $\tB{\perm}$. The columns of $\tB{\perm}$ will be denoted by $\dagcol_j(\perm)$, and the $j$th diagonal element of $\tOm{\perm}$ will be denoted by $\dagerrvar_j^2(\perm)$. It follows from these definitions and \eqref{eq:intro:streqnmodel} that 
\begin{align}
\label{eq:localsem}
\dagvec_{j}
&= \dagcol_j(\perm)^T\dagvec + \dagerr_j(\perm),
 \;\; \text{where\;\; $\dagerr_j(\perm)\sim\normalN(0,\dagerrvar_j^2(\perm))$,}
\end{align}
for $j=1,\ldots,p$. As a consequence, the joint distribution factors according to the local conditional distributions implied by \eqref{eq:localsem}, which makes $\tB{\perm}$ a \emph{Bayesian network} for each $\perm\in\allperms$. Furthermore, each $\tB{\perm}$ is a so-called minimal I-map for $\normalN_{p}(0,\trueCov)$ \citep{koller2009}.

\subsection{Global and restricted minimizers}
\label{subsec:prelim:minimizers}

In presenting our results, we will take advantage of the permutation decomposition of the equivalence class $\eqclass(\trueCov)$, according to Lemma~\ref{lem:eqclass}. In this section, we formalize these ideas further and provide some intuition behind the estimator \eqref{eq:intro:mainestimator}. In particular, these ideas are purely theoretical and are not intended to provide practical guidance for computing $\dagadjest$ (for more on computation, see Section~\ref{subsec:conclusion:comp}).

Recall that $\dagspace$ is the space of $p\times p$ real matrices that represent DAGs when interpreted as weighted adjacency matrices. For each permutation $\perm\in\allperms$, define a subset of $\dagspace$ by
\begin{align*}
\subdagspace{\perm}
&= \{\gendag\in\dagspace: \text{$\permmat_{\perm} \gendag$ is lower triangular}\}.
\end{align*}

\noindent
A DAG $B = [\,\beta_{1}\|\cdots\|\beta_{p}\,]\in\dagspace$ is in $\subdagspace{\perm}$ if and only if $\supp(\beta_{j})\subset \dagcand_{j}(\perm)$ for all $j=1,\ldots,p$, where
\begin{align}
\label{eq:def:dagcand}
\dagcand_{j}(\perm)
:= \{k : \perm^{-1}(k) > \perm^{-1}(j)\}
\end{align}
consists of the nodes $X_k$ that come after $X_j$ under the ordering $X_{\perm(i)}\prec X_{\perm(i+1)}$ for $i=1,\ldots,p-1$. In other words, for each node $X_{j}$, the permutation $\perm$ defines a unique set of candidate parents given by \eqref{eq:def:dagcand}, and $B\in\subdagspace{\perm}$ if and only if the parent set of $\beta_j$ comes from $\dagcand_{j}(\perm)$ for all $j$. By definition, $\tB{\perm}\in\subdagspace{\perm}$ for every $\perm$ and hence $\supp(\dagcol_j(\perm))\subset\dagcand_{j}(\perm)$ for all $j$.

Recall the estimator $\dagadjest$ defined via~\eqref{eq:intro:mainestimator} and~\eqref{eq:def:plsobj}. Using the decomposition of the feasible region in~\eqref{eq:intro:mainestimator} as $\dagspace = \bigcup_\perm\subdagspace{\perm}$, we can build some intuition about $\dagadjest$: Suppose we had oracle knowledge of some \emph{optimal} permutation $\goodperm$ in advance. For example, $\goodperm$ might be a permutation that minimizes the number of edges among all the equivalent DAGs. Then we can ignore any DAG that is not consistent with $\goodperm$, and hence restrict our search space to $\subdagspace{\goodperm}$. Thus, estimating the \emph{oracle DAG} $\tB{\goodperm}$ reduces to minimizing $\plsobj(\gendag)$ with the additional constraint that $B\in\subdagspace{\goodperm}$, which is exactly equivalent to estimating an autoregressive model. Of course, in practice we do not know $\goodperm$, so in addition to parameter estimation, an additional challenge that arises in DAG learning is adaptively finding an optimal permutation $\goodperm$, which is a challenging combinatorial problem.

These observations motivate the following definition:

\begin{defn}
\label{defn:restrictedmin}
For $\perm\in\allperms$, a \emph{restricted minimizer} of $\plsobj$ is
\begin{align}
\label{eq:def:restrictedmin}
\dagadjest(\perm)
&\;\in\; \argmin_{\gendag\, \in \,\subdagspace{\perm}} \;\plsobj(\gendag).
\end{align}

\noindent
The columns of $\dagadjest(\perm)$ will be denoted by $\dagcolest_{j}(\perm)$, $j=1,\ldots,p$.
\end{defn}

\noindent
Our strategy will be to show that $\dagadjest(\perm)\approx\tB{\perm}$ \emph{for all $\perm$} simultaneously. In Section~\ref{subsec:app:scorebased}, we will exploit this decomposition into restricted minimizers to establish guarantees on the global minimizer $\dagadjest$.

\section{Neighbourhood regression}
\label{sec:nhbd}

The core of our analysis is the regression decomposition \eqref{eq:localsem} which we will interpret as a neighbourhood regression problem.
In the literature on undirected graphical models, ``the'' neighbourhood of a fixed node $\dagvec_{j}$ is implicitly understood to be the set of all other variables, namely the collection $\dagvec_{-j}$ \citep{meinshausen2006}. For DAGs, we need to generalize the notion of  neighbourhood to an arbitrary subset $S\subset[p]_{j}$.
At a high-level, the \emph{neighbourhood regression problem} for the node $\dagvec_{j}$ and a subset $S\subset[p]_{j}$ is the problem of estimating the linear projection of $\dagvec_{j}$ onto $\dagvec_{S}$. These problems arise naturally from the iterative projection procedure described in Section~\ref{subsec:prelim:perms} and made explicit by \eqref{eq:localsem}.  Doing this for all possible permutations is tantamount to projecting $\dagvec_{j}$ onto all possible subsets $S\subset[p]_{j}$. In this section we formalize these notions and introduce the concept of a \emph{model selection exponent}, which quantifies the difficulty of a neighbourhood regression problem.

\subsection{Penalized least-squares estimators}
\label{subsec:nhbd:regression}

The following defines the popu-lation-level quantity that is our primary interest:

\begin{defn}
\label{defn:nhbd}
For any $S \subset [p]_j$, let 
\begin{align*}
\nhbdcoef_j(S) &:= \argmin_{\beta \,\in\, \R^{p}, \;\,\supp(\beta)\, \subset\, S}
\E \big[\dagvec_{j} - \beta^{T}\dagvec\big]^2.
\end{align*}

\noindent
We call $\nhbdcoef_j(S)$ the SEM coefficients for variable $j$ regressed on the variables $S$. 
\end{defn}

\noindent
The SEM coefficients $\nhbdcoef_j(S)$ are \emph{population} level quantities that depend on $\trueCov$, but not on the sample $\sampledagmat$. Furthermore, it is easy to verify that $\dagcol_{j}(\perm) = \nhbdcoef_j(\dagcand_{j}(\perm))$ from the definition of $\dagcol_{j}(\perm)$ as the $j$th column of $\tB{\perm}$.

Now we turn to the problem of estimating $\nhbdcoef_j(S)$ via PLS. For now, $\pl$ is assumed to be a fixed, possibly nonconvex, regularizer. We start with some abstraction in order to emphasize the key underlying assumptions.

\begin{defn}
\label{defn:abstractregression}
Suppose $\fixedy\in\R^{n}$ and $\fixeddesignmat \in \R^{n \times m}$. Let $S \subset [m]$ and consider the set defined by
\begin{align}
\label{eq:supp:restr:estim}
\abstractpls_\lambda(\fixedy, \fixeddesignmat; S) 
:=  \argmin_{\theta \in \R^m,\; \supp(\theta) \,\subset\, S} \;\frac{1}{2n} \norm{\fixedy -\fixeddesignmat \theta}_2^2 + \pl(\theta),
\end{align}

\noindent
i.e., the set of global minimizers of the support-restricted PLS problem above. Let $\abstractpls_\lambda(\fixedy,\fixeddesignmat) := \abstractpls_\lambda(\fixedy,\fixeddesignmat;[m])$ correspond to the case where there is no support restriction.
\end{defn}

In this abstract definition, $\fixedy$ is considered a fixed quantity and may have no relation to $\fixeddesignmat$. Of course, in practice we are interested the case where $\fixedy$ and $\fixeddesignmat$ are linked via a linear model: If $\sampley = \fixeddesignmat\lincoef + \samplelinerr$, where $\sampley\in\R^{n}$, $\fixeddesignmat \in \R^{n \times m}$, $\lincoef \in \R^m$ and $\samplelinerr \sim \normalN_{n}(0,\sigma^2 I_n)$, then $\abstractpls_\lambda(\sampley, \fixeddesignmat)$ is the collection of PLS estimators for $\lincoef$ in classical linear regression. The support-restricted version $\abstractpls_\lambda(\sampley, \fixeddesignmat; S)$ then allows us to properly define a neighbourhood regression problem. 
Let $\sampledagmatelem_{j}$ denote the $j$th column of $\sampledagmat$.

\begin{defn}[Neighbourhood regression]
\label{defn:nhbdregression}
The \emph{neighbourhood regression problem} for node $\dagvec_{j}$ given a neighbourhood $S\subset[p]_j$ is defined to be the (possibly nonconvex) program solved by $\abstractpls_\lambda(\sampledagmatelem_j,\sampledagmat; S)$. An arbitrary solution to this program will be denoted by $\hnhbdcoef_j(S)$, i.e. $\hnhbdcoef_j(S)\in\abstractpls_\lambda(\sampledagmatelem_j,\sampledagmat; S)$.
\end{defn}

\noindent
For any $\perm$ such that $\dagcand_j(\perm)=\dagcand$, a solution $\hnhbdcoef_j(\dagcand)$ estimates $\dagcol_{j}(\perm) = \nhbdcoef_j(\dagcand)$. 
Thus neighbourhood regression problems are the basic units in learning the parent set of a node and hence the DAG structure.
This concept allows us to do two things:
\begin{enumerate}
\item A key step in our proof is to show that the analysis of $\dagadjest$ can be  reduced to the study of a (very large) collection of neighbourhood regression problems;
\item We will associate to each neighbourhood regression problem an ``exponent'', which is a fundamental quantity measuring how difficult the corresponding model selection problem is for a given regularizer. These exponents will then be used to write down explicit, non-asymptotic upper bounds on the model selection failure of $\dagadjest$ in terms of the model selection failure of each neighbourhood regression. 
\end{enumerate}

\noindent
The details of these reductions can be found in Section~\ref{sec:proofs}.

\subsection{Model selection exponents}
\label{subsec:nhbd:msexp}

Given some $n\times m$ matrix $\fixeddesignmat$ and $m$-vector $\lincoef$, define a set of ``bad'' noise vectors as follows:
\begin{align}
\label{eq:def:mserrset}
\mserrset(\fixeddesignmat,\lincoef;S)
:= \Big\{
\fixedlinerr\in\R^{n} :
\supp(\hlincoef) &\neq \supp(\lincoef) \\
\nonumber
&\exists \;\hlincoef \in  \abstractpls_\lambda(\fixeddesignmat\lincoef+\fixedlinerr,\fixeddesignmat;S) \Big\}.
\end{align}

\noindent
For a random vector $\samplelinerr\in\R^n$ (e.g. $\samplelinerr\sim\normalN_{n}(0,\sigma^{2}I_{n})$), we then have the following model selection failure event:
\begin{align}
\label{eq:def:msevent}
\msevent(\samplelinerr,\fixeddesignmat,\lincoef;S)
:= \Big\{
\samplelinerr\in\mserrset(\fixeddesignmat,\lincoef;S)
\Big\}.
\end{align}

\noindent
As usual we use the shorthand $\msevent(\samplelinerr,\fixeddesignmat,\lincoef) = \msevent(\samplelinerr,\fixeddesignmat,\lincoef;[m])$.

\begin{defn}
\label{defn:msexp}
Given a regularizer $\pl$, the \emph{model selection exponent} for the regression problem $\sampley = \fixeddesignmat\lincoef + \samplelinerr$ is defined to be
\begin{align*}
\msexp_\lambda(\fixeddesignmat,\lincoef,\sigma^2) 
:= -\log \pr \big[\msevent(\samplelinerr,\fixeddesignmat,\lincoef)\big],
\end{align*}
where $\pr$ is taken with respect to the distribution of $\samplelinerr\sim\normalN_{n}(0,\sigma^2 I_n)$.
\end{defn}

A larger exponent corresponds to better model selection performance. To characterize the exponent further, let us introduce two quantities that are familiar from the regression literature: The sparsity level and the signal strength. The difference now is that instead of quantifying these for a single, ``true'' parameter, we need to consider all $p\cdot2^{p-1}$ SEM coefficients. 
Recall that we denote the $j$th column of $\tB{\perm}$ by $\dagcol_{j}(\perm)$. For any $\genu\in\R^{p}$, let $\betamin(\genu):=\min\{|\genu_{j}| : \genu_{j}\ne 0\}$. 

\begin{defn}
    \label{defn:eqclassparam}
    For any $\trueCov,$ let
    \begin{align}
    \label{eq:def:maxdeg:betamin}
    \maxdeg(\trueCov)
    :=\sup_{\perm,\, j}\; \norm{\dagcol_{j}(\perm)}_0, 
    \quad 
    \betamin(\trueCov) := \inf_{\perm, \,j }\;\betamin(\dagcol_{j}(\perm)).
    \end{align}
\end{defn}

\noindent
In Definition~\ref{defn:eqclassparam}, we could have equivalently used the SEM coefficients directly: $d(\Sigma) = \sup\{ \norm{\beta_j(S)}_0:\; S \subset [p]_j,\; j \in [p]\}$ and similarly for $\betamin(\trueCov)$. 

Let $\maxvar^{2}  := \max_{1\le j\le p}\var(\dagvec_j)$ and note that $\maxvar^{2}\le \maxeval(\trueCov)$. Each of the quantities $\maxdeg$, $\betamin$, and $\maxvar^{2}$ are allowed to depend on $n$, and in particular $d$ may diverge and $\betamin$ may vanish as $n\to\infty$.
Now, let
\begin{align}
\label{eq:def:unifmsexp}
\unifmsexp_{\lambda} = 
\unifmsexp_{\lambda}(\sampledagmat,\trueCov)
:= \inf_{0<\sigma\le\maxvar}\inf_{\substack{\norm{\theta}_0 \;\le\; d(\trueCov) \\ \betamin(\theta)\;\ge\;\betamin(\trueCov)}} \msexp_{\lambda}(\sampledagmat,\, \theta,\, \sigma^2),
\end{align}

\noindent
which  encodes what is usually proved in the regression literature: An upper bound on the probability of model selection failure given the maximum sparsity level $\maxdeg(\trueCov)$, minimum signal strength $\betamin(\trueCov)$, and the maximum variance $\maxvar^2$. This probability generally depends on $\lambda$, the regularization parameter, which in turn may depend on any of these quantities. We will provide an example of such a bound under the MCP in Section~\ref{subsec:results:supp}.

\section{Uniform control}
\label{sec:results}

Our main results are divided into two separate theorems: (i) a bound on the probability of false selection, and (ii) an upper bound on the estimation error. The key idea is to show that uniformly for all permutations $\perm\in\allperms$, we have $\supp(\dagadjest(\perm)) = \supp(\tB{\perm})$ and control over the deviations $\norm{\dagadjest(\perm)-\tB{\perm}}_{r}$ for $r=1,2$.


\subsection{Assumptions}
\label{subsec:prelim:assumptions}

Let us collect the assumptions needed for our main results. The following is our only assumption on $\trueCov$, and simply ensures that $\eqclass(\trueCov)$ is well-defined (cf.~Definition~\ref{defn:eqclass}):

\begin{condition}
    \label{condn:evals}
    $\trueSigma$ is positive definite, i.e., $\mineval(\trueSigma)>0$.
\end{condition}

In addition, we require some regularity conditions on the regularizer $\pl$.
We consider \emph{coordinate-separable} regularizers, that is with some abuse of notation we have $\pl(\gendag) = \sum_{i,j} \pl(|\beta_{ij}|)$, where $\pl:[0,\infty)\to[0,\infty)$ is a univariate regularizer. A key feature of our analysis is to provide insights into how $\pl$ affects the solution to \eqref{eq:intro:mainestimator}, so we will not specify a particular regularizer in advance. Instead, we make the following minimal assumptions on $\pl$:

\begin{condition}
    \label{condn:pl:basic}
     $\pl$ is concave, nondecreasing, right-differentiable at zero, and satisfies the following conditions:
     \begin{itemize}
     \item[(a)] $\pl(0)=0$;
     \item[(b)] $0<\penderiv<\infty$;
     \item[(c)] There are constants $\lowflat,\lowlin\ge0$, independent of $\lambda$, such that $\pl(x)\ge \min\{\lowlin\lambda x, \lowflat\lambda^{2}\}$.
     \end{itemize}
\end{condition}

\noindent
We are also interested in penalties that ``approximate'' the $\ell_{0}$ penalty: 
\begin{defn}
    \label{defn:L0compat}
    $\pl$ is called $\ell_{0}$-compatible if there is a constant $\uppflat\ge0$, independent of $\lambda$, such that $\pl(x)\le\uppflat\lambda^2$ for all $x\ge0$.
\end{defn}

\noindent
An elementary consequence of Condition~\ref{condn:pl:basic} is that $\pl$ is subadditive. Condition~\ref{condn:pl:basic}(c) says that $\pl$ can be bounded below by a capped-$\ell_{1}$ penalty: It is always true that a concave, nondecreasing function can be bounded below by a capped-$\ell_{1}$ penalty, and Condition~\ref{condn:pl:basic}(c) simply normalizes this capped-$\ell_{1}$ penalty in terms of $\lambda$. These conditions allow for most concave penalties, including the MCP and SCAD, along with the familiar $\ell_{1}$ penalty. When $\pl$ is $\ell_{0}$-compatible, we have $\pl(B)\le \uppflat\lambda^{2}\norm{B}_{0}$ for any $B\in\dagspace$. 

The following is a list of some admissible regularizers under our conditions:
\begin{enumerate}[--,leftmargin=0cm,itemindent=.5cm, itemsep=.5em]
    \item The minimax concave penalty (MCP) proposed by~\citet{zhang2010}:
    \begin{align}
    \label{eq:mcp}
    \pl(x;\mcpparam) &:= 
    \lambda\Big(x-\frac{x^2}{2\lambda\mcpparam}\Big)1(x<\lambda\mcpparam)
    + \frac{\lambda^2\mcpparam}{2}1(x\ge \lambda\mcpparam).
    \end{align}
    This satisfies Condition~\ref{condn:pl:basic} with $\penderiv=\lambda$, $\lowlin = 1/2$, and $\lowflat=\mcpparam/2$. Furthermore, it is $\ell_{0}$-compatible with $\uppflat=\mcpparam/2$. 
    
    \item \label{ex:ell1}
    The $\ell_{1}$ penalty, $\pl(x)=\lambda x$, also satisfies Condition~\ref{condn:pl:basic} with $\penderiv=\lambda$, $\lowlin = 1$, and $\lowflat\in[0,\infty)$. The $\ell_1$ penalty is not, however, $\ell_{0}$-compatible.
    
    \item \label{ex:ell0}
    The $\ell_{0}$ penalty, $\pl(x)=(\lambda^2/2)1(x\ne0)$, fails to satisfy Condition~\ref{condn:pl:basic}(b) since it is not right-differentiable at zero (although see Remark~\ref{rem:l0pen} and Section~\ref{subsec:conclusion:ell0}). Nonetheless, it satisfies all of the other assumptions, including Condition~\ref{condn:pl:basic}(c) with $\lowlin\in[0,\infty)$ and $\lowflat=1/2$.
    Of course, the $\ell_{0}$ penalty is trivially $\ell_{0}$-compatible with $\uppflat=1/2$. 
\end{enumerate}

\begin{remark}
    \label{rem:l0pen}
    \sloppypar{Even though the $\ell_{0}$ penalty does not satisfy Condition~\ref{condn:pl:basic}(b), all the results in this paper still apply to the $\ell_{0}$ penalty under possibly different values of the constants involved. This requires a small modification to the overall argument, which is discussed in Section~\ref{subsec:conclusion:ell0}.}
\end{remark}

\subsection{Support recovery}
\label{subsec:results:supp}

Our first result provides a guarantee for model selection consistency:

\begin{thm}
\label{thm:main:supp}
Suppose $\sampledagmat\iid\normalN_{p}(0,\trueCov)$ and assume Condition~\ref{condn:evals}. Then
\begin{align*}
\pr \Bigg(  \bigcap_{j=1}^{p}\bigcap_{\perm\in\allperms}\big\{\supp(\dagcolest_j(\perm)) &= \supp(\dagcol_j(\perm))\big\} \Bigg) \\
&= \pr \Bigg( \bigcap_{\perm\in\allperms}\big\{\supp(\dagadjest(\perm)) = \supp(\tB{\perm})\big\} \Bigg) \\
&\ge 1 - p\binom{p}{d}\E e^{-\unifmsexp_{\lambda}(\sampledagmat,\trueCov)}.
\end{align*}
\end{thm}

\noindent
Theorem~\ref{thm:main:supp} explicitly relates the behaviour of the neighbourhood regression problems and the restricted minimizers to the model selection exponent for a random design regression problem. Since $p\binom{p}{d} \le p^{d+1}$, it follows that
\begin{align}
\label{eq:mod:sel:suff:cond}
    \unifmsexp_{\lambda}(\sampledagmat,\trueCov) \ge C(d+1)\log p, \quad \text{for some $C>1$}
\end{align}

\noindent
is sufficient for simultaneous, \emph{uniform} model selection consistency for all $\dagadjest(\perm)$ and $\dagcolest_{j}(\perm)$. For example, a bound of the form~\eqref{eq:mod:sel:suff:cond} holds with high probability for the MCP 
as illustrated in Lemma~\ref{lem:mcp:example}.

Theorem~\ref{thm:main:supp} provides uniform control on the probability of false selection for all permutations $\perm\in\allperms$. Since there are $p!$ permutations, this is a fairly surprising result: A na\"ive union bound would have $p!$ in place of the $p\binom{p}{d}$ term. Even though existing bounds on false selection for linear regression decay exponentially as in \eqref{eq:mod:sel:suff:cond}, they are not strong enough to counteract $p!$. Theorem~\ref{thm:main:supp} combined with~\eqref{eq:mod:sel:suff:cond}, on the other hand, leads to a similar exponential decay in controlling this \emph{super}exponential collection of regression problems. The key is a novel monotonicity argument detailed in Section~\ref{subsec:proofs:supp}, more specifically Section~\ref{subsubsec:proofs:monotonicity}.

To illustrate, let us give explicit conditions for~\eqref{eq:mod:sel:suff:cond} to hold in the case of MCP, defined in~\eqref{eq:mcp}. 
\citet{huang2012} consider PLS estimators $\abstractpls_\lambda(\fixedy, \fixeddesignmat; S)$ as defined in \eqref{eq:supp:restr:estim}, applied to the data from a linear regression model $y  =  Z \lincoef +  \samplelinerr$, and provides conditions for model selection consistency. Adapting their result to our setup and notation, we have the following bound on model selection exponent for the MCP:
\begin{lemma}
\label{lem:mcp:example}
Suppose $\sampledagmat\iid\normalN_{p}(0,\trueCov)$. Take $\pl=\pl(\,\cdot\,;\mcpparam)$ as in \eqref{eq:mcp} and assume $\trueCov$ is positive definite with bounded eigenvalues. Assume that
\begin{enumerate}\itemsep=1ex
    \item $\maxdeg(\trueCov) \;\le\; \covconst_4\cdot\min\{p,n,n/\log p\}$,
    \item $\betamin(\trueCov) \;>\; (1+\mcpparam)\lambda$ for some $\mcpparam>\covconst_5>0$.
\end{enumerate}

\noindent
Then for any $\lambda \;\ge\; \covconst_6\cdot\sqrt{(d+1)\log p/n}$, it follows that $\E e^{-\unifmsexp_{\lambda}(\sampledagmat,\trueCov)} \le 3\exp{(-2\min\{d\log p, n\})}$. Here, $\covconst_j=\covconst_j(\trueCov)\,\,(j=4,5,6)$ are constants depending only on $\{\mineval(\trueCov), \maxeval(\trueCov)\}$.
\end{lemma}

\noindent
Taking $\lambda$ as in Lemma~\ref{lem:mcp:example}, it follows from Theorem~\ref{thm:main:supp} that if $n\ge C(d+1)\log p$ for a sufficiently large constant $C>0$,
\begin{align*}
\pr \Bigg(  \bigcap_{\perm\in\allperms}\big\{\supp(\dagadjest(\perm)) = \supp(\tB{\perm})\big\}\Bigg)
\to 1,
\end{align*}

\noindent
i.e., each restricted minimizer $\dagadjest(\perm)$ is model selection consistent.

\begin{remark}
    \label{rem:L0supp}
    \citet{geer2013} remark that a thresholded $\ell_{0}$-penalized maximum likelihood estimator for $(\gendag,\genvar)$ also attains exact support recovery with high probability if $\lambda$ is of order $s_0\log p/n$, where $s_0:=\argmin_{\perm}\norm{\tB{\perm}}_{0}$. Requiring that $\lambda\to 0$, this result allows for $p$ to grow with $n$, but is not truly high-dimensional: In general, unless the graph is extremely sparse (i.e. less than one parent per node), $s_0>p$ in which case this scaling requires $p\ll n$. By contrast, our results only require $\lambda$ on the order $d\log p/n$, which guarantees consistency when $p\gg n$ and allows $s_0$ as large as $O(pd)$ (i.e $d$ parents per node). 
\end{remark}

\subsection{Deviation bounds}
\label{subsec:results:dev}

In addition to support recovery, we can also bound the estimation errors for all columns $\norm{\dagcolest_j(\perm)-\dagcol_j(\perm)}_{r}$ simultaneously:

\begin{thm}
\label{thm:main:dev}
Under Conditions~\ref{condn:evals} and~\ref{condn:pl:basic}, there are positive constants $\covconst_1=\covconst_1(\trueCov)$ and $\covconst_2=\covconst_2(\trueCov)$ and universal constants $c_{3},c_{4}>0$ such that if
\begin{align*}
n
> \covconst_1 \, d\log p, 
\quad
\lambda
\ge \covconst_2\,  \sqrt{\frac{(d+1)\log p}{n}}
\end{align*}

\noindent
then for $r=1,2$, the following uniform deviation bounds hold:
\begin{align}
\label{eq:thm:main:dev:Lr}
&\pr\Bigg(
\bigcap_{\perm\in\allperms}
\Big\{\norm{\dagadjest(\perm)-\tB{\perm}}_{r}
\le c_{3} \frac{\penderiv}{\mineval(\trueCov)}\norm{\tB{\perm}}_{0}^{1/r}
\Big\}
\Bigg) \\
\label{eq:thm:main:dev:col}
&\ge\pr\Bigg(
\bigcap_{j=1}^{p}
\bigcap_{\perm\in\allperms}
\Big\{\norm{\dagcolest_j(\perm)-\dagcol_j(\perm)}_{r}
\le c_{4} \frac{\penderiv}{\mineval(\trueCov)}\norm{\dagcol_j(\perm)}_{0}^{1/r}\Big\}
\Bigg) \\
\nonumber
&\ge 1 - \probbound.
\end{align}
\end{thm}

\noindent
Thus, in addition to support recovery guarantees, we can bound the finite sample error in estimating the parameters $\tB{\perm}$, uniformly over the entire equivalence class $\eqclass(\trueCov)$. Taking the MCP~\eqref{eq:mcp} as an example for which $\penderiv=\lambda$, we have an $\ell_2$ error bound of the order $\lambda \norm{\dagcol_j(\perm)}_{0}^{1/2}$, uniformly for all $\perm$ and $j$. Using standard interpolation bounds for $\ell_{r}$ norms, this result can easily be extended to bounds for all $1\le r\le 2$. See, for example, the proof of Theorem~7.1 in \cite{bickel2009}.

\subsection{Discussion}
\label{subsec:results:discussion}

Let us pause to discuss Theorems~\ref{thm:main:supp} and~\ref{thm:main:dev}. In order to obtain model selection consistency, it is sufficient to show that \eqref{eq:mod:sel:suff:cond} holds: Of course, whether or not this holds for a particular regularizer depends on the regularity of the problem. For example, if the $\ell_{1}$ penalty is used, then it is well-known that the irrepresentable condition is required for a bound like \eqref{eq:mod:sel:suff:cond} to hold \citep{zhao2006}. By contrast, if nonconvex regularizers such as the MCP or SCAD are used, then this type of condition is not necessary \citep{loh2015,huang2012,zhang2010}. 
Condition~2 on the signal strength in Lemma~\ref{lem:mcp:example}, also called a \emph{beta-min} condition in the literature, is necessary in order to guarantee support recovery but is not necessary for the deviation bounds alone (Theorem~\ref{thm:main:dev}). For more on this assumption, see the related discussion in Section~\ref{subsubsec:app:sparsity}.

Theorem~\ref{thm:main:dev} (along with Lemma~\ref{lem:mcp:example}) requires the regularization parameter to scale as $\lambda\sim\sqrt{d\log p/n}$, which is the familiar scaling from the literature up to a factor of $\sqrt{d}$. As our bounds are nonasymptotic, $d$ may depend on $n$: For example, we may take $\lambda\to 0$ and $d\to\infty$ as long as $d=o(n/\log p)$. See also Remark~\ref{rem:betamin}. What's more, $d$ can often be explicitly bounded \citep[e.g. autoregressive models, see][]{geer2013}.
The $\sqrt{d}$ term reflects the cost of requiring \emph{uniform} control over all possible neighbourhood regression problems, which is a worthwhile trade-off as we shall see in the next section.

\section{Applications}
\label{sec:app}

The results in Section~\ref{sec:results} provide uniform control over the 
family of estimators $\{\dagadjest(\perm):\perm\in\allperms\}$. As stated, however, these theorems are fairly abstract. In this section we discuss three applications of these results to score-based learning, causal inference, and graphical models. Proofs are deferred to Section~\ref{sec:proofs}.

\subsection{Score-based learning}
\label{subsec:app:scorebased}

As our first application, we use Theorems~\ref{thm:main:supp} and~\ref{thm:main:dev} to provide statistical guarantees for the score-based estimator $\dagadjest$, defined by \eqref{eq:intro:mainestimator} and \eqref{eq:def:plsobj}. Our results will come in three flavours: (i) consistency guarantees (Section~\ref{subsubsec:app:consistency}), (ii) sparsity guarantees, (Section~\ref{subsubsec:app:sparsity}), and (iii) an oracle inequality (Section~\ref{subsubsec:app:oracle}). The results in Section~\ref{subsubsec:app:consistency} will show that there is a DAG (denoted by $\tB{\estperm}$) which $\dagadjest$ is ``close'' to, and in Sections~\ref{subsubsec:app:sparsity} and~\ref{subsubsec:app:oracle} we will carefully study the DAG $\tB{\estperm}$ in order to show that the program \eqref{eq:intro:mainestimator} adaptively selects an equivalent DAG with appealing properties such as sparsity.

\subsubsection{Consistency guarantees}
\label{subsubsec:app:consistency}

So far, our results have been stated in terms of the \emph{restricted} minimizers $\dagadjest(\perm)$ of Definition~\ref{defn:restrictedmin}.
Obviously, $\dagadjest$ is also a restricted minimizer for some permutation(s) $\perm$, although \emph{a priori}, we do not know which one(s). 
The following definition formalizes the collection of such permutations:

\begin{defn}
\label{defn:estperms}
The \emph{collection of estimated permutations} is
\begin{align*}
\estperms
= \{\perm\, \in\, \allperms : \dagadjest\in\subdagspace{\perm}\}.
\end{align*}
\noindent
An arbitrary element of $\estperms$ will be denoted by $\estperm$. 
\end{defn}

The following statements are equivalent:
(i)~$\perm \in \estperms$, (ii)~$\dagadjest\in\subdagspace{\perm}$, (iii)~$\permmat_{\perm}\dagadjest$ is lower triangular, and (iv)~$\dagadjest=\dagadjest(\perm)$. Since $\dagadjest$ depends on the random matrix $\sampledagmat$, the set of permutations $\estperms$ is also a random quantity. This randomness complicates any direct attempt at analyzing $\dagadjest$, and is why we first established uniform control over all possible permutations in Section~\ref{sec:results}. Note that to solve \eqref{eq:intro:mainestimator}, it is not necessary to compute $\estperms$---in fact, in practical applications, once a solution $\dagadjest$ to \eqref{eq:intro:mainestimator} is found, $\estperms$ can be obtained as a byproduct by finding topological sorts of the DAG  $\dagadjest$. 

Since a global minimizer is indeed a restricted minimizer, the following theorem is an immediate corollary of Theorems~\ref{thm:main:supp} and~\ref{thm:main:dev}:
\begin{thm}
\label{thm:scorebased}
Take $\pl=\pl(\,\cdot\,;\mcpparam)$ as in \eqref{eq:mcp} and assume $\trueCov$ is positive definite with bounded eigenvalues. Suppose further that Conditions~\ref{condn:evals} and~\ref{condn:pl:basic} hold. Then there exist positive constants $\covconst_j=\covconst_j(\trueCov)$ ($j=1,2,3$) such that if
$n > \covconst_1 \, d\log p$ and $\lambda\ge \covconst_2\, \sqrt{(d+1)\log p/n}$ then for $r=1,2$ and with probability at least $1-\probbound$,
\begin{align*}
\norm{\dagadjest-\tB{\estperm}}_{r}
&\les \frac{\penderiv}{\mineval(\trueCov)}\norm{\tB{\estperm}}_{0}^{1/r}
\quad\forall\,\estperm\in\estperms, 
\text{ and } \\
\norm{\dagcolest_j-\dagcol_j(\estperm)}_{r} 
&\les \frac{\penderiv}{\mineval(\trueCov)}\norm{\dagcol_j(\estperm)}_{0}^{1/r}
\quad\forall\,{1\le j\le p},\,\,\estperm\in\estperms.
\end{align*}

\noindent
If also $\betamin(\trueCov) \;>\; (1+\mcpparam)\lambda$ for some $\mcpparam>\covconst_3>0$, then with the same probability 
\begin{align*}
\supp(\dagadjest) = \supp(\tB{\estperm})
\quad\forall\,\estperm\in\estperms.
\end{align*}
\end{thm}

\noindent
This theorem provides preliminary justification for the use of score-based estimators on high-dimensional datasets, providing practitioners with finite-sample guarantees for structure learning and parameter estimation. In addition to an upper bound on the overall deviation $\norm{\dagadjest-\tB{\estperm}}_{r}$, Theorem~\ref{thm:main:dev} provides control over the column deviations $\norm{\dagcolest_j-\dagcol_j(\estperm)}_{r}$, which highlights another advantage of our uniform control results. Furthermore, although our result is stated for the MCP, similar results are available for other regularizers using, e.g. \citet{loh2014nonconvex}.

\subsubsection{Sparsity}
\label{subsubsec:app:sparsity}

Theorem~\ref{thm:scorebased} has a drawback: Although it guarantees that \emph{there exists} a DAG $\tB{\estperm}\in\eqclass(\trueCov)$ which our estimator $\dagadjest$ is ``close'' to, nothing is said about \emph{which} member of the equivalence class this is. Due to the inclusion of the regularizer $\pl$ in the score function, one hopes the estimated permutation $\estperm$ to be well-behaved, for example, in the sense that the corresponding population DAG $\tB{\estperm}$ is sparse. Theorem~\ref{thm:scorebased}, however, is silent about the properties of $\tB{\estperm}$. In this section, we complete this picture by showing that $\tB{\estperm}$ is sparse. In the next section, we will strengthen this result by proving an oracle-type inequality for $\tB{\estperm}$ that holds under weaker assumptions than those in the current section.

We now show that the number of edges in $\tB{\estperm}$ can be controlled under some assumptions on the \emph{signal strength} and a \emph{minimum-trace DAG} (Definition~\ref{defn:mintrace}). Since $\tB{\estperm}$ is approximated by $\dagadjest$ according to Theorem~\ref{thm:scorebased}, this implies sparsity of $\dagadjest$ as well.

\begin{condition}[Signal strength]
\label{condn:betamin}
The minimum signal $\betamin$ satisfies
\begin{align*}
\pl(\betamin)
\ge \betaminconst\frac{\penderiv^{2}}{\mineval(\trueCov)}\quad\text{for some $\betaminconst>2$.}
\end{align*}
\end{condition}

\noindent
If we assume that the eigenvalues of $\trueCov$ are bounded, then for penalties satisfying Condition~\ref{condn:pl:basic} and for which $\penderiv = O(\lambda)$ (e.g., both the MCP and $\ell_{1}$), Condition~\ref{condn:betamin} will hold if $\betamin\ges\lambda\ges\sqrt{(d+1)\log p/n}$, which is the same scaling required by Lemma~\ref{lem:mcp:example}. 

Our final condition involves a so-called \emph{minimum-trace DAG}:

\begin{defn}
\label{defn:mintrace}
A \emph{minimum-trace permutation} is any
\begin{align*}
\goodperm
\;\in\;\argmin_{\perm\,\in\,\allperms} \;\tr\tOm{\perm}.
\end{align*}

\noindent
The corresponding DAG $\tB{\goodperm}$ is called a \emph{minimum-trace DAG}.
\end{defn}

\noindent
Minimum-trace DAGs are discussed in more detail in Section~\ref{subsec:app:causal}.

\begin{condition}[Minimum-trace]
\label{condn:mintrace:bound}
There exists a minimum-trace permutation $\goodperm$ such that
\begin{align*}
\frac{\pl(\tB{\goodperm})}{\tr\tOm{\goodperm}}
\ge \mintraceconst\sqrt{\frac{(d+1)\log p}{n}}
\quad\text{for some $\mintraceconst>0$.}
\end{align*}
\end{condition}

\noindent
In light of the lower bound on $\lambda$ required by Theorem~\ref{thm:main:dev}, Condition~\ref{condn:mintrace:bound} can be rewritten as $\pl(\tB{\goodperm}) \ges (\lambda/\covconst_{2})\cdot\tr\tOm{\goodperm}$ which puts a lower bound on the penalty, as measured by the minimum-trace DAG $\tB{\goodperm}$.

The following result shows how the sparsity of $\dagadjest$ is essentially sandwiched between that of $\tB{\estperm}$ and $\tB{\goodperm}$:

\begin{thm}
\label{thm:main:sparse}
Assume the conditions of Theorem~\ref{thm:main:dev} hold along with Conditions~\ref{condn:betamin} and~\ref{condn:mintrace:bound}. Then 
\begin{align}\label{eq:temp:684}
\pl(\tB{\estperm})
\les \pl(\dagadjest)
\les \pl(\tB{\goodperm})
\end{align}
with probability at least $1 - \probbound$.
If $\pl$ is also $\ell_{0}$-compatible, then with the same probability,
\begin{align}
\label{eq:temp:685}
\norm{\tB{\estperm}}_0
\les \mineval(\trueCov)\,\frac{\lambda^2}{\penderiv^2}\,\norm{\tB{\goodperm}}_0.
\end{align}
\end{thm}

\noindent
The constants in~\eqref{eq:temp:684} depend only on $a_1$ and $a_2$ in Conditions~\ref{condn:betamin} and~\ref{condn:mintrace:bound}, while the constant in~\eqref{eq:temp:685} also depends on $\uppflat$ (from Definition~\ref{defn:L0compat}). 

Combining Theorems~\ref{thm:scorebased} and~\ref{thm:main:sparse}, it follows that $\dagadjest$ estimates a sparse DAG model for the distribution $\normalN_{p}(0,\trueCov$). This is arguably the best one can hope without assuming there is an identifiable DAG for the joint Gaussian distribution.

\begin{remark}
\label{rem:sparsity}
Theorem~\ref{thm:main:sparse} applies to regularizers that are not $\ell_{0}$-compat\-ible, such as the $\ell_{1}$ penalty. For example, we can achieve $\norm{\tB{\estperm}}_{1}\les\norm{\dagadjest}_{1}\les\norm{\tB{\goodperm}}_{1}$ with the $\ell_{1}$ penalty. The idea is that the penalty $\pl$ itself may be interpreted as a ``measure of sparsity'' that is weaker than the $\ell_{0}$ norm. Moreover, although  the $\ell_{0}$ penalty does not satisfy Condition~\ref{condn:pl:basic}, which is required in Theorem~\ref{thm:main:dev}, \eqref{eq:temp:684} still applies to the $\ell_{0}$ penalty (see Section~\ref{subsec:conclusion:ell0}). Compared to Theorem~3.1 in \citet{geer2013}, due to our use of the least-squares loss instead of the log-likelihood (see Section~\ref{subsec:conclusion:comp}), our results do not guarantee that $\norm{\tB{\estperm}}_{0}$ is on the same order as the number of edges in the sparsest DAG. Furthermore, Theorem~\ref{thm:main:sparse} stops short of providing a lower bound on $\norm{\tB{\estperm}}_0$, which allows for slightly stronger bounds compared to \eqref{eq:temp:684}. If we assume that $\tB{\goodperm}$ is also one of the sparsest DAGs, however, then the number of edges in $\dagadjest$ is on the same order as the sparsest DAGs in $\eqclass(\trueCov)$.
\end{remark}

\begin{remark}
\label{rem:betamin}
Combining Theorems~\ref{thm:scorebased} and~\ref{thm:main:sparse}, with high probability,
\begin{align}
\norm{\dagadjest-\tB{\estperm}}_{2}^2
&\les \frac{\lambda^2}{\mineval(\trueCov)} s_0,
\end{align}
where $s_0=\norm{\tB{\goodperm}}_{0}$ is the number of edges in a minimum-trace DAG. Let us investigate under which scalings of $(n,p,\maxdeg,s_0)$ the $\ell_2$ error vanishes asymptotically. Choose $\lambda$ to be on the order of $\sqrt{\maxdeg\log p /n}$ and assume that $\trueCov$ has bounded eigenvalues for simplicity. Since the squared $\ell_2$ error is essentially the sum over $p$ regression problems, we normalize it by $1/p$. Thus, to achieve consistency in the normalized $\ell_2$ error, it is sufficient to have $(s_0/p)\maxdeg\log p \ll n$, where $(s_0/p)$ is the average parent size of the minimum-trace DAG $\tB{\goodperm}$. If $\maxdeg\log p \le \epsilon n$ with $\epsilon=o(1)$, then we can have $s_0/p \to \infty$ as long as $s_0/p \ll 1/\epsilon$, which imposes a quite weak sparsity assumption on $\tB{\goodperm}$. Under this asymptotic scaling, Condition~\ref{condn:betamin} allows $\betamin\to0$ for the MCP and $\ell_1$, and Condition~\ref{condn:mintrace:bound} allows $\pl(\tB{\goodperm})/\tr\tOm{\goodperm}\to0$. Furthermore, it is possible to allow $p\gg n$, which justifies our results for high-dimensional data. To establish normalized $\ell_2$-consistency, \cite{geer2013} assume that $\epsilon$ is sufficiently small but does not necessarily vanish asymptotically (their Condition 3.4). The fact that we need a slightly stronger assumption on $\maxdeg$ to obtain $\ell_{2}$-consistency is probably the price we pay for obtaining uniform error control over all $\perm\in\allperms$.
\end{remark}

\begin{remark}
\label{rem:zerodag}
For the degenerate case $\tB{\goodperm}=0$, i.e. a minimum-trace DAG is the empty graph, Condition~\ref{condn:mintrace:bound} cannot hold. It is easy to check, however, that in this case it follows that $\tB{\perm}=0$ for all $\perm\in\allperms$, and hence Theorem~\ref{thm:main:supp} implies both $\pl(\tB{\estperm})=\pl(\dagadjest)=\pl(\tB{\goodperm})=0$ and $\norm{\tB{\estperm}}_{0}=\norm{\dagadjest}_{0}=\norm{\tB{\goodperm}}_{0}=0$.
\end{remark}

\subsubsection{An oracle inequality}
\label{subsubsec:app:oracle}

In order to control the sparsity of $\dagadjest$ and $\tB{\estperm}$, Theorem~\ref{thm:main:sparse} requires a condition on a minimum-trace permutation $\goodperm$ (Condition~\ref{condn:mintrace:bound}). 
Here we provide a different guarantee on $\tB{\estperm}$, in the form of an oracle inequality, which is true without any extra assumptions.

Observe that $\frac12\tr\tOm{\perm}$ is the same as the expected loss, i.e.
\begin{align}
\label{eq:exploss}
\E\Big[\frac{1}{2n}\norm{\sampledagmat-\sampledagmat\tB{\perm}}_{\frob}^2\Big]
= \frac12\tr\tOm{\perm}
\quad\text{for any $\tB{\perm}\in\eqclass(\trueCov)$.}
\end{align}

\noindent
This gives us an expression for the expected penalized loss, which we denote
\begin{align*}
    \Qt_{\lambda}(\pi):=\frac12\tr\tOm{\pi}+ \pl(\tB{\pi}).
\end{align*}

\noindent
If we let $\pi^*\in\argmin_{\pi}\Qt_{\lambda}(\pi)$, it follows from Lagrange duality that the DAG $\tB{\pi^*}\in\eqclass(\trueCov)$ minimizes the weak sparsity measure $\pl$ subject to some upper bound on the expected prediction error, i.e., $\tr\tOm{\perm}\leq c(\lambda)$. 
The next result shows that the expected penalized loss of an estimated permutation $\estperm$ is (up to a vanishing term) on the same order as that of the population minimizer, thus giving another kind of control over the sparsity of $\tB{\estperm}$.

\begin{thm}
\label{thm:oracle}
Suppose Conditions~\ref{condn:evals},~\ref{condn:pl:basic}, and~\ref{condn:betamin} hold. Then, there exist positive constants $\covconst_1=\covconst_1(\trueCov)$ and $\covconst_2=\covconst_2(\trueCov)$ such that if
\begin{align*}
n
> \covconst_1 \, d\log p, 
\quad
\lambda
\ge \covconst_2 \,\sqrt{\frac{(d+1)\log p}{n}},
\end{align*}

\noindent 
then
\begin{align*}
\Qt_{\lambda}(\estperm)\; \le\; \frac{3\betaminconst+2}{\betaminconst-2}\bigg(\, 1+6\sqrt{\frac{(d+1)\log p}{n}}\, \bigg)
    \cdot\inf_{\perm} \Qt_{\lambda}(\pi)
\end{align*}

\noindent
with probability at least $1-\probbound$.
\end{thm}

\noindent
The proof of this inequality can be found in Appendix~\ref{app:thm:oracle}. Note that the constant factor of $(3\betaminconst+2)/(\betaminconst-2)$ approaches 3 as the signal strength $\betamin$ increases, and $\sqrt{(d+1)\log p/n}=o(1)$ as long as $n\gg (d+1)\log p$.

\subsection{Causal DAG learning}
\label{subsec:app:causal}

\def\trueDAG{\dagnopi}
\def\trueOm{\varnopi}
So far we have avoided singling out any particular DAG for special treatment, instead opting to provide \emph{uniform} control over the entire equivalence class $\eqclass(\trueCov)$. Of course, without additional information about $\trueCov$, the parameters $(\dagnopi,\varnopi)$ are nonidentifiable. However, under additional assumptions it may be possible to uniquely identify a DAG $\trueDAG\in\eqclass(\trueCov)$. In this setting, one often refers to $\trueDAG$ as a \emph{causal DAG}: It represents the unique (directed) graphical structure that describes the relationship between the random variables $(\dagvec_{1},\ldots,\dagvec_{p})$ via the SEM \eqref{eq:intro:streqnmodel}, and thus can be interpreted as a ``causal'' explanation of the underlying joint distribution as long as there are no unmeasured confounders \citep[among other assumptions, see][for details]{spirtes2000}. In other words, one can regard $\trueDAG$ as the true DAG for the observed data. It is then an interesting problem to study under what assumptions $\dagadjest$ recovers $\trueDAG$. Note that $\trueDAG$ may be consistent with more than one permutation. 

In the previous section, we defined the concept of a minimum-trace DAG (Definition~\ref{defn:mintrace}), which by definition minimizes the sum of the error variances $\sum_j \dagerrvar_j^2(\perm)$ over all $\perm$. Thus, minimum-trace DAGs minimize the total squared prediction error among all  DAG parameterizations of the underlying joint Gaussian distribution. While minimum-trace DAGs are not unique in general, when they are we can guarantee their recovery via $\dagadjest$.

Given $\trueDAG\in\eqclass(\trueCov)$, define $\trueOm$ to be such that $\trueCov(\trueDAG,\trueOm)=\trueCov$ (cf. \eqref{eq:def:sigmafcn}).

\def\trconst{a_{3}}
\begin{thm}
\label{thm:mintr}
Suppose that the minimum-trace DAG $\trueDAG$ is unique and there exists $\trconst>10$ such that 
\begin{align*}
\frac{\tr\trueOm}{\tr\tOm{\perm}}
&\le 1-\trconst\sqrt{\frac{(d+1)\log p}{n}}
\quad\forall\,\tr\tOm{\perm} \ne \tr\trueOm.
\end{align*}

\noindent
Then there exist constants $\covconst_{j}=\covconst_{j}(\trueCov)$ such that if
\begin{enumerate}\itemsep=1ex
    \item $\pl$ is $\ell_{0}$-compatible,
    \item $n > \covconst_1 \, (d+1)^{3}\log p$,
    \item $\lambda\ge \covconst_2\,  \sqrt{\frac{(d+1)\log p}{n}}$,
    \item $\betamin(\trueCov) \;>\; (1+\mcpparam)\lambda$ for some $\mcpparam>\covconst_3>0$,
\end{enumerate}

\noindent
then with probability at least $1-\probbound$, it holds that 
\begin{align*}
\supp(\dagadjest) = \supp(\trueDAG) 
\quad\text{ and }\quad
\norm{\dagadjest-\trueDAG}_{r}
\les \frac{\penderiv}{\mineval(\trueCov)}\norm{\trueDAG}_{0}^{1/r}.
\end{align*}

\noindent
Furthermore, if $\trueDAG$ is not unique, then all of the above statements are true with $\trueDAG$ replaced by some minimum-trace DAG. 
\end{thm}

\noindent
Although the previous theorem remains valid when $\trueDAG$ is not unique, it is still of interest to provide conditions under which minimum-trace DAGs are indeed unique. Suppose that we know the error variances $\trueOm$ for the true DAG. Then by proper rescaling of $X_j$, we can always make $\trueOm=\omega_0^2I$, which makes $\trueDAG$ an \emph{equivariance DAG}. As it turns out, an equivariance DAG is automatically the unique minimum-trace DAG and therefore by Theorem~\ref{thm:mintr}, $\dagadjest$ recovers the true causal DAG $\trueDAG$ accurately.

\begin{lemma}
\label{lem:equalvar}
Suppose $\trueCov$ is given and \eqref{eq:intro:streqnmodel} holds for some $\dagnopi\in\dagspace$ with $\varnopi=\omega_0^2I$. Then $\dagnopi$ is the unique minimum-trace DAG.
\end{lemma}

\noindent
Thus, a practical use case of Theorem~\ref{thm:mintr} is as follows: If one can estimate the error variances from data beforehand or if one has reasons to believe that the true error variances are approximately identical, then score-based learning \eqref{eq:intro:mainestimator} can be used to estimate the true DAG. 

As a related result, \citet{peters2013} have shown that equivariance DAGs are identifiable.
Note that a minimum-trace DAG always exists, whereas an equivariance DAG may not (Example~\ref{ex:mainexample}). Therefore, the minimum-trace property provides a natural generalization of the equivariance property to general covariance matrices that are not equivariant.
As a consequence of Lemma~\ref{lem:equalvar}, we immediately obtain the following corollary:

\begin{cor}
\label{cor:eqvar}
Theorem~\ref{thm:mintr} remains true if ``minimum-trace'' is replaced with ``equivariance''.
\end{cor}

\noindent
Thus, if we are interested in recovering an equivariance DAG $\trueDAG$ then Corollary~\ref{cor:eqvar} shows that $\supp(\dagadjest)=\supp(\trueDAG)$. To the best of our knowledge, this is the first result that guarantees consistent structure learning of a Gaussian DAG when $p\gg n$.

\begin{remark}
\sloppypar{A related but weaker result first appeared in \citet{geer2013}: Deviation, sparsity bounds, and permutation recovery are proved for equivariance DAGs under a low-dimensional setting with $p\les n/\log n$. 
Theorem~\ref{thm:mintr} thus provides a strengthening of this result to include support recovery for arbitrary covariance matrices and general nonconvex penalties under the high-dimensional scaling $n\ges (d+1)^{3}\log p$.}
\end{remark}

\subsection{Conditional independence learning}
\label{subsec:app:ci}

Our final application concerns the problem of learning CI relations for a $p$-variate Gaussian distribution $\normalN_p(0,\trueCov)$.
Throughout this section we assume the reader is familiar with concepts from the graphical modeling literature such as $d$-separation, Markov equivalence, faithfulness, Markov perfectness, etc. In particular,  note that the equivalence class $\eqclass(\trueCov)$ is \emph{not} the same as the Markov equivalence class.

\def\ci{\mathcal{I}}
In graphical modeling, the goal is to represent a distribution via a graph in which separation can be used to read off CI relations.
In this way, a DAG $G=(V=[p],E)$ defines---via the notion of $d$-separation---a set of pairwise CI relations, which we denote by $\ci(G)\subset[p]\times[p]\times2^{[p]_{ij}}$. A triplet $(i,j,S)\in \ci(G)$ indicates that that $X_{i}$ is $d$-separated from $X_{j}$ by $X_{S}$ in the graph $G$, and the Markov property for Bayesian networks implies that $X_{i}\independent X_{j}\| X_{S}$.
Specialized to our setting, we use the shorthand $\ci(\perm)=\ci(\tB{\perm})$ and $\widehat{\ci}(\perm) = \ci(\dagadjest(\perm))$. Finally, for the joint Gaussian $\normalN_p(0,\trueCov)$, we denote by $\ci(\trueCov)$ the set of all pairwise CI relations among $X_1,\ldots,X_p$. Recall that an \emph{I-map} is any graph such that $\ci(G)\subset\ci(\trueCov)$, and a \emph{minimal I-map} is an I-map such that removing any edge from the graph violates this inclusion.

In general, the sets $\ci(\perm)$, $\widehat\ci(\perm)$, and $\ci(\trueCov)$ can all be different, although $\tB{\perm}$ is always a minimal I-map for $\normalN_p(0,\trueCov)$, and hence in particular $\ci(\perm)\subset\ci(\trueCov)$. On the other hand, it is not hard to see that $\bigcup_{\perm\, \in\, \allperms}\ci(\perm) = \ci(\trueCov)$. If there exists a single permutation $\perm$ such that $\ci(\perm)=\ci(\trueCov)$, then we call $\tB{\perm}$ \emph{faithful}---i.e. all pairwise CI relations of the underying Gaussian distribution can be read off from the graph $\tB{\perm}$ via the $d$-separation criterion. Under the assumptions of Theorem~\ref{thm:mintr}, if the joint Gaussian distribution $\normalN_p(0,\trueCov)$ is faithful to the DAG $\tB{\goodperm}$, then the graphical structure of $\dagadjest$ contains all CI relations among $X_1,\ldots,X_p$, which is implied by the exact support recovery. That is, $\ci(\dagadjest)=\ci(\estperm)=\ci(\goodperm)=\ci(\trueCov)$. Unfortunately, faithfulness is a strong assumption that rarely holds in practice \citep{lin2014,uhler2013}. We illustrate this with the following simple examples:

\begin{example}
\label{ex:ci}
Consider the covariance matrix $\trueCov = \trueInv^{-1}$ where
\begin{align}
\trueInv = \begin{pmatrix}
10 & 1 & 0 & 2 \\
1 & 10 & 3 & 0 \\
0 & 3 & 10 & 4 \\
2 & 0 & 4 & 10
\end{pmatrix}.
\end{align}

\noindent
This covariance matrix implies two conditional independence relations: 
\begin{align*}
X_{1}\independent X_{3}\|\{X_{2},X_{4}\}, \quad X_{2}\independent X_{4}\|\{X_{1},X_{3}\}.
\end{align*}
This example corresponds to the so-called undirected ``diamond'' graph, which does not admit a faithful DAG, as implied by the following argument. 
Notwithstanding, the equivalence class $\eqclass(\trueCov)$ is still well-defined and can be partitioned into two sets $\eqclass(\trueCov)=\eqclass_{1}\cup\eqclass_{2}$, where $\eqclass_{j}$ is a Markov equivalence class of DAGs for $j=1,2$. That is, the DAGs in each $\eqclass_{j}$ encode the same set of CI relations (via $d$-separation). In this example, each DAG in $\eqclass_{1}$ implies exactly one CI relation, namely $X_{1}\independent X_{3}\|\{X_{2},X_{4}\}$, and similarly $\eqclass_{2}$ implies exactly one CI relation, namely $X_{2}\independent X_{4}\|\{X_{1},X_{3}\}$. Using the notation we introduced above, if $\tB{\perm}\in\eqclass_{1}$ then $\ci(\perm)=\{(1,3,\{2,4\})\}$ and if $\tB{\perm}\in\eqclass_{2}$ then $\ci(\perm)=\{(2,4,\{1,3\})\}$.
\end{example}

\begin{example}
The deficiency shown in Example~\ref{ex:ci} is not unique to directed graphs: There exist $\trueCov$ that are faithful to a DAG but that are not perfect with respect to any Markov network. An example is given by the so-called \emph{v-structure}, i.e. $X_{1}\to X_{2}\leftarrow X_{3}$. Moreover, by combining a $v$-structure with the diamond graph in the previous example, we can construct a Gaussian distribution that is neither faithful  nor perfect. Let $\normalN(0,\trueCov_{1})$ denote the distribution in Example~\ref{ex:ci}, $\normalN(0,\trueCov_{2})$ denote any trivariate Gaussian distribution parametrized by a faithful $v$-structure, and define
\begin{align*}
\trueCov
= \begin{pmatrix}
\trueCov_{1} & 0\\
0 & \trueCov_{2}
\end{pmatrix}
\in\R^{7\times 7}.
\end{align*}

\noindent
Then $\normalN(0,\trueCov)$ is neither faithful nor perfect.
\end{example}

Evidently, directed and undirected graphs encode different sets of CI relationships. Furthermore, there exist examples where a minimal DAG is much more parsimonious than a minimal Markov network and vice versa~\citep{koller2009}. The key is that in practice one cannot determine in advance which representation is optimal, so it is important to be able to learn as much as possible from either type of model. Is it possible to learn $\ci(\trueCov)$ without the existence of a faithful DAG or a perfect Markov network?

To answer this question, we consider the use of the restricted minimizers $\dagadjest(\perm)$ for a potentially large set of permutations to learn CI relations from data without assuming that $\normalN_p(0,\trueCov)$ is faithful to some DAG. One hopes that by taking the union over all (or a large collection of) the CI relations obtained in this way, we subsume the CI relations implied by any single DAG or Markov network. This heuristic is justified by the following theorem:

\begin{thm}
\label{thm:ci}
Under the conditions of Theorem~\ref{thm:scorebased}, with probability at least $1-\probbound$:
\begin{enumerate}
\item $\dagadjest(\perm)$ is a minimal $I$-map of $\normalN_p(0,\trueCov)$ for each $\perm$; 
\item $\bigcup_{\perm\, \in\, \allperms}\widehat{\ci}(\perm)=\ci(\trueCov)$.
\end{enumerate}
\end{thm}

\noindent
This result is a consequence of uniform model selection consistency in Theorem~\ref{thm:main:supp} and is proved in Section~\ref{subsec:proofs:ci}. Conclusion 1 implies that each of the learned graphs $\dagadjest(\perm)$ are as parsimonious as possible, and Conclusion 2 shows that we can detect all CI relations by enumerating all permutations $\pi$. Moreover, if we take any subset $K\subset\allperms$ of permutations and compute $\cup_{\perm\in K}\widehat{\ci}(\perm)$, Conclusion 2 implies that there will be no false positive errors in this set. 
Applied to Example~\ref{ex:ci}, this shows that although no single permutation is sufficient to learn the full CI structure, \emph{two} permutations suffice.

Since $\dagadjest(\perm)$ corresponds to an autoregressive model (Section~\ref{subsec:prelim:minimizers}), it can be estimated efficiently and consistently via vanilla sparse regression. Theorem~\ref{thm:ci} then justifies testing any subset of permutations (possibly \emph{all} permutations) to discover CI relations in a distribution. Since the set $K$ can be as large as $|K|=p!$, establishing the uniform validity of all possible subsets is clearly nontrivial. 
Furthermore, Theorem~\ref{thm:ci} has an interesting implication for the PC algorithm \citep{kalisch2007}: It implies that one can use regression as an independence oracle (e.g. in place of the Fisher $z$-test) to detect CI relations consistently from high-dimensional data. Since the PC algorithm runs in $O(p^{d})$ time, this gives an efficient way to perform the search over permutations described above.

\section{Comparison}
\label{subsec:app:comparison}

Our results are most closely related to \citet{geer2013}, who study a thresholded $\ell_{0}$-penalized MLE under the same statistical model. Our results generalize their results in several ways:

\begin{enumerate}
    \item We prove that $\dagadjest$ attains exact support recovery when $p\gg n$, which has not been shown previously to the best of our knowledge. For a more detailed discussion, see Remark~\ref{rem:L0supp}.
    \item We have shown how $\tB{\estperm}$ mimics the behaviour of an oracle minimizer of the expected penalized loss (Section~\ref{subsubsec:app:oracle}).
    \item We generalize existing results on recovery of an equivariance DAG to high-dimensions and general nonconvex regularizers (Section~\ref{subsec:app:causal}).
    In order to exploit equivariance, one must essentially assume that the noise variance is known up to a multiplicative constant. By relaxing this assumption to minimum-trace DAGs, our results provide theoretical guarantees in the practical setting where we have no prior knowledge about the error variances $\varnopi$.
    \item We allow for a general class of tractable penalties $\pl$ beyond the $\ell_{0}$ penalty, which is known to be intractable.
    Although still nonconvex, with a concave penalty the program \eqref{eq:intro:mainestimator} is defined over a continuous parameter space and can be solved exactly with dynamic programming \citep{silander2012,xiang2013}.
    \item \sloppypar{We consider optimization over the full space of parameters $\dagspace$, instead of a thresholded parameter space as in \cite{geer2013} (see their Condition~3.3). This requires a more delicate analysis and makes our method more practical since the thresholded parameter space involves an unknown constant which is difficult to select.}
    \item We apply our results to the problem of CI learning (Section~\ref{subsec:app:ci}) and formally justify the use of penalized regression as an independence oracle inside the PC algorithm. These results are made feasible by our novel results on support recovery for all permutations $\perm$.
\end{enumerate}

\noindent
Our theorems on deviation bounds (Theorem~\ref{thm:main:dev}) and sparsity (Theorem~\ref{thm:main:sparse}) are similar in spirit to \cite{geer2013}, however, our sparsity bounds use a minimum-trace DAG in the upper bound, as opposed to a minimal-edge DAG (Remark~\ref{rem:sparsity}). 

The present work is also related to \citet{raskutti2014}, which seeks to find the sparsest Bayesian network for a distribution $\pr$ in a nonparametric setting. They show that their algorithm is uniformly consistent \citep{zhang2002} under strictly weaker assumptions than the PC algorithm \citep{kalisch2007}, although it is not discussed if these results are applicable when $p\gg n$ or what the scaling between $(n,p,d)$ would be in such a setting. Their method is a two-stage method, with the primary step involving a constraint-based search similar to the PC algorithm. By contrast, we have focused on score-based estimators, which are faster and more popular in applications (Section~\ref{subsec:intro:motivation}). For example, compared to \citet{raskutti2014}, whose numerical simulations only explore small graphs with $p\le 8$, \citet{aragam2015} provide examples with $p=8000$.

Finally, \citet{shojaie2010} proved the model selection consistency of an $\ell_{1}$-based PLS estimator for a single, \emph{fixed} permutation $\fixperm$ under an irrepresentability condition.
By assuming that $\fixperm$ is known, they circumvent all of the technical challenges associated with estimating an \emph{unknown} permutation, in effect proving the consistency of $\dagadjest(\fixperm)$ for a \emph{single} permutation. This is equivalent to estimating a standard autoregressive model, which is well-studied.

\section{Discussions}
\label{sec:conclusion}

To conclude, we discuss various issues related to extending the results derived herein to more general settings.

\subsection{The $\ell_{0}$ case}
\label{subsec:conclusion:ell0}

Condition~\ref{condn:pl:basic} precludes certain regularizers including the $\ell_{0}$ penalty, however, a simple modification to the proofs accounts for any regularizer that is $\ell_{0}$-compatible as defined in Definition~\ref{defn:L0compat}, including the $\ell_{0}$ penalty (see Remark~\ref{rem:ell0}). 
Thus everything derived in this work applies equally to the $\ell_{0}$ penalty. In particular, Theorem~\ref{thm:main:sparse} implies $\norm{\tB{\estperm}}_0\les \norm{\dagadjest}_0\les \norm{\tB{\goodperm}}_0$. If $\tB{\goodperm}$ is one of the sparsest DAGs in $\eqclass(\trueCov)$, then the number of edges in $\dagadjest$ will be on the same order as that of the sparsest ones, as in \cite{geer2013}. For example, this occurs when an equivariance DAG is amongst the sparsest.

\subsection{Computation and loss functions}
\label{subsec:conclusion:comp}

Since \eqref{eq:intro:mainestimator} is a nonconvex program, computation of $\dagadjest$ is challenging and in fact NP-hard \citep{chickering2004}. Fortunately, there are fast algorithms via dynamic programming for finding globally optimal Bayesian networks \citep{ott2003,singh2005,silander2012}. For example, by combining dynamic programming with A* search, \citet{xiang2013} propose an exact algorithm to compute the $\ell_{1}$-regularized version of $\dagadjest$ (cf. \eqref{eq:def:plsobj}) that is tractable on problems with hundreds of nodes. 

This highlights one reason for choosing the least-squares loss, i.e. it leads to more efficient computation in practice. Moreover, the least-squares loss has the potential for generalization to subgaussian distributions for which there is no closed-form likelihood.
What is necessary in our analysis, however, is that~\eqref{eq:def:plsobj} factors into $p$ neighbourhood problems (for each fixed ordering of the variables), as detailed in Lemma~\ref{lem:restricted:nhbd}, a property that holds for other losses as well. For example, we could have used the corresponding negative log-likelihood (NLL) of~\eqref{eq:intro:streqnmodel}, as in~\cite{geer2013}, instead of least-squares. 
When used with the $\ell_{0}$ penalty, the NLL has the appealing property of \emph{score equivalence} \citep{heckerman1995}, which assigns the same score to DAGs that are Markov equivalent. While conceptually appealing, our results show that this property is not necessary to identify a sparse DAG.

\subsection{Undirected models and extensions}
\label{subsec:conclusion:ext}
 
The problem of learning DAGs from data is closely related to the problem of learning undirected \emph{Markov networks}. In particular, our results must be contrasted with the much simpler problem of estimating an undirected Gaussian graphical model, which has received much more attention. \citet{meinshausen2006} showed how to reduce the estimation of a Gaussian graphical model to neighbourhood regression by regressing each $\dagvec_{j}$ onto the rest of the variables $\dagvec_{-j}$. This procedure works for other families of undirected graphical models, including Ising models \citep{ravikumar2010}, discrete models \citep{loh2013}, and exponential random families \citep{yang2015}. In this setting, there are only $p$ neighbourhoods, compared to $O(p!)$ for DAGs.

As an immediate corollary of the results derived in this work, we also obtain consistency of this popular neighbourhood-regression based estimator of the Gaussian graphical model. This follows since, in particular, we have uniform control over the neighbourhood problems $\abstractpls_\lambda(\sampledagmatelem_{j}, \sampledagmat; [p]_{j})$ for each $j=1,\ldots,p$. In this case, there are only $p$ neighbourhood problems to control, much smaller than $p{p\choose d}$ for DAGs. One advantage of our results is the relaxation of the strong \emph{neighbourhood stability} assumption (also known as the \emph{irrepresentability condition}) through the use of nonconvex penalties such as the MCP and SCAD.

It is natural to inquire how these results extend to more general settings with non-Gaussian structural equations. For example, the neighourhood-level concepts (e.g. Definition~\ref{defn:msexp} and Theorem~\ref{thm:fixed:l1l2bound}) easily extend to subgaussian errors, and as a consequence much of the machinery developed here applies to subgaussian classes. The difficulty, however, is that the \emph{population-level} structure of such structural equations (i.e. $\eqclass(\trueCov)$) is not well understood in the non-Gaussian setting, and so while one could obtain formal results in this direction it is not clear how strong the assumptions would be. For example, for different permutations the errors $\dagerr_j(\perm)$ are no longer guaranteed to be independent of the parents $\dagvec_{\dagcand_{j}(\perm)}$ if the joint distribution is not Gaussian. It is an interesting future direction to consider such extensions.


\section{Outline of proofs}
\label{sec:proofs}

In the present section we outline the main ideas of our proof technique; detailed proofs of the various technical results are postponed to Appendices~\ref{app:tech}--\ref{app:sparse}. Sections~\ref{subsec:proofs:supp} and~\ref{subsec:proofs:dev} outline the proofs of Theorems~\ref{thm:main:supp} and~\ref{thm:main:dev}, respectively, which involve most of the technical heavy lifting. We then outline how to apply these two theorems to derive the results in Sections~\ref{subsec:app:scorebased}, \ref{subsec:app:causal}, and~\ref{subsec:app:ci} in Sections~\ref{subsec:proofs:scorebased}, \ref{subsec:proofs:causal}, and~\ref{subsec:proofs:ci}, respectively.

The proofs of Theorems~\ref{thm:main:supp} and~\ref{thm:main:dev} will be broken down into several steps. First, we establish some basic properties of the objective function and the probability space in order to reduce the neighbourhood regression analysis to a family of maximal sets denoted by $\nhbdmax_{j}(S)$ (Definition~\ref{defn:nhbdmax}). Then we introduce the notion of monotonicity (Lemma~\ref{lem:estim:supp:monot}) that is central to our proofs, and exploit this to provide a uniform bound on the probability of false selection for any neighbourhood problem (Proposition~\ref{prop:gen:reduct}). Then, we will derive an independent result that gives a deviation bound for a \emph{fixed} design regression problem (Theorem~\ref{thm:fixed:l1l2bound}). Using a similar monotonicity argument as in the first step, we then show that this result can be applied uniformly to each neighbourhood problem, which yields our claimed deviation bounds (see Proposition~\ref{prop:abstract:dev}).

\subsection{Support recovery}
\label{subsec:proofs:supp}

We begin by controlling the event 
\begin{align}
\label{eq:def:suppevent}
\suppevent
:= \{\supp(\dagadjest(\perm)) \ne \supp(\tB{\perm})\,\;\exists\,\perm\in\allperms\}.
\end{align}

\noindent
We will do this by reducing the analysis of $\dagadjest(\perm)$ to a series of neighbourhood regression problems. There are two key steps: (i) Showing that each estimator $\dagadjest(\perm)$ is equivalent to solving a series of $p$ regression problems given by $\abstractpls_{\lambda}(\sampledagmatelem_{j},\sampledagmat;\dagcand_{j}(\perm))$ (cf. Definition~\ref{defn:nhbdregression}), and (ii) Controlling the total number of sets $\dagcand$ that need to be considered.

\subsubsection{Reduction to neighbourhood regression}
\label{subsubsec:proofs:reduction}
Recall that the $j$th column of $\dagadjest(\perm)$ is denoted by $\dagcolest_{j}(\perm)$ and as usual, denote the sample version of $\dagerr_j(\perm)$ by boldface, i.e. $\sampledagerrcol_{j}(\perm):= \sampledagmatelem_j - \sampledagmat\dagcol_j(\perm)$. The first step above is justified by the following result. The symbol $\independent$ is used here to denote independence of random variables.

\begin{lemma}
\label{lem:restricted:nhbd}
Suppose $\sampledagmat\iid\normalN_{p}(0,\trueCov)$ and $\lambda\ge 0$. Then the following statements are true:
\begin{enumerate}\itemsep =1ex
\item[$(a)$] For any $j \in [p]$ and $\perm\in\allperms$, $\sampledagerrcol_{j}(\perm) \independent\sampledagmat_{\dagcand_{j}(\perm)}$.
\item[$(b)$] A matrix $\dagadjest(\perm)\in\dagspace$ is a restricted minimizer (Definition~\ref{defn:restrictedmin}) if and only if
$\dagcolest_{j}(\perm)\in \abstractpls_{\lambda}(\sampledagmatelem_{j},\sampledagmat;\dagcand_{j}(\perm))$ for each $j \in [p]$.
\item[$(c)$] $\dagadjest=\dagadjest(\estperm)$ is a global minimizer of $\plsobj(B)$ if and only if $\dagcolest_{j}(\estperm)\in\abstractpls_{\lambda}(\sampledagmatelem_{j},\sampledagmat;\dagcand_{j}(\estperm))$ for each $j \in [p]$ and $\estperm\in\estperms$.
\end{enumerate} 
\end{lemma}

\noindent
The proof of this lemma, which is a simple consequence of how the least-squares loss and the regularizer factor, is found in Appendix~\ref{app:lem:restricted:nhbd}. This allows us to  formally establish the equivalence between the DAG problem and neighbourhood regression: In order to construct $\dagadjest(\perm)$, it suffices to solve a neighbourhood regression problem for each column of $\dagadjest(\perm)$, given by $\abstractpls_{\lambda}(\sampledagmatelem_{j},\sampledagmat;\dagcand_{j}(\perm))$. A key observation is that through the independence established in Lemma~\ref{lem:restricted:nhbd}(a) and a conditioning argument, we can reduce the regression problem given by $\abstractpls_{\lambda}(\sampledagmatelem_{j},\sampledagmat;\dagcand_{j}(\perm))$ to a fixed design problem. The details are outlined in the proof of Proposition~\ref{prop:gen:reduct}.

\subsubsection{Invariant sets and monotonicity}
\label{subsubsec:proofs:monotonicity}
As a consequence of Lemma~\ref{lem:restricted:nhbd}, we have (cf. \eqref{eq:def:suppevent})
\begin{align*}
\suppevent
= \bigcup_{j=1}^{p}\{\supp(\dagcolest_{j}(\perm)) \ne \supp(\dagcol_{j}(\perm))\,\;\exists\,\perm\in\allperms\}.
\end{align*}

\noindent
Since there are $p!$ total permutations, in principle the event $\{\supp(\dagcolest_{j}(\perm)) \ne \supp(\dagcol_{j}(\perm))\,\exists\,\perm\in\allperms\}$ involves controlling a superexponential number of estimators, which seems hopeless. In order to reduce the total number of estimators we must control, we will introduce the notion of an \emph{invariant set}.

First, recall the definition of $\beta_j(S)$ (cf. Definition~\ref{defn:nhbd}) and for any $j\in[p]$ and $S\subset[d]_{j}$ define the error (or noise) for the associated neighbourhood regression as the following residual:
\begin{align*}
\nhbderr_j(S) 
&:= \dagvec_j - \beta_j(S)^{T}\dagvec.
\end{align*}

\noindent 
The support set of $\nhbdcoef_j(S)$ is denoted by $\nhbdsupp_j(S) := \supp(\beta_j(S))$ and the error variance by $\nhbdvar_{j}^{2}(S) := \var(\nhbderr_j(S))$.

\begin{defn}
\label{defn:invariant}
For any $S \subset [p]_j$, define a collection of subsets by
\begin{align*}
\nhbdset_j(S) 
:= \{T \subset[p]_j: \; \nhbdcoef_j(T) = \nhbdcoef_j(S) \} = \{T \subset[p]_j: \; \nhbdsupp_j(T) = \nhbdsupp_j(S) \},
\end{align*}

\noindent
where $\nhbdcoef_j(S)$ and $\nhbdsupp_j(S)$ are defined in Definition~\ref{defn:nhbd}. If $T\in \nhbdset_j(S)$, we call $T$ an \emph{invariant set of $S$ for $j$}, or \emph{$S$-invariant} for short.
\end{defn}

In other words, for any $j$, $\nhbdset_j(S)$ is the collection of candidate sets $T\subset[p]_{j}$ such that the projection of $X_{j}$ onto $\{X_i, \, i\in T\}$ is invariant. With some abuse of terminology, let us refer to $ \nhbdsupp_j(T) = \supp (\nhbdcoef_j(T)) $ as the \emph{support of neighbourhood $T$} (for node $j$). An equivalent description of $\nhbdset_j(S)$ is the set of neighbourhoods $T$ whose support (for node $j$) is the same and equals $m_j(S)$. 

\begin{example}
\label{ex:mainexample:cont}
Continuing with Example~\ref{ex:mainexample}, one can verify that \linebreak $\nhbdset_3(\{1,2\}) = \{\{1\},\{1,2\}\}$ and $m_3(\{1,2\}) = \{1\}$. Note also that $\nhbdset_3(S) = \nhbdset_3(\{1,2\})$ for any $S \in \nhbdset_3(\{1,2\})$. In the same example,
\begin{align*}
     \{\{1\},\{1,2\}\} = \nhbdset_3(\{1,2\}) \neq \nhbdset_3(\{1,4\}) = \{\{1,4\},\{1,2,4\}\}.
\end{align*}
\end{example}

The following lemma illustrates a crucial property of invariant sets:

\begin{lemma}
\label{lem:lattice}
$T_1,T_2 \in \nhbdset_{j}(S) \implies T_1 \cup T_2 \in \nhbdset_{j}(S)$. 
\end{lemma}

\noindent
In other words, if two neighbourhoods share the same support, the union of these neighbourhoods must also have the same support. This justifies the following definition:

\begin{defn}
\label{defn:nhbdmax}
The unique largest element of $\nhbdset_{j}(S)$ shall be denoted by $\nhbdmax_j(S)$. Formally,
\begin{align*}
\nhbdmax_j(S) 
:= \bigcup \nhbdset_{j}(S)
= \bigcup \{ T \subset [p]_j:\; \beta_j(T) = \beta_j(S)\}.
\end{align*}
\end{defn}

\noindent
For instance, in Example~\ref{ex:mainexample:cont} we have $\nhbdmax_3(\{1\}) = \{1,2\}$. The name ``$S$-invariant set'' comes from the fact that for any $T\in\nhbdset_{j}(S)$, we have the following useful identities:
\begin{align}
\label{eq:invariant:nhbdcoef}
\nhbdcoef_{j}(\nhbdsupp_{j}(S)) = \nhbdcoef_{j}(S) &= \nhbdcoef_{j}(T) = \nhbdcoef_{j}(\nhbdmax_{j}(S)), \\
\label{eq:invariant:nhbderr}
\nhbderr_{j}(\nhbdsupp_{j}(S)) = \nhbderr_{j}(S) &= \nhbderr_{j}(T) = \nhbderr_{j}(\nhbdmax_{j}(S)).
\end{align}

The reason for introducing invariant sets is that it is generally \emph{sufficient} to study the neighbourhood problem for $\nhbdmax_j(S)$ in the sense that once we have model selection consistency for each estimator in $\abstractpls_\lambda(\sampledagmatelem_j,\sampledagmat; \nhbdmax_j(S))$, the same is  \emph{guaranteed} for estimators based on every other neighbourhood in $\nhbdset_j(S)$. In fact, we have the following result, which says that model selection properties of the $S$-restricted estimators are monotone with respect to those sets $S$ that contain the true support.

\begin{lemma}
\label{lem:estim:supp:monot}
Suppose that $\fixeddesignmat\in\R^{n\times m}$ is fixed and consider the regression problem $\sampley = \fixeddesignmat \lincoef + \samplelinerr$ for some $\lincoef\in\R^m$. If $\supp(\lincoef) \subset S \subset U$, then we have the following inclusion: $\mserrset(\fixeddesignmat,\lincoef;S) \subset \mserrset(\fixeddesignmat,\lincoef;U).$ In particular,
\begin{align*}
\msevent(\samplelinerr,\fixeddesignmat,\lincoef;S)
&\subset \msevent(\samplelinerr,\fixeddesignmat,\lincoef;U),
\end{align*}

\noindent
where $\mserrset(\fixeddesignmat,\lincoef;S)$ and $\msevent(\samplelinerr,\fixeddesignmat,\lincoef;S)$ are defined in \eqref{eq:def:mserrset}--\eqref{eq:def:msevent}.
\end{lemma}

Intuitively, for a fixed support, the set of ``bad'' noise vectors for the larger problem involving $U$ is at least as big as the set of ``bad'' noise vectors for the smaller problem involving $S$. We are interested in the model selection failure of $\hnhbdcoef_j(S)$ for $\nhbdcoef_j(S)$, which can be stated as
\begin{align}
\Big\{ \supp(\hnhbdcoef_j(S)) &\neq \supp(\nhbdcoef_j(S)) \notag\\
&\text{ for some } \hnhbdcoef_j(S) \in \abstractpls_\lambda(\sampledagmatelem_j,\sampledagmat; S)\Big\}
 = \msevent(\sampledagerrcol_{j}(S), \sampledagmat, \nhbdcoef_j(S); S) \label{eq:msevet:356} 
\end{align}
in the notation introduced in~\eqref{eq:def:msevent}. The next result encapsulates a notion of \emph{monotonicity} that is used throughout the rest of the proof, establishing that the invariant sets for any fixed neighbourhood are monotone in the same sense as Lemma~\ref{lem:estim:supp:monot}.

\begin{cor}
\label{cor:ms:monotone}
Suppose $\sampledagmat\iid\normalN_{p}(0,\trueCov)$. For any $S \subset [p]_j$, we have
\begin{align*}
\msevent\Big(\sampledagerrcol_{j}(S), \sampledagmat, \nhbdcoef_j(S); S \Big)
\subset\msevent\Big(\sampledagerrcol_{j}(\nhbdmax_{j}(S)), \sampledagmat, \nhbdcoef_j(\nhbdmax_{j}(S));\nhbdmax_j(S) \Big).
\end{align*}
\end{cor}

\noindent
This proves that in order to control the neighbourhood regression problem for some set $S\subset[p]_{j}$, it suffices to control the strictly harder problem given by $\nhbdmax_{j}(S)$. Note that Lemma~\ref{cor:ms:monotone} is a \emph{deterministic} statement about the events defined in \eqref{eq:msevet:356}.

\subsubsection{A bound on false selection}
\label{subsubsec:proofs:bound}

For any $\trueCov\posdef 0$ and fixed node $\dagvec_{j}$, define the following collections of subsets:
\begin{align}
\label{eq:def:activesets}
\activesets_j(\trueCov) 
&:= \{ \nhbdsupp_j(S) :\; S\subset[p]_{j}\}, \\
\label{eq:def:maxsets}
\maxsets_j(\trueCov) 
&:= \{ \nhbdmax_j(S) :\; S\subset[p]_{j}\}.
\end{align} 

\noindent 
Note that $|\activesets_j(\trueCov)| = |\maxsets_j(\trueCov)|$. As long as it is clear whether the argument is a set $\dagcand$ or a matrix $\trueCov$, this should not cause any confusion with $\nhbdsupp_{j}(\dagcand)$ and $\nhbdmax_{j}(\dagcand)$. In Example~\ref{ex:mainexample}, $\activesets_3(\trueCov) = \{\emptyset, \{2\}, \{4\}, \{1,2\}, \{2,4\}, \{1,2,4\}\}$. In contrast, one can verify that $\activesets_2(\trueCov) = \{\emptyset, \{3\}, \{4\}, \{3,4\}, \{1,3,4\}\}$ which illustrates how some nodes have a reduced set of parent sets.

For any neighbourhood $S\subset[p]_{j}$, recall that the associated error variance is given by $\nhbdvar_{j}^{2}(S) = \var(\nhbderr_j(S))$. With some more abuse of notation, let
\begin{align}
\label{eq:def:nhbdmsexp}
\nhbdmsexp_{j}(S)
:= \msexp_{\lambda}(\sampledagmat_{S},\, (\nhbdcoef_{j}(S))_{S},\, \nhbdvar_{j}^{2}(S)).
\end{align}

\noindent
Note that we must restrict the SEM coefficients $\nhbdcoef_{j}(S)$ to the subset $S$ in order for this exponent to be well-defined. Since $\supp(\nhbdcoef_{j}(S))\subset S$, this does not change anything. 

The following general result gives a uniform upper bound on the probability of false selection for any neighbourhood problem in terms of the maximal sets $\nhbdmax_{j}(T)$, and is the main ingredient in proving Theorem~\ref{thm:main:supp}.

\begin{prop}
\label{prop:gen:reduct}
Fix $j\in[p]$. Under Condition~\ref{condn:evals}, we have
\begin{align*}
\pr \big(  \supp(\hnhbdcoef_j(S)) \neq \supp(\nhbdcoef_j(S)),\; \exists\,S\subset[p]_{j} \big)\;\;
\le\;\; \sum_{T \in \nhbdsupp_{j}(\trueCov)} \; \E e^{-\nhbdmsexp_j(\nhbdmax_{j}(T))},
\end{align*}

\noindent
where $\activesets_j(\trueCov)$ is defined by \eqref{eq:def:activesets} and $\nhbdmsexp_j(\,\cdot\,)$ is defined by \eqref{eq:def:nhbdmsexp}.
\end{prop}

\noindent
The proof of this result can be found in Appendix~\ref{app:prop:gen:reduct}.

Proposition~\ref{prop:gen:reduct} says that to control the probability of false selection uniformly for all $2^{p-1}$ neighbourhoods $S$ of the node $j$, it suffices to control a much smaller class of problems given by the neighbhourhoods $\nhbdmax_{j}(T)$ for each support set $T\in\nhbdsupp_{j}(\trueCov)$. By Definition~\ref{defn:eqclassparam}, $|\nhbdsupp_{j}(S)|\le d$ for all $j$ and $S$, which implies that $|\nhbdsupp_{j}(\trueCov)|\le \binom{p}{d} \le p^{d}$. Since there are $2^{p-1}$ subsets of $[p]_{j}$, $|\activesets_j(\trueCov)|\le p^{d}\ll 2^{p-1}$ as long as $d\ll p$; i.e. as long as $d$ is much smaller than $p$ the cardinality of $\activesets_j(\trueCov)$ is much smaller than that of $2^{[p]_{j}}$. All of a sudden, instead of a superexponential number of neighbourhood problems to control, we have a subexponential (i.e. polynomial) number to control.

\subsubsection{Proof of Theorem~\ref{thm:main:supp}}
\label{subsubsec:proofs:mainproof}

For any $T \in \nhbdsupp_{j}(\trueCov)$, Lemma~\ref{lem:estim:supp:monot} applied with $S = \nhbdmax_{j}(T)$ and $U = [p]$ yields
\begin{align*}
\nhbdmsexp_{j}(\nhbdmax_{j}(T))
&\ge \msexp_{\lambda}(\sampledagmat,\, \nhbdcoef_{j}(T),\, \nhbdvar_{j}^{2}(T)).
\end{align*}

\noindent
Recalling $d(\trueCov)$ and $\betamin(\trueCov)$ in Definition~\ref{defn:eqclassparam}, we have $\norm{\nhbdcoef_{j}(T)}_{0} \le d(\trueCov)$ and $\betamin(\nhbdcoef_{j}(T))\ge\betamin(\trueCov)$, as well as $\nhbdvar_{j}^{2}(T)\le\maxvar^2$. The previous expression combined with \eqref{eq:def:unifmsexp} implies:
\begin{align}
\label{eq:thm:main:supp:proof:1}
\nhbdmsexp_{j}(\nhbdmax_{j}(T))
\ge \unifmsexp_{\lambda}(\sampledagmat,\trueCov)
\quad\text{for all $T\in\maxsets_{j}(\trueCov)$.}
\end{align}

Combining Proposition~\ref{prop:gen:reduct}, \eqref{eq:thm:main:supp:proof:1} and a union bound over $j \in [p]$,
\begin{align}
\nonumber
\pr \big(  \supp(\hnhbdcoef_j(S)) &\neq \supp(\nhbdcoef_j(S)),\; \exists\,j\in[p],\,S \subset[p]_{j} \big) \\
\nonumber
&\le\;\; \sum_{j=1}^{p}\sum_{T \in \nhbdsupp_j(\trueCov)} \; \E \exp(-\nhbdmsexp_{j}(\nhbdmax_{j}(T)) \\
\label{eq:thm:msc:1}
&\le\;\; p\binom{p}{d} \; \E \exp(-\unifmsexp_{\lambda}(\sampledagmat,\trueCov)),
\end{align}

\noindent
since there are at most $\binom{p}{d}$ subsets in $\nhbdsupp_{j}(\trueCov)$. In order to complete the proof, note that Lemma~\ref{lem:restricted:nhbd} implies
\begin{align*}
\Big\{
\supp(\dagadjest(\perm)) &\neq \supp(\tB{\perm}),\; \exists\,\pi\in\allperms\Big\} \\
&= \Big\{
\supp(\dagcolest_{j}(\perm)) \neq \supp(\dagcol_{j}(\perm)),\; \exists\,\pi\in\allperms,\, j\in[p]
\Big\} \\
&= \Big\{
\supp(\dagcolest_{j}(\dagcand)) \neq \supp(\nhbdcoef_{j}(\dagcand)),\; \exists\,\dagcand\subset[p]_{j},\, j\in[p]
\Big\}.
\end{align*}

\noindent
In the last line we used $\dagcol_{j}(\perm) = \nhbdcoef_{j}(\dagcand_{j}(\pi))$. Combined with \eqref{eq:thm:msc:1}, this gives the desired result.

\subsection{Deviation bounds}
\label{subsec:proofs:dev}

In order to establish uniform control over the probability of false selection for all possible neighbourhood regression problems in the previous section, we relied on a monotonicity property (cf. Lemma~\ref{lem:estim:supp:monot}) of model selection. Unfortunately, this may not hold for weaker properties such as deviation bounds. In order to bound the deviations $\norm{\dagadjest(\perm)-\tB{\perm}}_{r}$ ($r=1,2$), we will use a modified argument that invokes a different kind of monotone class.

\subsubsection{An upper bound for fixed design}
We start by establishing a general bound on the $\ell_{r}$ ($r=1,2$) estimation errors for a fixed design regression problem with a general regularizer $\pl$. The objective here is to derive conditions under which we can guarantee such bounds for a fixed design problem, and then show that  these conditions hold uniformly for all neighbourhood problems. The conditions we will need are familiar from the literature: A \emph{Gaussian width condition} and a \emph{restricted eigenvalue condition}.

For the rest of this subsection, we let $\fixeddesignmat\in\R^{n\times m}$ and $\fixedlinerr\in\R^{n}$ be a fixed matrix and fixed vector, respectively.

\begin{defn}[Gaussian width]
\label{defn:gwcond}
We say that the \emph{Gaussian width (GW) condition} holds for $(\fixedlinerr,\fixeddesignmat)$ relative to $\pl$ if there is a numerical constant $\gwconst\in(0,1)$ such that 
\begin{align*}
\frac{1}{n}| \ip{\fixedlinerr,\fixeddesignmat \genu} |
\le \gwconst\bigg[\frac{1}{2n} \norm{\fixeddesignmat \genu}_{2}^{2} 
+ \pl(\genu)\bigg], \; \forall \genu\in\R^{m},
\end{align*}

\noindent
in which case we write $(\fixedlinerr,\fixeddesignmat)\in\GW_{\pl}(\gwconst)$. If this inequality is strict for all $u \neq 0$, we write $(\fixedlinerr,\fixeddesignmat)\in\GW_{\pl}^\circ(\gwconst)$.
\end{defn}

We will be interested in the case where both $\fixedlinerr$ and $\fixeddesignmat$ are allowed to be random but independent. In this setting, for Gaussian designs considered in this paper, the GW condition holds with high probability for the $\ell_{1}$ penalty (this follows from a standard H\"older inequality argument), and has similarly been shown to hold for penalties induced by $\ell_{q}$ norms for $0\le q\le 1$ \citep{raskutti2011}. \citet{zhang2012} provide a version of this condition that applies to general nonconvex regularizers.

Before we proceed, let us note the following key relation between model selection consistency and the GW condition:
\begin{lemma}
\label{lem:epms:equiv}
 Consider the setup of Lemma~\ref{lem:estim:supp:monot}, namely, the regression problem $\sampley = \fixeddesignmat \lincoef + \samplelinerr$ but with $\lincoef = 0$. Then
    \begin{align*}
    \msevent(\samplelinerr/\delta, \,Z, \,0)^c
    &= \big\{ (\samplelinerr,Z) \in \GW_{\pl}^\circ(\gwconst) \big\}.
    \end{align*}
\end{lemma}
Thus, in order to ensure the GW condition for $(\samplelinerr,Z_S)$, it suffices to show that the corresponding regression problem is model selection consistent when the \emph{true coefficients are all set to zero} and the noise variance is inflated by a factor of $1/\gwconst^{2}$. \cite{zhang2012} refer to this property as \emph{null-consistency}.

\smallskip
For any set $A\subset[m]$ and $\xi>0$, define the following ``cone'':
\begin{align}
\label{eq:def:rescone}
\cone_{\pl}(A,\xi)
&:=\{\genu\in\R^m : \pl(\genu_{A^c}) \le \xi\pl(\genu_A)\}.
\end{align}

\noindent
This definition also depends on the ambient dimension $m$; when we wish to emphasize this we will write $\cone^{m}_{\pl}(A,\xi)$. The term ``cone'' here is used in an extended sense, in analogy with the $\ell_{1}$ cone found in previous work.

\begin{defn}[Generalized restricted eigenvalue]
\label{defn:genrec}
The \emph{generalized restricted eigenvalue (RE) constant} of $\fixeddesignmat$ with respect to $\pl$ over a subset $A$ is 
\begin{align}
\label{eq:defn:genrec}
\rec^{2}_{\pl}(\fixeddesignmat, A;\,\xi)
&:=\inf\left\{\frac{\norm{\fixeddesignmat \genu}_2^2}{n\norm{\genu}_2^2} : \genu\in \cone_{\pl}(A,\xi),\,u\ne 0\right\}.
\end{align}
\end{defn}

\noindent
In the sequel, we often suppress the dependence of the generalized RE constants on $\lambda$ and $\xi$, writing $\rec^{2}_{\pen}(\fixeddesignmat, A)=\rec^{2}_{\pl}(\fixeddesignmat, A;\xi)$. Note that the usual restricted eigenvalue is  equivalent to the special case $\pl=\lambda\norm{\cdot}_{1}$ \citep{bickel2009}. 

\medskip
Consider the usual linear regression set up, $\fixedy = \fixeddesignmat\lincoef + \fixedlinerr$,
where $\lincoef\in\R^{m}$ and we define $S^{*} = \supp(\lincoef)$. The following general result establishes that the two conditions $(\fixedlinerr,\fixeddesignmat)\in\GW_{\pen}(\gwconst)$ and $\rec_{\pen}^{2}(\fixeddesignmat,S^{*})>0$ are sufficient to bound the deviation $\hlincoef-\lincoef$:

\begin{thm}
\label{thm:fixed:l1l2bound}
Assume $(\fixedlinerr,\fixeddesignmat)\in\GW_{\pl}(\gwconst)$ for some $\pl$ satisfying Condition~\ref{condn:pl:basic} and $\gwconst\in(0,1)$. Let $\xi=\xi(\gwconst):=(1+\gwconst)/(1-\gwconst)$ and assume $\rec^{2} := \rec_{\pen}^2(\fixeddesignmat, S^{*};\,\xi) > 0$. Then any $\hlincoef\in\abstractpls_\lambda(\fixeddesignmat\lincoef+\fixedlinerr, \fixeddesignmat)$ satisfies
\begin{align}
\label{eq:L2bound}
\norm{\hlincoef - \lincoef}_{2}
&\le \frac{2\,\xi}{\rec^{2}}\,\penderiv\norm{\lincoef}_{0}^{1/2}, \\
\norm{\hlincoef-\lincoef}_{1}
\label{eq:L1bound}
&\le \frac{2\,\xi(1+\xi)}{\rec^{2}}\,\penderiv \norm{\lincoef}_{0}.
\end{align}
\end{thm}

\noindent
The proof of Theorem~\ref{thm:fixed:l1l2bound} is found in Appendix~\ref{app:thm:fixed:l1l2bound}. 
The GW condition is quantified by the constant $\gwconst\in(0,1)$, and the restricted eigenvalue condition depends on the free parameter $\xi>0$; these two are linked via the relation $\xi(\gwconst) = (1+\delta)/(1-\delta)$ and play subtle roles in the proof. A slightly modified version of this result first appeared in \citet{zhang2012}, under different assumptions. The particular version presented here is important to derive uniform bounds for all permutations---see Section~\ref{sec:unif:dev:bounds}.

\subsubsection{Uniform deviation bounds}
\label{sec:unif:dev:bounds}

Our strategy from here will be to show that the two sufficient conditions in Theorem~\ref{thm:fixed:l1l2bound}---namely, the GW condition and the generalized RE condition---hold uniformly for all $S$ with $w=\sampledagerrcol_{j}(S)$ and $Z=\select{\sampledagmat}{S}$.
First, we provide a uniform bound on the restricted eigenvalues $\rec_{\pen}^2(\sampledagmat_{\dagcand}, \nhbdsupp_{j}(\dagcand))$ in terms of the smallest eigenvalue of $\trueCov$. More precisely, we show in Proposition~\ref{prop:restricted:recprob} that with high probability $\rec_{\pen}^{2}(\sampledagmat_{S},A) \ge \mineval(\trueCov)$ for all pairs $A \subset S$ with $|A| \le \maxdeg$.

The next step is to show that with high probability, the GW condition holds for $(\sampledagerrcol_{j}(\dagcand),\sampledagmat_{\dagcand})$ for all $j$ and $S$. First, let us define
 \begin{align}
 \label{eq:Ejpi}
 \epevent_{S}(\gwconst,\lambda) =\epevent_{S}(\gwconst,\lambda; j)
 := \Big\{ \big(\sampledagerrcol_{j}(S), \select{\sampledagmat}{S}  \big) \in \GW_{\pl}^\circ(\gwconst)
 \Big\}.
 \end{align}

\noindent
According to Lemma~\ref{lem:epms:equiv}, we have $\epevent_{S}(\gwconst,\lambda) = \msevent\big(\sampledagerrcol_{j}(S)/\gwconst,\sampledagmat, 0; S \big)^{c}$. An immediate consequence is 
that the complements of the events $\epevent_{S}(\gwconst,\lambda)$ are monotonic in the sense that they obey Corollary~\ref{cor:ms:monotone} when $\nhbdcoef_{j}(S)$ is replaced with the zero vector. Analogous to the arguments leading up to Proposition~\ref{prop:gen:reduct} and its proof, this allows us to provide uniform control on $\epevent_{S}(\gwconst,\lambda)$, in a sense made precise by Proposition~\ref{prop:gwcond:probbound}. Once again, the key is to reduce the total number of events to control---\emph{a priori} superexponential in size---down to a tractable number which can be controlled.

Together, Propositions~\ref{prop:gwcond:probbound} and~\ref{prop:restricted:recprob} show that we have uniform control over both the restricted eigenvalues and the Gaussian widths for the neighbourhood problems $\nhbdproblem$. Thus, we can apply Theorem~\ref{thm:fixed:l1l2bound} to each of these problems and obtain, with high probability, deviation bounds of the form
\begin{align*}
\norm{\dagcolest_{j}(\dagcand)-\nhbdcoef_{j}(\dagcand)}_{2}
&\le \frac{2\,\xi}{\mineval(\trueCov)}\penderiv\cdot{\norm{\nhbdcoef_{j}(\dagcand)}_{0}^{1/2}}
\end{align*}

\noindent
for all $j$ and $S\subset[p]_j$. The precise statement can be found in Proposition~\ref{prop:abstract:dev}. Theorem~\ref{thm:main:dev} is an immediate consequence of these two results, the details of which are discussed in the next section.

\subsubsection{Proof of Theorem~\ref{thm:main:dev}}
\label{app:thm:main:dev}

Let us consider the $\ell_{2}$ bound, noting that the $\ell_{1}$ version follows similarly:
\begin{align}
\label{eq:prop:abstract:dev:L2}
\norm{\dagadjest-\tB{\estperm}}_{2}
&\le \frac{2\xi}{\mineval(\trueCov)}\penderiv\norm{\tB{\estperm}}_{0}^{1/2}.
\end{align}
\eqref{eq:prop:abstract:dev:L2} can be seen by applying Proposition~\ref{prop:abstract:dev} with $\dagcand = \dagcand_j(\estperm)$, and noting that
$\dagcolest_j(\estperm) = \dagcolest_j(\dagcand_j(\estperm))$, $\dagcol_j(\estperm) = \widetilde\beta_j(S_j(\estperm))$, and the fact that $\norm{\genu_j - \genv_j}_2 \le a \norm{\genv_j}_0^{1/2}$ for all $j$, implies $\sum_j \norm{\genu_j - \genv_j}_2^2 \le a^2 \sum_j \norm{\genv_j}_0$. The bound on the probability of~\eqref{eq:prop:abstract:dev:L2} follows by combining  Proposition~\ref{prop:abstract:dev} with Proposition~\ref{prop:GW:capped:ell1}.

\subsection{Score-based learning}
\label{subsec:proofs:scorebased}

We now show how Theorems~\ref{thm:main:supp} and~\ref{thm:main:dev} can be applied to derive the results in Section~\ref{subsec:app:scorebased}. Theorem~\ref{thm:scorebased} is an immediate corollary of Theorems~\ref{thm:main:supp} and~\ref{thm:main:dev} applied to the special case $\perm=\estperm$. Next, we control the sparsity of the estimate $\dagadjest$ and the candidate DAG $\tB{\estperm}$ via Theorems~\ref{thm:main:sparse} and~\ref{thm:oracle}.

We start with a kind of basic inequality that is adapted to the present, nonidentifiable setting:
\begin{lemma}
\label{lem:prelimbound}
Let $\bErr{\perm} := \sampledagmat - \sampledagmat\tB{\pi}$. For any $\perm\in\allperms$ and $\estperm\in\estperms$,
\begin{align}
\label{eq:basic:ineq}
\begin{split}
\frac{1}{2n}\norm{\sampledagmat(\tB{\estperm} - \dagadjest)}_{\frob}^2 + \pl(\dagadjest)
\;&\le\; \frac{1}{2n}\norm{\bErr{\perm}}_{\frob}^2 
- \frac{1}{2n}\norm{\bErr{\estperm}}_{\frob}^2 \\
&\qquad+ \frac{1}{n}\tr\left(\bErr{\estperm}^T\sampledagmat(\tB{\estperm} - \dagadjest)\right) \\
&\qquad+ \pl(\tB{\perm}).   
\end{split}
\end{align}
\end{lemma}

\noindent
In contrast to the basic inequality used in usual regression,~\eqref{eq:basic:ineq} also leverages the many possible decompositions of $\sampledagmat$, one for each permutation. Three terms in this inequality are particularly important:
\begin{enumerate}\setlength\itemsep{5pt}
\item The difference in residuals $\norm{\bErr{\perm}}_{\frob}^2/(2n) - \norm{\bErr{\estperm}}_{\frob}^2/(2n)$ explains the origin of the minimum-trace permutation: We would like to make $\norm{\bErr{\perm}}_{\frob}^2/(2n)$ as small as possible in order to minimize this difference. By standard concentration arguments, $\norm{\bErr{\perm}}_{\frob}^2/n$ is close to its expectation, $\tr\tOm{\perm}$ (cf.~\eqref{eq:exploss}). Hence, we choose $\pi$ to minimize $\tr \tOm{\perm}$. The details of this argument are in Appendix~\ref{app:resbound}; the explicit upper bound we use is detailed in Proposition~\ref{prop:mintrace:prob}.

\item \sloppypar{The quantity $\tr(\bErr{\estperm}^T\sampledagmat(\tB{\estperm} - \dagadjest))/n$ can be bounded using the Gaussian width condition (Definition~\ref{defn:gwcond}). There is a subtlety regarding whether to decompose this along rows or columns; see Lemma~\ref{lem:trace}.}

\item The penalty on $\dagadjest$ can be replaced with $\pl(\tB{\estperm})$ by showing that $\pl(\dagadjest)\ges\pl(\tB{\estperm})$ (Lemma~\ref{lem:hBlowerbound}). This is true whenever the deviation $\dagadjest-\tB{\estperm}$ is not too large, which is guaranteed by Theorem~\ref{thm:main:dev}. 
\end{enumerate}

Once we have established control of these three terms, the details of which are found in Appendix~\ref{app:sparse}, it is not hard to prove Theorems~\ref{thm:main:sparse} and~\ref{thm:oracle}.
More precisely, for Theorem~\ref{thm:main:sparse}, we have the following bound in terms of the constants $\gwconst$ (cf. Definition~\ref{defn:gwcond}) and $\mintraceconst$ (cf. Condition~\ref{condn:mintrace:bound}):
\begin{align*}
\frac{2\gwconst}{1-\gwconst}\pl(\tB{\estperm})
\le \pl(\dagadjest)
\le \frac{2}{1-\gwconst}\Big(1+\frac{10}{\mintraceconst}\Big)\pl(\tB{\goodperm}).
\end{align*}

\noindent
The precise statement and proof are given by Proposition~\ref{prop:abstract:sparse}, and Theorem~\ref{thm:main:sparse} follows as a consequence of this bound and an $\ell_{0}$-compatibility argument (see Appendix~\ref{app:thm:main:sparse}). The oracle inequality Theorem~\ref{thm:oracle} follows from a straightforward manipulation of \eqref{eq:basic:ineq} combined with a Gaussian concentration argument (Lemma~\ref{lem:pi:tracebound}), see Appendix~\ref{app:thm:oracle} for details.

\subsection{Causal DAG learning}
\label{subsec:proofs:causal}

The main result in Section~\ref{subsec:app:causal} is Theorem~\ref{thm:mintr}, whose proof is similar to the proofs in the previous section and follows mostly from Lemma~\ref{lem:prelimbound}. The complete proof is in Appendix~\ref{app:thm:mintr}.

\subsection{Conditional independence learning}
\label{subsec:proofs:ci}

Theorem~\ref{thm:ci} from Section~\ref{subsec:app:ci} follows almost immediately from Theorem~\ref{thm:main:supp}. For completeness, we include a proof below.

\subsubsection{Proof of Theorem~\ref{thm:ci}}
Recall that $\tB{\perm}$ is always a minimal I-map for $\normalN_p(0,\trueCov)$ and $\bigcup_{\perm\, \in\, \allperms}\ci(\perm) = \ci(\trueCov)$ \citep[e.g.][]{koller2009}. Theorem~\ref{thm:main:supp}, combined with Lemma~\ref{lem:mcp:example}, implies that 
\begin{align*}
\pr \Bigg( \bigcap_{\perm\in\allperms}\big\{\supp(\dagadjest(\perm)) = \supp(\tB{\perm})\big\} \Bigg)
\ge 1-\probbound.
\end{align*}

\noindent
But $\supp(\dagadjest(\perm)) = \supp(\tB{\perm})$ implies $\widehat{\ci}(\perm)=\ci(\perm)$ (i.e. $d$-separation is a property of the graph alone and is independent of the exact values of the parameters). Since $\tB{\perm}$ is a minimal I-map for each $\perm$, $\dagadjest(\perm)$ must also be a minimal I-map. Furthermore, we deduce that
\begin{align*}
\pr \Bigg( \bigcup_{\perm\in\allperms}\widehat{\ci}(\perm) = \bigcup_{\perm\in\allperms}\ci(\perm) \Bigg)
&\ge \pr \Bigg( \bigcap_{\perm\in\allperms}\big\{\widehat{\ci}(\perm)=\ci(\perm)\big\} \Bigg) \\
&\ge \pr \Bigg( \bigcap_{\perm\in\allperms}\big\{\supp(\dagadjest(\perm)) = \supp(\tB{\perm})\big\} \Bigg) \\
&\ge 1-\probbound.
\end{align*}

\noindent
Using $\bigcup_{\perm\, \in\, \allperms}\ci(\perm) = \ci(\trueCov)$ completes the proof.

\appendix

\section{Proofs of technical results}
\label{app:tech}

\subsection{Proof of Lemma~\ref{lem:eqclass}}
\label{app:lem:eqclass}

We need the following simple lemma, which follows since $P_\pi A = P A P^T$ for some permutation matrix $P$:
\begin{lemma}
\label{lem:cholperm}
 $A= MNM^T \iff \permmat_\pi A=(\permmat_\pi M)(\permmat_\pi N) (\permmat_\pi M)^T$.
\end{lemma}

Recall the modified Cholesky decomposition of $A$ (also called the LDLT decomposition): $A=LDL^T$ for a lower triangular matrix $L$, {with unit diagonal entries}, and a diagonal matrix $D$. When $A$ is positive definite, the pair $(L,D)$ is unique and we refer to it as the \emph{Cholesky decomposition of $A$}.

\smallskip
Let us denote the set of all pairs $(\dagnopi,\varnopi)$ satisfying $\Sigma^{-1} =  (I - \dagnopi) \varnopi^{-1} (I-\dagnopi^T)$ (equivalently, \eqref{eq:def:sigmafcn}) as $\eqclass'$. Next, note that  $\dagnopi \in \dagspace$ if and only if $P_\pi \dagnopi$ is lower triangular for some permutation $\pi$. Lemma~\ref{lem:cholperm} implies that $(\dagnopi,\varnopi) \in \eqclass'$ iff $(I-P_\pi \dagnopi, P_\pi\varnopi^{-1})$ is a Cholesky decomposition of $P_\pi \Sigma^{-1}$ for some $\pi$. 

By the definition~\eqref{eq:def:tB}, $(I-P_\pi\tB{\perm}, P_\pi\tOm{\perm}^{-1})$ is also a  Cholesky decomposition of $P_\pi \Sigma^{-1}$. Since the Cholesky decomposition is unique for positive definite matrices, we have $(\dagnopi,\varnopi) \in \eqclass'$ iff $(\dagnopi,\varnopi) = (\tB{\perm}, \tOm{\perm})$ for some $\pi$, which gives the desired result, since $\eqclass(\Sigma)$ is the projection of $\eqclass'$ onto its first coordinate. 

\subsection{Proof of Lemma~\ref{lem:equalvar}}
\label{app:lem:equalvar}

Consider the following program:
\begin{align}
\label{eq:lem:equalvar:program}
\min\sum_{j=1}^{p} x_j^2
\,\text{ subject to }\,
\sum_{j=1}^{p} \log x_j^2 = C.
\end{align}

\noindent
The solution to this program is given by $x_j^2=e^{C/p}$ for all $j=1,\ldots,p$. In other words, the minimum is attained by a constant vector. It is straightforward to verify that $\log\det\tOm{\pi}=\log\det\trueCov$ and hence $\log\det\tOm{\pi}=\sum_j \log\dagerrvar_{j}^{2}(\pi)$ is constant for all $\pi\in\allperms$. Thus for any $\perm\in\allperms$, the vector $(\dagerrvar_1^2(\pi),\ldots,\dagerrvar_p^2(\pi))\in\R^p$ is feasible for \eqref{eq:lem:equalvar:program}, which implies that $\tr\tOm{\pi}$ is minimized whenever $\dagerrvar_1^2(\pi)=\cdots=\dagerrvar_p^2(\pi)$. Finally, uniqueness of $\tB{\goodperm}$ follows from Theorem~1 in \citet{peters2013}.

\subsection{Proof of Lemma~\ref{lem:mcp:example}}
\label{app:lem:mcp:example}

This is a consequence of Theorem~4.2 in \cite{huang2012} and Proposition~2 in \cite{zhang2008}. Briefly, \cite{huang2012} show that the least-squares MCP estimator correctly recovers the support of a linear model as long as the so-called \emph{sparse Riesz condition} holds. We then use \cite{zhang2008} to bound the probability that $\sampledagmat$ satisfies this condition. For the special case $\nhbdcoef_j(\dagcand)=0$ \citep[which is not covered by][]{huang2012} we can invoke Proposition~\ref{prop:gwcond:probbound}.

\subsection{Proof of Lemma~\ref{lem:restricted:nhbd}}
\label{app:lem:restricted:nhbd}

The first conclusion (a) follows from elementary properties of conditional expectation and the identity
\begin{align*}
\E(\dagvec_{j}\|\dagvec_{S_{j}(\pi)})
=\dagcol_{j}(\pi)^{T}\dagvec.
\end{align*}

\noindent
Since (c) is a special case of (b), it suffices to prove (b). 

Fix $\pi\in\allperms$ and let $S_{j}=\dagcand_{j}(\pi)$. If $\dagcolest_{j}(\pi)\in\abstractpls_{\lambda}(\sampledagmatelem_{j},\sampledagmat;S_{j})$ for each $j$, then evidently $\dagadjest(\pi) = [\,\dagcolest_{1}(\pi)\|\cdots\|\dagcolest_{p}(\pi)\,]$ minimizes $Q(\gendag)$ over $\subdagspace{\perm}$. For the reverse direction, recall that $\select{\sampledagmat}{S_{j}}$ is the $n\times |S_{j}|$ matrix formed by extracting the columns in $S_{j}$, and similarly for $(\beta_{j})_{S_{j}}$. For any $B\in\subdagspace{\pi}$ we have $(\beta_{j})_{S_{j}^{c}} = 0$ for each $j$, so we can write
\begin{align*}
\frac{1}{2n}\norm{\sampledagmat-\sampledagmat B}_{\frob}^{2} + \pl(B)
&= \sum_{j=1}^{p}\Big\{\frac{1}{2n}\norm{\sampledagmatelem_{j}-\sampledagmat\beta_{j}}_{2}^{2} + \pl(\beta_{j})\Big\} \\
&= \sum_{j=1}^{p}\Big\{\frac{1}{2n}\norm{\sampledagmatelem_{j}-\select{\sampledagmat}{S_{j}}(\beta_{j})_{S_{j}}}_{2}^{2} + \pl((\beta_{j})_{S_{j}})\Big\}.
\end{align*}

\noindent
Then $\dagadjest(\pi) = [\,\dagcolest_{1}(\pi)\|\cdots\|\dagcolest_{p}(\pi)\,]\in \min_{\subdagspace{\perm}}\plsobj(\gendag)$ if and only if
\begin{align*}
\dagcolest_{j}(\pi)
\in \argmin_\beta \frac{1}{2n}\norm{\sampledagmatelem_{j}-\sampledagmat\beta}_{2}^{2} + \pl(\beta)
\quad\text{subject to}\,\,\beta_{S_j^c}=0.
\end{align*}

\noindent
In other words, $\dagcolest_{j}(\pi)\in\abstractpls_{\lambda}(\sampledagmatelem_{j},\sampledagmat;S_{j})$ for each $j$. Since $\pi$ was arbitrary, the desired claim follows.

\subsection{Proof of Lemma~\ref{lem:lattice}}
\label{app:lem:lattice}

The proof relies on the following property of $L^2$ projections: For any two sets $S,R \subset [p]_j$, we have
\begin{align}
\label{eq:L2proj}
    \nhbdcoef_j(S \cup R) = \nhbdcoef_j(S) \iff \nhbderr_j(S) \independent \dagvec_{i}, \; \forall i \in R.
\end{align}

\def\Ss{S^{*}}
To lighten the notation, let $\Ss = \nhbdsupp_j(S)$. Note that $\nhbdcoef_j(S) = \nhbdcoef_j(\Ss)$ since $\supp(\nhbdcoef_j(S)) = \Ss$. It follows from \eqref{eq:L2proj} that $\nhbderr_j(\Ss) \independent \dagvec_{i}$ for $i \in S \setminus \Ss$. Similarly, since $\supp(\nhbdcoef_j(T_k)) = \Ss$, we have  $\nhbderr_j(\Ss) \independent \dagvec_{i}$ for $i \in T_k \setminus \Ss$ and $k = 1,2$. It follows that
\begin{align*}
\nhbderr_j(\Ss) \independent \dagvec_{i}, \; \forall i \in (T_1 \setminus \Ss) \cup (T_2 \setminus \Ss)
\end{align*}

\noindent
hence the application of \eqref{eq:L2proj} in the reverse direction yields
\begin{align*}
\nhbdcoef_j(T_1 \cup T_2) 
&= \nhbdcoef_j \big( \Ss \cup (T_1 \setminus \Ss) \cup (T_2\setminus \Ss)\big) 
= \nhbdcoef_j(\Ss) 
= \nhbdcoef_j(S).
\end{align*}

\subsection{Proof of Lemma~\ref{lem:estim:supp:monot}}
\label{app:lem:estim:supp:monot}

It suffices to show 
$$\mserrset(\fixeddesignmat,\lincoef;U)^{c}\subset\mserrset(\fixeddesignmat,\lincoef;S)^{c}.$$ 
Suppose $\fixedlinerr\in\mserrset(\fixeddesignmat,\lincoef;U)^{c}$, i.e., $\supp(\widetilde{\theta})=\supp(\lincoef):=S^{*}$ for any $\widetilde{\theta}\in\abstractpls_{\lambda}(\fixeddesignmat\lincoef+\fixedlinerr,\fixeddesignmat;U)$. We wish to show that for any $\hlincoef\in\abstractpls_{\lambda}(\fixeddesignmat\lincoef+\fixedlinerr,\fixeddesignmat;S)$, it must also be true that $\supp(\hlincoef)=S^{*}$. Let
\begin{align*}
F(\theta)
= \frac{1}{2n} \norm{\fixeddesignmat(\lincoef-\theta)+\fixedlinerr}_2^2 + \pl(\theta)
\end{align*}
denote the objective function in Definition~\ref{defn:abstractregression} of $\abstractpls_\lambda(\fixedy, \fixeddesignmat; S)$  with $\fixedy=\fixeddesignmat\lincoef+\fixedlinerr$.
Since $\supp(\hlincoef)\subset S\subset U$, $\hlincoef$ is feasible for the $U$-restricted problem, whence
\begin{align*}
F(\widetilde{\theta}) \le F(\hlincoef) 
\end{align*}

\noindent
for any $\widetilde{\theta}\in\abstractpls_{\lambda}(\fixeddesignmat\lincoef+\fixedlinerr,\fixeddesignmat;U)$. But $\widetilde{\theta}$ is also feasible for the $S$-restricted problem since $\supp(\widetilde{\theta}) = S^{*}\subset S$, so that 
\begin{align*}
F(\widetilde{\theta}) \ge F(\hlincoef)
\implies F(\widetilde{\theta}) = F(\hlincoef).
\end{align*}

\noindent
Since the value $F(\widetilde{\theta})$ is by definition the global minimum of $F$ for the $U$-restricted problem and $\supp(\hlincoef)\subset U$, $\hlincoef$ must be a global minimizer of $F$ for the $U$-restricted problem, i.e., $\hlincoef\in\abstractpls_{\lambda}(\fixeddesignmat\lincoef+\fixedlinerr,\fixeddesignmat;U)$, whence $\supp(\hlincoef)=S^{*}$ as desired.

\subsection{Proof of Corollary~\ref{cor:ms:monotone}}
\label{app:cor:ms:monotone}

By Lemma~\ref{lem:estim:supp:monot} and the fact that $S\subset \nhbdmax_{j}(S)$, we have
\begin{align}
\label{eq:cor:ms:monotone:1}
\msevent\Big(\sampledagerrcol_{j}(S), \sampledagmat, \nhbdcoef_j(S); S \Big)
&\subset \msevent\Big(\sampledagerrcol_{j}(S), \sampledagmat, \nhbdcoef_j(S); \nhbdmax_{j}(S)\Big).
\end{align}

\noindent
Using \eqref{eq:invariant:nhbdcoef} and \eqref{eq:invariant:nhbderr}, we have the following identity:
\begin{align*}
\msevent\Big(\sampledagerrcol_{j}(S), \sampledagmat, \nhbdcoef_j(S); \nhbdmax_{j}(S)\Big)
&= \msevent\Big(\sampledagerrcol_{j}(\nhbdmax_{j}(S)), \sampledagmat, \nhbdcoef_j(\nhbdmax_{j}(S)); \nhbdmax_{j}(S)\Big).
\end{align*}

\noindent
Plugging this into \eqref{eq:cor:ms:monotone:1} yields the desired result.

\subsection{Proof of Proposition~\ref{prop:gen:reduct}}
\label{app:prop:gen:reduct}

Throughout, for simplicity, let
\begin{align*}
    \msevent_S :=  \msevent(\sampledagerrcol_{j}(S),\sampledagmat, \nhbdcoef_j(S); S).
\end{align*}
Fix $S\subset[p]_j$ and let $\lincoef=\nhbdcoef_{j}(S)$, $s^{*}=|\nhbdsupp_{j}(S)|=\norm{\lincoef}_{0}$ and $\eps^{*}=\sampledagerrcol_{j}(S)$ so that $\msevent_S = \msevent(\eps^*,\sampledagmat, \lincoef; S)$. Note that $\msevent(\eps^*,\sampledagmat, \lincoef; S)$ represents the following model selection failure:
\begin{align*}
\supp(\hlincoef) \ne \supp(\lincoef)
\quad\exists\,\hlincoef\in\abstractpls_{\lambda}(\sampledagmat\lincoef  + \eps^{*}, \sampledagmat; S).
\end{align*}
Since $\supp(\lincoef)\subset S$, we can restrict $\sampledagmat$ and $\lincoef$ to $S$, so that the above is equivalent to
\begin{align*}
\supp(\hlincoef) \ne \supp(\lincoef_{S})
\quad\exists\,\hlincoef\in\abstractpls_{\lambda}(\select{\sampledagmat}{S}\lincoef_{S}  + \eps^{*}, \select{\sampledagmat}{S}).
\end{align*}
which is the same event as $\msevent(\eps^*, \select{\sampledagmat}{S}, \lincoef_S)$. To summarize, $\msevent_S = \msevent(\eps^*, \select{\sampledagmat}{S}, \lincoef_S)$.

Since $\eps^*$ is independent of $\select{\sampledagmat}{S}$ by Lemma~\ref{lem:restricted:nhbd}(a), by conditioning on $\select{\sampledagmat}{S}$ we are dealing with a fixed design regression problem with Gaussian noise $\eps^* = \sampledagerrcol_{j}(S) \sim \normalN_n(0,\nhbdvar_{j}^2(S) I_n)$. We obtain
\begin{align*}
\pr(\msevent_S) 
&= \E\Big[\pr \big( \msevent(\eps^*, \select{\sampledagmat}{S}, \lincoef_S) \big) \; \big| \; \select{\sampledagmat}{S}\big)\Big] \\
&\le \E\exp[-\msexp_\lambda(\select{\sampledagmat}{S},\select{\lincoef}{S},\nhbdvar_{j}^{2}(S))] \\
&= \E\exp(-\nhbdmsexp_{j}(S)),
\end{align*}
where the last line uses \eqref{eq:def:nhbdmsexp}. Now we have 
\begin{align*}
    \{\supp(\hnhbdcoef_j(S)) &\neq \supp(\nhbdcoef_j(S)),\; \exists\,S\subset[p]_{j} \}  = \bigcup_{S\subset[p]_j} \msevent_{S} \; \subset\; \bigcup_{T \in \nhbdsupp_{j}(\trueCov)} \msevent_{\nhbdmax_j(T)},
\end{align*}
where the equality is by~\eqref{eq:msevet:356} and the inclusion follows from Corollary~\ref{cor:ms:monotone}. Note that this is the key step where the reduction occurs. Hence, we have
\begin{align*}
\pr \big( \bigcup_{S\subset[p]_j} \msevent_{S} \big)
&\le \pr \Big(\bigcup_{T \in \nhbdsupp_{j}(\trueCov)} \msevent_{\nhbdmax_j(T)} \Big) \\
&\le \sum_{T \in \nhbdsupp_{j}(\trueCov)} \pr(\msevent_{\nhbdmax_j(T)} ) 
\le \sum_{T \in \nhbdsupp_{j}(\trueCov)}  \E\exp(-\nhbdmsexp_{j}(\nhbdmax_j(T))),
\end{align*}
which is the desired probability bound.

\subsection{Proof of Theorem~\ref{thm:fixed:l1l2bound}}
\label{app:thm:fixed:l1l2bound}

Recall that $S^{*}:=\supp(\lincoef)$. To lighten notation, for any vector $\genu$ let $\genu_{1} := \genu_{S^{*}}$, $\genu_{2}:=\genu_{(S^{*})^{c}}$, and also $\Delta:=\hlincoef-\lincoef$. Then invoking the subadditivity of $\pl$ (this is a consequence of Condition~\ref{condn:pl:basic}),
\begin{align}
\pl(\hlincoef) - \pl(\lincoef) 
&= \pl(\Delta+\lincoef) - \pl(\lincoef) \notag \\
&=  \pl(\Delta_1+\lincoef_1) + \pl(\Delta_2) -\pl(\lincoef_1) \notag \\
&\ge -\pl(\Delta_1) + \pl(\Delta_2).\label{eq:temp:475}
\end{align}

It is straightforward to derive 
\begin{align}
\label{eq:thm:fixed:l1l2bound:1}
\frac{1}{2n} \norm{\fixedy - \fixeddesignmat\hlincoef}_2^2
- \frac{1}{2n} \norm{\fixedy - \fixeddesignmat\lincoef}_2^2
= \frac{1}{2n}\norm{\fixeddesignmat \Delta}^2  
- \frac{1}{n} \ip{\fixedlinerr,\fixeddesignmat \Delta}.
\end{align}

\noindent
Since $(\fixedlinerr,\fixeddesignmat)\in\GW_{\pl}(\gwconst)$, we can invoke the GW condition with $\genu=\Delta$,
\begin{align}
\label{eq:thm:fixed:l1l2bound:2}
-\frac{1}{n} \ip{\fixedlinerr,\fixeddesignmat \Delta} 
\ge -\frac{1}{n}| \ip{\fixedlinerr,\fixeddesignmat \Delta} |
\ge -\gwconst\frac{1}{2n} \norm{\fixeddesignmat \Delta}^2 
- \gwconst\pl(\Delta).
\end{align}
It follows that
\begin{align}
0
&\ge \frac{1}{2n} \norm{\fixedy - \fixeddesignmat\hlincoef}_2^2
- \frac{1}{2n} \norm{\fixedy - \fixeddesignmat\lincoef}_2^2
+ \pl(\hlincoef) - \pl(\lincoef) \notag \\
&\ge \frac{1}{2n}\norm{\fixeddesignmat \Delta}^2  - \frac{1}{n} \ip{\fixedlinerr,\fixeddesignmat \Delta} -\pl(\Delta_1) + \pl(\Delta_2) \notag \\
&\ge \frac{1-\gwconst}{2n}\norm{\fixeddesignmat \Delta}^2 - \gwconst\pl(\Delta) -\pl(\Delta_1) + \pl(\Delta_2) \notag \\
&= \frac{1-\gwconst}{2n}\norm{\fixeddesignmat \Delta}^2 - (1+\gwconst)\pl(\Delta_{1}) + (1-\gwconst)\pl(\Delta_2) \notag \\
\label{eq:thm:fixed:l1l2bound:4}
&= (1-\gwconst)\Big[\frac{1}{2n}\norm{\fixeddesignmat \Delta}^2 + \pl(\Delta_{2}) - \xi\pl(\Delta_{1})\Big],
\end{align}

\noindent
where the first inequality by optimality of $\hlincoef$, the second by~\eqref{eq:thm:fixed:l1l2bound:1}, and the third by~\eqref{eq:thm:fixed:l1l2bound:2}. The next line follows from an an application of $\pl(\Delta) = \pl(\Delta_{1}) + \pl(\Delta_{2})$. Since $\gwconst<1$ by assumption, it follows that $\pl(\Delta_2) \le \xi\pl(\Delta_{1})$ which implies  $\Delta\in\rescone(S^{*},\xi(\gwconst))$.

\smallskip
Recalling the definition \eqref{eq:defn:genrec} of $\rec_{\pen}^2(\fixeddesignmat, S^{*})$, we conclude that
$\frac{1}{2n}\norm{\fixeddesignmat \Delta}^2
\ge \frac{\rec^{2}}{2} \norm{\Delta}_{2}^{2}
$
which combined with \eqref{eq:thm:fixed:l1l2bound:4}, dropping $\pl(\Delta_2)$, gives
\begin{align*}
0 \; \ge \; \frac{\rec^{2}}{2}\norm{\Delta}_{2}^{2} - \xi\pl(\Delta_{1}).
\end{align*}
Combining with the following (note $\norm{\Delta_{1}}_0 \le \norm{\lincoef}_0$),
\begin{align}
\label{eq:thm:fixed:l1l2bound:8}
\pl(\Delta_{1}) \le \penderiv\norm{\Delta_{1}}_{1}\le\penderiv\norm{\lincoef}_{0}^{1/2}\norm{\Delta}_{2}
\end{align}
and re-arranging proves~\eqref{eq:L2bound}.
For \eqref{eq:L1bound}, since $\Delta\in\rescone(S^{*},\xi(\gwconst))$, we can use Lemma~\ref{lem:restrictedcone} to construct a set $M\subset[p]$ with $|M|=|S^{*}| = \norm{\lincoef}_{0}$ such that $\Delta\in\cone_{1}(M,\xi(\gwconst))$. Then
\begin{align*}
\norm{\Delta}_{1}
= \norm{\Delta_{M}}_{1} + \norm{\Delta_{M^{c}}}_{1} 
&\le (1+\xi)\norm{\Delta_{M}}_{1} \\
&\le (1+\xi)\norm{\lincoef}_{0}^{1/2}\norm{\Delta_{M}}_{2} \\
&\le \frac{2\,\xi(1+\xi)}{\rec^{2}}\cdot\penderiv \norm{\lincoef}_{0}.\qedhere
\end{align*}

\begin{remark}
\label{rem:ell0}
Theorem~\ref{thm:fixed:l1l2bound} also applies to the $\ell_{0}$ penalty, even though Condition~\ref{condn:pl:basic} does not hold. Simply replace \eqref{eq:thm:fixed:l1l2bound:8} with
\begin{align*}
\pl(\Delta_{1}) 
\le \uppflat\lambda^{2}\norm{\Delta_{1}}_{0}
\le \uppflat\lambda^{2}\norm{\lincoef}_{0}.
\end{align*}

\noindent
The rest of the proof requires no changes. The resulting $\ell_{2}$ bound \eqref{eq:L2bound} becomes
\begin{align*}
\norm{\hlincoef-\lincoef}_{2}
&\le \sqrt{\frac{2\uppflat\xi}{\rec^2}}\cdot\lambda\norm{\lincoef}_{0}^{1/2}.
\end{align*}

\noindent
The proof of Lemma~\ref{lem:hBlowerbound} in Appendix~\ref{app:sparse} also simplifies significantly under the $\ell_{0}$ penalty.
\end{remark}

\subsection{Proof of Lemma~\ref{lem:epms:equiv}}
\label{app:lem:epms:equiv}

If $(\samplelinerr,\fixeddesignmat) \in \GW_{\pl}^\circ(\gwconst)$, then for any $\genu\ne0$,
\begin{align*}
\frac{\gwconst}{2n}\norm{\fixeddesignmat\genu}_{2}^{2} -\frac{1}{n}\samplelinerr^{T}\fixeddesignmat\genu + \gwconst\pl(\genu) 
&> 0 \\
\iff
\frac{1}{2n}\norm{\samplelinerr/\gwconst - \fixeddesignmat\genu}_{2}^{2} + \pl(\genu)
&> \frac{1}{2n}\norm{\samplelinerr/\gwconst}_{2}^{2}.
\end{align*}

\noindent
The latter inequality implies  
\begin{align*}
\{0\}
&= \argmin_{\genu}\norm{\samplelinerr/\gwconst - \fixeddesignmat\genu}_{2}^{2}/(2n) + \pl(\genu),
\end{align*}
that is, $0$ is the unique global minimizer of the right hand side. Recalling the definition of $\msevent\big(\samplelinerr/\gwconst,\fixeddesignmat, 0\big)$ in \eqref{eq:def:msevent}, we obtain the desired result.

\subsection{Proof of Lemma~\ref{lem:prelimbound}}
\label{app:lem:prelimbound}

Observe that for any $\perm\in\allperms$,
\begin{align}
\label{eq:lem:prelimbound:1}
\plsobj(\dagadjest)
\le \plsobj(\dagadjest(\perm))
\le \plsobj(\tB{\perm}),
\end{align}

\noindent
where $\dagadjest(\perm)$ is a restricted minimizer as in \eqref{eq:def:restrictedmin}. Moreover, we have the following alternative expression for $\plsobj$: 
\begin{align}
\label{eq:lem:prelimbound:2}
\plsobj(\gendag)
= \frac{1}{2n}\norm{\sampledagmat(\tB{\estperm} - \gendag) + \sampledagerrmat(\estperm)}_{2}^{2} + \pl(\gendag),
\quad\text{for any $\estperm\in\estperms$.}
\end{align}

\noindent
Thus, using \eqref{eq:lem:prelimbound:1} and \eqref{eq:lem:prelimbound:2},
\begin{align*}
0 
&\le \plsobj(\tB{\perm}) - \plsobj(\dagadjest) \\
&= \frac{1}{2n}\norm{\sampledagerrmat(\perm)}_{2}^{2} 
- \frac{1}{2n}\norm{\sampledagmat(\tB{\estperm} - \dagadjest) - \sampledagerrmat(\estperm)}_{2}^{2}
+ \pl(\tB{\perm}) - \pl(\dagadjest)\\
&= \frac{1}{2n}\norm{\sampledagerrmat(\perm)}_{2}^{2} - \frac{1}{2n}\norm{\sampledagerrmat(\estperm)}_{2}^{2}
- \frac{1}{2n}\norm{\sampledagmat(\tB{\estperm} - \dagadjest)}_{2}^{2} \\
&\qquad + \frac1n\tr\left(\bErr{\estperm}^T\sampledagmat(\tB{\estperm} - \dagadjest)\right) 
+ \pl(\tB{\perm}) - \pl(\dagadjest).
\end{align*}

\noindent
Since \eqref{eq:lem:prelimbound:1} holds for any $\perm$, this completes the proof.

\section{Auxiliary results for deviation bounds}
\label{app:dev}

This section contains the bulk of our technical results on controlling the deviations $\hnhbdcoef_j(S) - \nhbdcoef_j(S)$ for a neighbourhood problem. These constitute the main ingredients used in proving Theorem~\ref{thm:main:dev}, which we also do in this section. 

For any $\gwconst\in(0,1)$ and $\lambda\ge0$, define the following events:
\begin{align}
\label{eq:def:epevent}
\epevent(\gwconst, \lambda)
&= \Big\{ \big(\sampledagerrcol_{j}(S), \select{\sampledagmat}{S}  \big) \in \GW_{\pl}^\circ(\gwconst),  \,\,\forall j\in[p],S\subset[p]_{j}
 \Big\}, \\
 \label{eq:def:recevent}
\recevent(\gwconst)
&= \Big\{
\rec_{\pen}^2(\select{\sampledagmat}{S}, \nhbdsupp_{j}(S)) \ge \mineval(\trueCov) > 0, \,\,\forall j\in[p],S\subset[p]_{j}
\Big\}. 
\end{align}

\noindent
Note that $\epevent(\gwconst, \lambda) = \bigcap_{j=1}^{p}\bigcap_{S\subset[p]_j}\mE_{S}(\gwconst,\lambda;j)$, where $\epevent_{S}(\gwconst,\lambda;j)$ was defined in \eqref{eq:Ejpi}.

The structure of the proofs will be to show that on one or both of these events, the desired conclusions follow. Explicit bounds on the probabilities of these events are established in  Section~\ref{app:devnhbd}.

\subsection{Uniform deviation bounds}
\label{app:devnhbd}

As discussed in Section~\ref{sec:unif:dev:bounds}, one of the main ingredients in proving Theorem~\ref{thm:main:dev} is a general result regarding deviation bounds for neighbourhood problems, given by Proposition~\ref{prop:abstract:dev} below. The other ingredient is Proposition~\ref{prop:GW:capped:ell1} below, regarding the behavior of certain model selection exponents defined in analogy with \eqref{eq:def:unifmsexp}:
\begin{align}
\label{eq:def:unifgwexp}
\unifgwexp_{\lambda}(\sampledagmat,\maxvar^2; \gwconst)
:= \inf_{0 \,\le \,\sigma\, \le\, \maxvar\;} \msexp_{\lambda}(\sampledagmat,\, 0,\, \sigma^2/\gwconst^{2}).
\end{align}
We often suppress the dependence on $\delta$ and write $\unifgwexp_{\lambda}(\sampledagmat,\maxvar^2)$.
Note that, in view of Lemma~\ref{lem:epms:equiv}, $\unifgwexp_{\lambda}(\sampledagmat,\maxvar^2)
$ describes the conditional probability, given $\sampledagmat$, that $(\sigma \samplelinerr,\sampledagmat)$ violates a GW condition, where $\samplelinerr \sim \normalN_n(0,I_n)$ is independent of $\sampledagmat$. More precisely,
\begin{align*}
\sup_{0\le\sigma\le\maxvar} \pr\big[ (\sigma \samplelinerr,\sampledagmat) \notin \GW^\circ_{\pl}(\delta) \; \big|\; \sampledagmat\big] 
&\;=\; \sup_{0\le\sigma\le\maxvar} \exp[-\msexp_{\lambda}(\sampledagmat,\, 0,\, \sigma^2/\gwconst^{2})] \\
&\;=\; \exp[-\unifgwexp_{\lambda}(\sampledagmat,\maxvar^2)].
\end{align*}

We also recall the relation
\begin{align}
 \label{eq:def:xi}
 \xi=\xi(\gwconst)
 =\frac{1+\gwconst}{1-\gwconst}.
\end{align}

\begin{prop}
\label{prop:abstract:dev}
Assume that $\trueCov$ satisfies Condition~\ref{condn:evals} and $\pl$ satisfies Condition~\ref{condn:pl:basic}. Suppose $\sampledagmat\iid\normalN_{p}(0,\trueCov)$, $\gwconst\in(0,1)$, and define $\xi$ by \eqref{eq:def:xi}. Then there exist constants $c_{0},c_{1},c_{2}>0$ such that the following holds:
If
\begin{align*}
n 
&> c_{0}\frac{\maxvar^{2}(1+\xi)^{2}}{\mineval(\trueCov)}d\log p,
\end{align*}
then with probability at least $1 - c_{1}\exp(-c_{2}n) - p\binom{p}{d}\E\exp(-\unifgwexp_{\lambda}(\sampledagmat,\maxvar^2;\delta))$,
\begin{align}
\label{eq:prop:abstract:dev:col:L2}
\norm{\dagcolest_{j}(\dagcand)-\nhbdcoef_{j}(\dagcand)}_{2}
&\le \frac{2\,\xi}{\mineval(\trueCov)}\penderiv\cdot\norm{\nhbdcoef_{j}(\dagcand)}_{0}^{1/2}, \\
\label{eq:prop:abstract:dev:col:L1}
\norm{\dagcolest_{j}(\dagcand)-\nhbdcoef_{j}(\dagcand)}_{1}
&\le \frac{2\,\xi(1+\xi)}{\mineval(\trueCov)}\penderiv\cdot\norm{\nhbdcoef_{j}(\dagcand)}_{0},
\end{align}

\noindent
uniformly over all $j\in[p]$ and $S\subset[p]_{j}$.
\end{prop}
For future reference, inspection of the proof shows that the conclusion of Proposition~\ref{prop:abstract:dev} holds on $\epevent(\gwconst, \lambda) \cap \recevent(\gwconst)$.
For regularizers that satisfy the lower bound in Condition~\ref{condn:pl:basic}(c) we have the following control on the exponent $\unifgwexp_{\lambda}(\sampledagmat,\maxvar^2)$:
\begin{prop}
\label{prop:GW:capped:ell1}
    Assume that $\sampledagmat\iid\normalN_{p}(0,\trueCov)$, and that $\pl$ satisfies Condition~\ref{condn:pl:basic}(c). Then there exist constants $c>0$ and $C = C(\lowlin,\lowflat)$ such that for any $\gwconst\in(0,1)$, if
    \begin{align}
        \lambda \ge C \delta^{-1}  \maxvar \norm{\trueCov}^{1/4} \sqrt{\frac{(d+1)\log p}{n}}
    \end{align}
    then $\E\exp(-\unifgwexp_{\lambda}(\sampledagmat,\maxvar^2;\delta)) \le c\exp(-\min\{2(d+1)\log p, n\})$.
\end{prop}

\noindent
The proof of Proposition~\ref{prop:GW:capped:ell1} follows from an argument similar to that in~\cite{zhang2012} and is omitted for brevity. In order to prove  Proposition~\ref{prop:abstract:dev}, we need the following two intermediate results, providing uniform control on RE constants and GW conditions. Recall $\epevent(\gwconst, \lambda)$ as defined in~\eqref{eq:def:epevent}.

\begin{prop}[Uniform GW control]
\label{prop:gwcond:probbound}
For any $\delta \in (0,1)$ and $\lambda > 0$,
\begin{align*}
\pr [\epevent(\gwconst, \lambda)]
&\ge 1 - p\binom{p}{d}\E\exp\big[{-\unifgwexp_{\lambda}(\sampledagmat,\maxvar^2;\gwconst})\big].
\end{align*}
\end{prop}

\def\mF{\mathcal{F}}
\begin{proof}
Fix $\delta \in (0,1)$. By analogy with \eqref{eq:def:nhbdmsexp}, for any neighbourhood $S\subset[p]_{j}$, let
\begin{align}
\label{eq:def:nhbdgwexp}
\nhbdgwexp_{j}(S)
    \;:= \;\msexp_{\lambda}(\sampledagmat_{S},\, 0,\, \nhbdvar_{j}^{2}(S)/\gwconst^{2}) 
    \;\ge\; \unifgwexp_{\lambda}(\sampledagmat,\maxvar^2; \gwconst),
\end{align}
where the inequality follows from \eqref{eq:def:unifgwexp} and $\nhbdvar_{j}^{2}(S)\le\maxvar^{2}$. We follow the proof of Proposition~\ref{prop:gen:reduct}, but with $\beta_j(S)$ replaced with 0, and $\sampledagerrcol_{j}(S)$ replaced with $\sampledagerrcol_{j}(S)/\delta$. To simplify, let $\mF^j_S = \epevent_{S}(\gwconst,\lambda; j)^c$ and $\epevent = \epevent(\gwconst,\lambda)$. By the comment following~\eqref{eq:Ejpi},
\begin{align*}
    \mF^j_S = \msevent\big(\sampledagerrcol_{j}(S)/\gwconst,\sampledagmat, 0; S \big) = 
    \msevent\big(\sampledagerrcol_{j}(S)/\gwconst,\sampledagmat_S, 0 \big)
\end{align*}
where the second equality is by the same argument in the proof of Proposition~\ref{prop:gen:reduct}. Since $\sampledagerrcol_{j}(S)/\gwconst \sim \normalN\big(0,[\nhbdvar_{j}^{2}(S)/\gwconst^{2}] I_n\big)$ independent of $\sampledagmat_S$, we conclude, using Definition~\ref{defn:msexp}, that 
\begin{align*}
    \pr \big(\mF^j_S \mid \sampledagmat_S\big) = \exp[- \xi_j(S)],
\end{align*}
hence $\pr(\mF^j_S) \le \E \exp [ - \unifgwexp_{\lambda}(\sampledagmat,\maxvar^2)]$, $\forall S \subset [p]_j$, using the inequality in~\eqref{eq:def:nhbdgwexp}.
The events $\mF^j_S$ are monotonic in $S$ according to Corollary~\ref{cor:ms:monotone}. (The division of $\nhbderr_j(S)$ by $\delta$ does not change anything in that proof.) It follows that 
\begin{align*}
    \epevent^c = \bigcup_{j=1}^p \,\bigcup_{S \,\subset\, [p]_j} \mF^j_S
        \;\;\subset\;\; \bigcup_{j=1}^p \; \bigcup_{T \,\in\, m_j(\trueCov)} \mF_{M_j(T)}^j.
\end{align*}
Taking the union bound, and using $ |m_j(\trueCov)| \le \binom{p}{d}$ and 
\begin{align*}
    \pr\big[\,\mF_{M_j(T)}^j \,\big] \le \E \exp [ - \unifgwexp_{\lambda}(\sampledagmat,\maxvar^2)], \quad \forall\, T \in m_j(\trueCov),
\end{align*}
 finishes the proof.
\end{proof}

\begin{prop}[Uniform RE control]
\label{prop:restricted:recprob}
Assume $\sampledagmat\iid\normalN_{p}(0,\trueCov)$, $\trueCov$ satisfies Condition~\ref{condn:evals} and $\pl$ satisfies Condition~\ref{condn:pl:basic}. There exist universal constants $c_{0},c_{1},c_{2}>0$, such that if
\begin{align*}
n > c_{0}\,\frac{\maxvar^{2}(1+\xi)^2}{\mineval(\trueCov)}d(\trueCov)\log p
\end{align*}

\noindent
then with probability at least $1-c_{1}\exp(-c_{2}n)$, 
\begin{align*}
\inf_{1\le j\le p}\, \inf_{S \,\subset\, [p]_{j}}\, \inf_{\substack{A \,\subset\, S\\|A|\le \maxdeg}} \; \rec_{\pen}^{2}(\sampledagmat_{S},A;\,\xi)
\;\ge\; \mineval(\trueCov).
\end{align*}
\end{prop}
The proof of this proposition appears in Section~\ref{subsec:proof:unif:RE} below. Recalling the definition of $\recevent(\gwconst)$ in \eqref{eq:def:recevent}, combined with $m_j(S) = \norm{\beta_j(S)}_0 \le d$ (cf. Definition~\ref{defn:eqclassparam}), Proposition~\ref{prop:restricted:recprob} implies that $\recevent(\gwconst)$ holds with probability at least $1-c_{1}\exp(-c_{2}n)$.
Let us show how Proposition~\ref{prop:abstract:dev} follows.

\begin{proof}[Proof of Proposition~\ref{prop:abstract:dev}]
\label{app:}

Recall the definitions of $\epevent(\gwconst, \lambda)$ in \eqref{eq:def:epevent} and $\recevent(\gwconst)$ in \eqref{eq:def:recevent}. Propositions~\ref{prop:gwcond:probbound} and~\ref{prop:restricted:recprob} guarantee that 
\begin{align*}
\pr\big(
\recevent(\gwconst) \cap \epevent(\gwconst, \lambda)
\big)
\ge 1 - c_{1}\exp(-c_{2}n) - p\binom{p}{d}\E\exp(-\unifgwexp_{\lambda}(\sampledagmat,\maxvar^2;\gwconst)).
\end{align*}

\noindent
Thus, it suffices to deduce \eqref{eq:prop:abstract:dev:col:L2} and \eqref{eq:prop:abstract:dev:col:L1} whenever we are on the event $\recevent(\gwconst) \cap \epevent(\gwconst, \lambda)$. The case $\nhbdcoef_j(S)=0$ follows from \eqref{eq:Ejpi} and Lemma~\ref{lem:epms:equiv}, and the case $\nhbdcoef_j(S)\ne0$ follows from Theorem~\ref{thm:fixed:l1l2bound} applied to the corresponding neighbourhood regression problems.
\end{proof}

\subsection{Uniform control of restricted eigenvalues}
\label{subsec:proof:unif:RE}
In this section, we collect the necessary results that lead to the proof of Proposition~\ref{prop:restricted:recprob}. We begin with a definition which generalizes the familiar ($\ell_{1}$) restricted eigenvalue \citep{bickel2009,raskutti2011}:
\begin{defn}
\label{defn:re:conditionZ}
$\fixeddesignmat\in\R^{n\times m}$ satisfies a \emph{generalized restricted eigenvalue condition of order $k$ w.r.t. $\pl$} with parameters $\recparam,\xi>0$, denoted as $\fixeddesignmat\in\REC_{\pl}(k,\recparam;\xi)$, if
\begin{align*}
\frac{1}{n}\norm{\fixeddesignmat\genu}_{2}^{2}\ge\recparam^{2}\norm{\genu}_{2}^{2}
\quad\forall\genu\in\cone_{\pl}(A,\xi),
\end{align*}
uniformly for all $A\subset[m]$ with $|A|=k$. Equivalently, recalling Definition~\ref{defn:genrec},
\begin{align*}
\fixeddesignmat\in\REC_{\pl}(k,\recparam;\xi) \iff 
    \inf_{\substack{A\subset[m] \\|A|=k}}\rec_{\pl}^{2}(\fixeddesignmat, A, \xi) 
\ge \recparam^{2}.
\end{align*}
\end{defn}

In the sequel, we will suppress the dependence of various quantities on $\lambda$, $\xi$ and $m$, when no confusion arises. For example, we will write $\pl=\pen$, $\REC_{\pen}(k,\recparam)=\REC_{\pl}(k,\recparam;\xi)$ or $\cone_{\pen}(A) = \cone^{m}_{\pl}(A,\xi)$.
The following lemma collects some simple consequences of these definitions.
\begin{lemma} The following hold:
\label{lem:RECcontain}
\begin{itemize}\itemsep=1ex
    \item[(a)] When $\pen$ is nondecreasing,
    \begin{itemize}
        \item $A'\subset A \implies \cone_{\pen}(A')\subset\cone_{\pen}(A)$,
        \item $A'\subset A \implies \rec_{\pen}^{2}(\fixeddesignmat, A')\ge\rec_{\pen}^{2}(\fixeddesignmat, A)$,
        \item $\fixeddesignmat\in\REC_{\pen}(k,\recparam) \implies \rec_{\pen}^{2}(\fixeddesignmat, A) \ge \alpha^2, \; \forall A:\;|A| \le k$,
        \item $k' \le k \implies \REC_{\pen}(k,\recparam) \subset \REC_{\pen}(k',\recparam) $.
    \end{itemize}
    \item[(b)]  $\fixeddesignmat\in\REC_{\pen}(k,\recparam) \implies \fixeddesignmat_{S}\in\REC_{\pen}(k \wedge |S|,\recparam)$.
\end{itemize}
\end{lemma}

The next result shows that we can control the generalized RE constants for $\select{\fixeddesignmat}{S}$ uniformly by a suitable generalized RE constant for $\fixeddesignmat$:

\begin{lemma}
\label{lem:rec:uniformfixed}
If $\fixeddesignmat\in\REC_{\pen}(\maxdeg,\recparam)$, then
\begin{align}
\label{eq:lem:rec:uniformfixed}
\inf_{1\le j\le m}\; \inf_{S\,\subset\, [m]_{j}}\; \inf_{\substack{A\, \subset\, S \\ |A|\, \le\, \maxdeg}}\rec_{\pen}^{2}(\fixeddesignmat_{S},A)
\ge \recparam^{2}.
\end{align}
\end{lemma}
\begin{proof}
Fix $j$, $S \subset[m]_{j}$ and $A \subset S$ with $|A| \le d$. By Lemma~\ref{lem:RECcontain}(b), the assumption implies $\fixeddesignmat_S \in\REC_{\pen}(\maxdeg \wedge |S|,\recparam)$. Then, the last assertion in Lemma~\ref{lem:RECcontain}(a) implies $\fixeddesignmat_S \in \REC_{\pen}(|A|,\recparam)$, hence $\rec_{\pen}^2(\fixeddesignmat_S, A) \ge \recparam^{2}$. Since the lower bound does not depend on $A$, $S$ or $j$, we get the desired result.
\end{proof}

Next, we show that we can control RE constants for $\rho$ by those for the $\ell_1$ norm. Let us write $\cone_{1}^{m}(A,\xi)$ for the cone corresponding to $\rho = \norm{\cdot}_1$, and similarly for the RE constants. We have the following lemma:
\begin{lemma}
\label{lem:restrictedcone}
Under Condition~\ref{condn:pl:basic} on $\pen$,
\begin{align*}
\rescone^{m}(A,\xi)
\subset\bigcup_{\substack{A'\subset[m] \\|A'|=|A|}}\cone_{1}^{m}(A',\xi).
\end{align*}
\end{lemma}

\begin{proof}[Proof of Lemma~\ref{lem:restrictedcone}]
Fix nonzero $\genu\in\rescone(A,\xi)$ and assume, without loss of generality that $|u_i| > 0$ for all $i \in [m]$; otherwise, we can inflate all the zero entries by $\epsilon > 0$, change $\xi$ to $\xi + \pen(\epsilon)|A^c|/\pen(u_A)$, and let $\epsilon \to 0$ at the end.

Let $M = M(u)$ be the index set of the $|A|$ largest $|u_i|$. Then $\pen(\genu_{A^c}) \le \xi\pen(\genu_{A})$ implies $\pen(\genu_{M^c}) \le \xi\pen(\genu_{M})$ since $\pen$ is nondecreasing. 
We note that $M = \{i:\; |u_i| \ge \tau\}$, for some $\tau > 0$, assuming $M^c \neq \emptyset$ without loss of generality. As a consequence of Condition~\ref{condn:pl:basic}, $x \mapsto \pl(x)/x$ is nonincreasing. Then,
\begin{align*}
    \tau \pen(|u_i|) \le |u_i| \pen(\tau), \; i \in M,
 \end{align*} 
with the reverse inequality for $i \in M^c$. Summing over $M$ and $M^c$ and rearranging, we have 
\begin{align*}
    \frac{\pen(u_M)}{\norm{u_M}_1} \le  \frac{\pen(\tau)}{\tau} \le \frac{\pen(u_{M^c})}{\norm{u_{M^c}}_1},
\end{align*}
or $\norm{u_{M^c}}_1 / \norm{u_M} \le \pen(u_{M^c}) / \pen(u_{M}) \le \xi$. Hence, $u \in \cone^{m}_1(M(u),\xi)$ where $|M(u)| = |A|$, which gives the desired result.
\end{proof}

As a consequence of Lemma~\ref{lem:restrictedcone}, $\fixeddesignmat\in\REC_1(|A|,\recparam)$ implies $\rec_{\pen}^2(\fixeddesignmat, A) \ge \alpha^2$, from which we get:
\begin{lemma}\label{lem:RE:rho:ell1}
    Under Condition~\ref{condn:pl:basic}, $\REC_1(d,\recparam) \subset \REC_{\pen}(d,\recparam)$.
\end{lemma}

In particular, the conclusion of Lemma~\ref{lem:rec:uniformfixed} holds under the (stronger) assumption $\fixeddesignmat \in \REC_1(d,\recparam)$. In other words, to obtain uniform control over generalized restricted eigenvalues for all possible neighbourhood regression problems, it suffices to show that $\sampledagmat\in\REC_{1}(\maxdeg,\recparam)$ for some constant $\recparam>0$.
This is guaranteed by the following lemma, which is essentially a restatement of Corollary~1 in \cite{raskutti2010}:
\begin{lemma}
\label{thm:rec:modified}
Assume $\sampledagmat\iid\normalN_{p}(0,\trueCov)$ for some $\trueCov$ satisfying Condition~\ref{condn:evals}. There exist universal constants $c_{0},c_{1},c_{2}>0$, such that if
\begin{align*}
n > c_{0}\frac{\maxvar^{2}(1+\xi)^{2}}{\mineval(\trueCov)}d(\trueCov)\log p
\end{align*}
then with probability at least $1-c_{1}\exp(-c_2 n)$,
\begin{align*}
    \sampledagmat \, \in\, \REC_{1}\!\big(\maxdeg(\trueCov),\sqrt{\mineval(\trueCov)};\,\xi\big).
\end{align*}
\end{lemma}

Proposition~\ref{prop:restricted:recprob} now follows as a straightforward consequence of  Lemmas~\ref{lem:rec:uniformfixed},~\ref{lem:RE:rho:ell1} and~\ref{thm:rec:modified}.

\section{Auxiliary results for score-based learning}
\label{app:sparse}

This section provides some additional results which are needed to prove Theorems~\ref{thm:main:sparse},~\ref{thm:oracle}, and~\ref{thm:mintr}, which are also proved in this section.

For any $\gwconst\in(0,1)$, $\lambda\ge0$, $\resdevconst>0$, and $\pi\in\allperms$, define the following event:
\begin{align}
\label{eq:def:errevent}
\errevent(\resdevconst, \lambda; \pi)
&= \Bigg\{
\frac{1}{2n}\norm{\bErr{\pi}}_{\frob}^2 - \frac{1}{2n}\norm{\bErr{\estperm}}_{\frob}^2
\le \resdevconst\pl(\tB{\pi})
\Bigg\}.
\end{align}

\noindent
As in Appendix~\ref{app:dev}, the idea will be to show that on this event (along with \eqref{eq:def:epevent} and \eqref{eq:def:recevent}), the desired conclusions hold. In Appendix~\ref{app:resbound}, we provide an explicit bound on the probability of $\errevent(\resdevconst, \lambda; \pi)$.

\subsection{Some intermediate lemmas} 
\label{app:intermediate}

Recall the definitions of $\epevent(\gwconst, \lambda)$ and $\recevent(\gwconst)$ in \eqref{eq:def:epevent}--\eqref{eq:def:recevent}.
We start with the following extension of GW bounds: 
\begin{lemma}
\label{lem:trace}
Let $\Delh := \dagadjest - \tB{\estperm}$. On $\epevent(\gwconst, \lambda)$, we have
\begin{align}
\label{eq:empprocmat:estpi}
\frac{1}{n} \Bigg|\tr\Big(\bErr{\estperm}^T\sampledagmat \Delh\Big)\Bigg|
    < \gwconst\Big[ \frac{1}{2n}\norm{\sampledagmat \Delh }_{\frob}^{2} + \pl(\Delh)    \Big].
\end{align}
\end{lemma}

\begin{proof}
Let $\Delh_j := \dagcolest_j - \dagcol_{j}(\estperm)$ be the $j$th column of~$\Delh$. Then
\begin{align}
\label{eq:lem:trace:1}
\frac{1}{n}\Bigg|\tr\Big(\bErr{\estperm}^T\sampledagmat \Delh \Big)\Bigg|
&\le \frac{1}{n} \sum_{j=1}^{p} |\ip{\sampledagerrcol_{j}(\estperm),\sampledagmat\Delh_j}|.
\end{align}

\noindent
According to \eqref{eq:def:epevent}, on $\epevent(\gwconst, \lambda)$, we have $(\sampledagerrcol_{j}(S), \select{\sampledagmat}{S}  \big) \in \GW_{\pl}^\circ(\gwconst)$ for all $S \subset [p]_j$. In particular, applying with $S = S_j(\estperm)$ and using $u = \Delh_j$ in the Definition~\ref{defn:gwcond} of GW, we have
\begin{align*}
\frac{1}{n}| \ip{\sampledagerrcol_{j}(\estperm),\sampledagmat\Delh_j} |  
    < \gwconst \Big[\frac{1}{2n} \norm{\sampledagmat\Delh_j}_{2}^2 + \pl(\Delh_j)\Big], \quad \forall j
\end{align*}
Summing over $j$ and plugging into \eqref{eq:lem:trace:1} yields \eqref{eq:empprocmat:estpi}.
\end{proof}

For any matrix $A=(a_{ij})\in\R^{p\times p}$ and $S\subset[p]\times[p]$, let $\setzero{A}{S}$ denote the $p\times p$ matrix formed by zero-ing the elements outside of $S$, i.e.
\begin{align*}
(\setzero{A}{S})_{ij}
= \begin{cases}
a_{ij}, & (i,j)\in S, \\
0, & (i,j)\notin S.
\end{cases}
\end{align*}

\newcommand{\Tc}{\tau}
In analogy with Condition~\ref{condn:betamin} on signal strength, let us define
\begin{align}\label{eq:Tc}
    \Tc_\lambda(\alpha;\trueCov) := \inf \Big\{ \tau :\; 
    \frac{\penderiv^{2}}{\pl(\tau)}  \le \frac{\mineval(\trueCov)}{\alpha} \Big\}
\end{align}
where we often suppress the dependence on $\trueCov$. Note that we can write Condition~\ref{condn:betamin} equivalently as $\betamin \ge \Tc_\lambda(a_1)$.

\begin{lemma}
\label{lem:hBlowerbound}
Assume that $\pl$ satisfies Condition~\ref{condn:pl:basic} and  
\begin{align}
\label{eq:pencondn}
\betamin \ge \Tc_\lambda \Big(\frac{2\xi}{1-\penconst} \Big), \quad \text{ for some $\penconst\in(0,1)$}
\end{align}
where $\xi=\xi(\gwconst)$ is defined by \eqref{eq:def:xi}. Then, on $\recevent(\gwconst) \,\cap\,\epevent(\gwconst, \lambda)$,
\begin{align}\label{eq:temp:859}
\pl(\dagadjest) 
\;\ge\; \penconst\pl(\tB{\estperm}) + \pl\Big(\setzero{(\dagadjest - \tB{\estperm})}{\supp(\tB{\estperm})^c}\Big).
\end{align}
\end{lemma}

\begin{proof}
To lighten the notation, let $\Delta = \dagadjest - \tB{\estperm}$, $S_{1}=\supp(\tB{\estperm})$, $\Delta_{1}=\setzero{\Delta}{S_{1}}$, and $\Delta_{2}=\setzero{\Delta}{S_{1}^{c}}$. We have
\begin{align}
\label{eq:lem:hBlowerbound:1}
\pl(\Delta_1)
\le \penderiv\norm{\Delta_{1}}_{1}
\le \penderiv\norm{\tB{\estperm}}_{0}^{1/2}\norm{\Delta_{1}}_{2}.
\end{align}

\noindent
Since we are on $\recevent(\gwconst) \,\cap\,\epevent(\gwconst, \lambda)$, Proposition~\ref{prop:abstract:dev} yields the $\ell_{2}$ deviation bound \eqref{eq:prop:abstract:dev:col:L2}, which we use with $S = S_j(\estperm)$. Plugging into \eqref{eq:lem:hBlowerbound:1} and using $\norm{\Delta_{1}}_{2} \le \norm{\Delta}_2$,
\begin{align}
\label{eq:lem:hBlowerbound:3}
\pl(\Delta_1)
\le \big[\penderiv\big]^2\frac{2\xi}{\mineval(\trueCov)}\norm{\tB{\estperm}}_0.
\end{align}

\noindent
Trivially, we have $\pl(\tB{\estperm}) \ge \pl(\betamin)\norm{\tB{\estperm}}_0$, so that by \eqref{eq:lem:hBlowerbound:3}
\begin{align}
\label{eq:lem:hBlowerbound:4}
\pl(\Delta_1)
\le \Big[ \frac{\penderiv^{2}}{\pl(\betamin)}\frac{2\xi}{\mineval(\trueCov)}\Big] \pl(\tB{\estperm}) 
    \;\le\; (1-\penconst)\pl(\tB{\estperm}),
\end{align}
where the last inequality follows from~\eqref{eq:pencondn}. Finally, note that
\begin{align*}
\pl(\dagadjest)
&\;\ge\; \pl(\tB{\estperm}) + \pl(\Delta_2) - \pl(\Delta_1) \\
&\;\ge\; \penconst\pl(\tB{\estperm}) + \pl(\Delta_2).
\end{align*}
where the first inequality is by arguments similar to those leading to~\eqref{eq:temp:475} and the second is by~\eqref{eq:lem:hBlowerbound:4}. 
\end{proof}

\begin{remark}
\cite{geer2013} use a slightly weaker beta-min condition in which only a constant fraction of the edges of each DAG are assumed to be sufficiently large. Lemma~\ref{lem:hBlowerbound} and the ensuing arguments carry through under such an assumption: Under Condition~3.5 in \cite{geer2013}, we can use
\begin{align*}
\pl(\tB{\estperm})
\ge (1-\eta_{1})\pl(\betamin)\norm{\tB{\estperm}}_{0},
\end{align*}

\noindent
between \eqref{eq:lem:hBlowerbound:3} and \eqref{eq:lem:hBlowerbound:4} and obtain a bound similar to \eqref{eq:temp:859}, with only the constants modified.
\end{remark}

The conclusion of Lemma~\ref{lem:hBlowerbound} is stronger than what we need in the sequel. We only use the weaker inequality $\pl(\dagadjest) \ge \penconst\pl(\tB{\estperm})$ implied by~\eqref{eq:temp:859}.

\subsection{A sparsity bound}
\label{app:prop:abstract:sparse}

Proposition~\ref{prop:abstract:sparse} below is the main ingredient used in the proof of Theorem~\ref{thm:main:sparse}. This result provides an explicit, nonasymptotic relationship between the ``weak'' sparsity of $\{\dagadjest, \tB{\estperm}, \tB{\goodperm}\}$ as measured by the regularizer $\pl$.

\begin{prop}
\label{prop:abstract:sparse}
Assume $n>8\,(d+1)\log p$. Under Condition~\ref{condn:pl:basic} on $\pl$, further assume
\begin{align*}
\betamin \ge \Tc_\lambda \Big( \frac{2(1+\gwconst)}{1-3\gwconst} \Big) \quad \text{for some $\gwconst\in(0,1/3)$}.
\end{align*}
Then, given $\goodperm$ satisfying Condition~\ref{condn:mintrace:bound} for some $\mintraceconst>0$, we have
\begin{align}
\label{eq:prop:abstract:sparse}
\frac{2\gwconst}{1-\gwconst}\pl(\tB{\estperm})
\;\overset{(i)}{\le}\; \pl(\dagadjest)
\;\overset{(ii)}{\le}\; \frac{2}{1-\gwconst}\Big(1+\frac{10}{\mintraceconst}\Big)\pl(\tB{\goodperm}),
\end{align}

\noindent
with probability at least $1 - c_{1}e^{-c_{2} \min\{n,\, (d+1)\log p\}} - p\binom{p}{d}\E e^{-\unifgwexp_{\lambda}(\sampledagmat,\maxvar^2;\gwconst)}$.
\end{prop}

\begin{proof}
Recall the definition of $\errevent(\resdevconst, \lambda; \perm)$ in \eqref{eq:def:errevent}. Fix some $\goodperm$ satisfying Condition~\ref{condn:mintrace:bound} with $a_2 > 0$. Taking (arbitrarily) $C=1$ and $\resdevconst = 10/\mintraceconst$ in Proposition~\ref{prop:mintrace:prob}, we have 
\begin{align*}
\pr\big[\errevent(\resdevconst, \lambda; \goodperm)^c\big]
\le 2e^{-(d+1)\log p}.
\end{align*}

\noindent
Combined with Propositions~\ref{prop:restricted:recprob} and~\ref{prop:gwcond:probbound}, we obtain
\begin{align*}
\pr\big(
\errevent&(\resdevconst, \lambda; \goodperm) \cap \epevent(\gwconst, \lambda) \cap \recevent(\gwconst)
\big) \\
&\ge 1 - c_{1}\exp(-c_{2}\min\{n, (d+1)\log p\}) - p\binom{p}{d}\E\exp(-\unifgwexp_{\lambda}(\sampledagmat,\maxvar^2;\gwconst)).
\end{align*}
Thus, we may assume we are on  $\errevent(\resdevconst, \lambda; \goodperm)\cap\epevent(\gwconst, \lambda) \cap \recevent(\gwconst)$. 
Since we are on $\epevent(\gwconst, \lambda)$, we can combine Lemma~\ref{lem:trace} with Lemma~\ref{lem:prelimbound} (applied with $\pi = \goodperm$) to deduce (recall $\Delh := \dagadjest - \tB{\estperm}$)
\begin{align*}
\frac{1}{2n}\norm{\sampledagmat \Delh}_{\frob}^2 
+ \pl(\dagadjest) 
\;\le\; &\frac{\gwconst}{2n}\norm{\sampledagmat \Delh }_{\frob}^2 + \gwconst\pl(\Delh) \\
&+ \frac{1}{2n}\norm{\bErr{\goodperm}}_{\frob}^2 
- \frac{1}{2n}\norm{\bErr{\estperm}}_{\frob}^2 
+ \pl(\tB{\goodperm}).
\end{align*}

\noindent
Dropping the prediction loss terms (those involving $\norm{\sampledagmat \Delh}_{\frob}^2$), and using that we are on $\errevent(\resdevconst, \lambda; \goodperm)$ to bound $\frac{1}{2n}\norm{\bErr{\goodperm}}_{\frob}^2 
- \frac{1}{2n}\norm{\bErr{\estperm}}_{\frob}^2$, we have after rearranging,
\begin{align}
\label{eq:thm:main:sparse:3}
\pl(\dagadjest)
&\le 
(1+\resdevconst)\pl(\tB{\goodperm})
+ \gwconst\pl(\dagadjest - \tB{\estperm}) \\
\nonumber
&\le 
(1+\resdevconst)\pl(\tB{\goodperm})
+ \gwconst\pl(\tB{\estperm})
+ \gwconst\pl(\dagadjest).
\end{align}

\noindent
Let $\penconst=2\gwconst/(1-\gwconst)$, so that $\xi/(1-\penconst)=(1+\gwconst)/(1-3\gwconst)$ (cf. \eqref{eq:def:xi}). Furthermore, since $\gwconst<1/3$ by assumption, $\penconst<1$, so that Lemma~\ref{lem:hBlowerbound} implies $\pl(\dagadjest)\ge\penconst\pl(\tB{\estperm})$ which gives (i) in~\eqref{eq:prop:abstract:sparse}. 

\smallskip
Since $\pl(\tB{\estperm})\le(1/\penconst)\pl(\dagadjest)$, the bounds in \eqref{eq:thm:main:sparse:3} imply that
\begin{align*}
\pl(\dagadjest)
&\le 
(1+\resdevconst)\pl(\tB{\goodperm})
+ \frac{\gwconst}{\penconst}\pl(\dagadjest)
+ \gwconst\pl(\dagadjest).
\end{align*}
Rearranging we get
\begin{align*}
\pl(\dagadjest)
&\le  \big[1 - \gwconst(1+\penconst)/\penconst\big]^{-1}(1+\resdevconst)\pl(\tB{\goodperm}).
\end{align*}
We have $[1 - \gwconst(1+\penconst)/\penconst]^{-1}(1+\delta_0) = \frac{2}{1-\gwconst}(1+\frac{10}{\mintraceconst})$, using $\delta_0 = 10/a_2$ and $\delta_1=2\delta/(1-\delta)$ as before. This proves (ii) in~\eqref{eq:prop:abstract:sparse}.
\end{proof}

\subsection{Proof of Theorem~\ref{thm:main:sparse}}
\label{app:thm:main:sparse}

For regularizers satisfying Condition~\ref{condn:pl:basic}, the desired bound follows by taking $\gwconst=(\betaminconst-2)/(3\betaminconst+2)\in(0,1/3)$ in Proposition~\ref{prop:abstract:sparse}, and using Proposition~\ref{prop:GW:capped:ell1} to complete the probability bound.

To deduce the $\ell_0$ bound, consider the case where $\pl$ is also $\ell_0$-compatible. Condition~\ref{condn:betamin} implies 
\begin{align*}
\pl(\tB{\estperm}) \ge \betaminconst\frac{\penderiv^2}{\mineval(\trueCov)}\norm{\tB{\estperm}}_0,
\end{align*}

\noindent
while on the other hand, $\ell_{0}$-compatibility (Definition~\ref{defn:L0compat}) gives $\pl(\tB{\goodperm}) \le \uppflat\lambda^2\norm{\tB{\goodperm}}_0$. Combining these with~\eqref{eq:prop:abstract:sparse} yields
\begin{align*}
\betaminconst\frac{\penderiv^2}{\mineval(\trueCov)}\norm{\tB{\estperm}}_0
&\;\le\;  \frac{1}{\gwconst} \Big( 1+\frac{10}{a_2} \Big) \uppflat \; \lambda^2\norm{\tB{\goodperm}}_0 \\
\implies
\norm{\tB{\estperm}}_0
&\;\le\; \covconst_3\cdot\Big[\frac{\lambda^2}{\penderiv^2} \Big] \norm{\tB{\goodperm}}_0,
\end{align*}
\noindent
as desired. Here, $\covconst_3 = \covconst_3(\trueCov) = \frac{1}{a_1} \big(\frac{3\betaminconst+2}{\betaminconst-2}\big) \big( 1+\frac{10}{a_2} \big)  \uppflat\mineval(\trueCov)$, using our earlier choice of $\delta$.

\subsection{Proof of Theorem~\ref{thm:oracle}}
\label{app:thm:oracle}

Assume that we are on the event $\recevent(\gwconst)\cap\epevent(\gwconst,\lambda)$. As in the proof of Theorem~\ref{thm:main:sparse}, choose $\gwconst=(\betaminconst-2)/(3\betaminconst+2)\in(0,1/3)$. Applying \eqref{eq:empprocmat:estpi} to Lemma~\ref{lem:prelimbound} yields
\begin{align}
\label{eq:thm:oracle:1}
\frac{1}{2n}\norm{\bErr{\estperm}}_{\frob}^2
+ \pl(\dagadjest)
&\le \frac{1}{2n}\norm{\bErr{\pi}}_{\frob}^2
+ \pl(\tB{\pi})
+ \gwconst\pl(\dagadjest-\tB{\estperm})
\end{align}

\noindent
for any $\perm\in\allperms$. Lemma~\ref{lem:trace} combined with and Propositions~\ref{prop:GW:capped:ell1} and~\ref{prop:gwcond:probbound} imply that \eqref{eq:thm:oracle:1} holds with probability at least $1 - \probbound$.

Invoking the subadditivity of $\pl$ and re-arranging, 
\begin{align}
\label{eq:thm:oracle:2}
\frac{1}{2n}\norm{\bErr{\estperm}}_{\frob}^2
+ (1-\gwconst)\pl(\dagadjest)
&\le \frac{1}{2n}\norm{\bErr{\pi}}_{\frob}^2
+ \pl(\tB{\pi})
+ \gwconst\pl(\tB{\estperm}).
\end{align}

\noindent
Now apply Lemma~\ref{lem:hBlowerbound} with $\penconst=2\gwconst/(1-\gwconst)$ to deduce that
\begin{align}
\frac{1}{2n}\norm{\bErr{\estperm}}_{\frob}^2
+ \gwconst\pl(\tB{\estperm})
&\le \frac{1}{2n}\norm{\bErr{\pi}}_{\frob}^2
+ \pl(\tB{\pi}).
\end{align}

\noindent
Taking $u=\sqrt{4(d+1)\log p}$ in Lemma~\ref{lem:pi:tracebound} and observing that $1-h_n(u)>1/2>\gwconst$ (cf. \eqref{eq:def:littleh:bigH} for the definitions of $h_{n}(u)$, $H_{n}(u)$) for $n>5(d+1)\log p$, we have
\begin{align*}
\frac12\tr\tOm{\estperm}
+ \pl(\tB{\estperm})
&\le \frac{1+H_n(u)}{\gwconst}\cdot\Bigg[\frac12\tr\tOm{\pi}
+ \pl(\tB{\pi})\Bigg] \\
&\le \frac1\gwconst\Bigg(1+6\sqrt{\frac{(d+1)\log p}{n}}\Bigg)\cdot\Bigg[\frac12\tr\tOm{\pi}
+ \pl(\tB{\pi})\Bigg].
\end{align*}

\noindent
Combined with the probability bound implied by Propositions~\ref{prop:GW:capped:ell1} and~\ref{prop:gwcond:probbound}, this holds
with probability at least $1 - \probbound$ ($c_1,c_2$ possibly different from above). Since $\perm$ was arbitrary, substituting $\gwconst=(\betaminconst-2)/(3\betaminconst+2)$ completes the proof.

\subsection{Proof of Theorem~\ref{thm:mintr}}
\label{app:thm:mintr}

The desired result follows from combining Theorem~\ref{thm:scorebased} with the following (slightly) more general result:

\begin{prop}
\label{prop:mintr:ident}
Assume that $\pl$ is $\ell_{0}$-compatible and there exists $\trconst>10$ such that
\begin{align}
\label{prop:mintr:ident:1a}
\frac{\tr\tOm{\goodperm}}{\tr\tOm{\perm}}
&\le 1-\trconst\sqrt{\frac{(d+1)\log p}{n}}
\quad\forall\,\perm\in\allperms,
\quad\text{and} \\
\label{prop:mintr:ident:1b}
d
&\le \frac{\trconst - 10}{2(1+\gwconst)\uppflat}\sqrt{\frac{(d+1)\log p}{n}}\cdot\frac{\mineval(\trueCov)}{\lambda^{2}}.
\end{align}

\noindent
If
\begin{align}
\label{prop:mintr:ident:1c}
\lambda 
&\ge C \delta^{-1}  \maxvar \norm{\trueCov}^{1/4} \sqrt{\frac{(d+1)\log p}{n}}
\end{align}

\noindent
then with probability at least $1-\probbound$, it holds that $\estperm\in\argmin_{\perm}\tr\tOm{\perm}$.
\end{prop}

\begin{proof}
\def\minperms{\mathbb{S}_{0}}
Define $\minperms:=\argmin_{\perm}\tr\tOm{\perm}$ and let $\goodperm\in\minperms$ be arbitrary. Assuming that we are on $\epevent(\gwconst, \lambda)$, we can combine Lemma~\ref{lem:trace} with Lemma~\ref{lem:prelimbound} (applied with $\pi = \goodperm$) to deduce (recall $\Delh := \dagadjest - \tB{\estperm}$)
\begin{align*}
\frac{1}{2n}\norm{\sampledagmat \Delh}_{\frob}^2 
+ \pl(\dagadjest) 
\;\le\; &\frac{\gwconst}{2n}\norm{\sampledagmat \Delh }_{\frob}^2 + \gwconst\pl(\Delh) \\
&+ \frac{1}{2n}\norm{\bErr{\goodperm}}_{\frob}^2 
- \frac{1}{2n}\norm{\bErr{\estperm}}_{\frob}^2 
+ \pl(\tB{\goodperm}).
\end{align*}

\noindent
This implies, in particular, that 
\begin{align}
\label{eq:prop:eqvar:ident:2}
\pl(\dagadjest) 
+ \frac{1}{2n}\norm{\bErr{\estperm}}_{\frob}^2 
- \frac{1}{2n}\norm{\bErr{\goodperm}}_{\frob}^2 
\le \gwconst\pl(\Delh) + \pl(\tB{\goodperm}).
\end{align}

Suppose $\estperm\notin\minperms$, so that $\tr\tOm{\estperm}>\tr\tOm{\goodperm}$. Then Lemma~\ref{lem:pi:tracebound} combined with \eqref{prop:mintr:ident:1a} implies that for any $0<u<n/\sqrt{n+1}$,
\begin{align}
\nonumber
\frac{1}{2n}\norm{\bErr{\estperm}}_{\frob}^2
&- \frac{1}{2n}\norm{\bErr{\goodperm}}_{\frob}^2 \\
\nonumber
&\ge \frac12\tr\tOm{\estperm}\Big[1-h_{n}(u)\Big]
- \frac12\tr\tOm{\goodperm}\Big[1+H_{n}(u)\Big] \\
\nonumber
&= \Bigg[\frac12\tr\tOm{\estperm} - \frac12\tr\tOm{\goodperm}\Bigg]
- \frac12\tr\tOm{\estperm}\,h_n(u) - \frac12\tr\tOm{\goodperm}\,H_n(u) \\
\nonumber
&> \Bigg[\frac12\tr\tOm{\estperm} - \frac12\tr\tOm{\goodperm}\Bigg]
- \frac12\tr\tOm{\estperm}\Big(h_n(u) + H_n(u)\Big) \\
\nonumber
&> \frac{\trconst - 10}{2}\sqrt{\frac{(d+1)\log p}{n}}\tr\tOm{\estperm} \\
\label{eq:prop:eqvar:ident:3}
&\ge \frac{\trconst - 10}{2}\sqrt{\frac{(d+1)\log p}{n}}p\,\mineval(\trueCov).
\end{align}

\noindent
Note that here we have used $h_n(u)+H_n(u)\le 5u/\sqrt{n}$ and the (arbitrary) choice of $u=\sqrt{4(d+1)\log p}$.

Combining \eqref{eq:prop:eqvar:ident:2} and \eqref{eq:prop:eqvar:ident:3}, we have
\begin{align*}
\pl(\dagadjest) 
+ \frac{\trconst - 10}{2}\sqrt{\frac{(d+1)\log p}{n}}\,p\,\mineval(\trueCov)
&< \gwconst\pl(\Delh) + \pl(\tB{\goodperm}) \\
&\le \gwconst\pl(\dagadjest) + \gwconst\pl(\tB{\estperm}) + \pl(\tB{\goodperm}).
\end{align*}

\noindent
Thus
\begin{align*}
\frac{\trconst - 10}{2}\sqrt{\frac{(d+1)\log p}{n}}\,p\,\mineval(\trueCov)
< \gwconst\pl(\tB{\estperm}) + \pl(\tB{\goodperm})
&\le (1+\gwconst)\uppflat\lambda^{2}\,dp \\
\implies
\frac{\trconst - 10}{2}\sqrt{\frac{(d+1)\log p}{n}}\cdot\mineval(\trueCov)
&< (1+\gwconst)\uppflat\lambda^{2}\,d.
\end{align*}

\noindent
In the second inequality we have invoked $\ell_{0}$-compatibility of $\pl$ (Definition~\ref{defn:L0compat}). But this contradicts \eqref{prop:mintr:ident:1b}, whence $\estperm\in\minperms$, as desired. The probability bound on $\pr[\epevent(\gwconst, \lambda)]$ follows by combining Propositions~\ref{prop:GW:capped:ell1} and~\ref{prop:gwcond:probbound} along with \eqref{prop:mintr:ident:1c}.
\end{proof}

\subsection{A bound on the sample residuals}
\label{app:resbound}

In this section, we prove the following result, which is used in the proof of Proposition~\ref{prop:abstract:sparse}:

\begin{prop}
\label{prop:mintrace:prob}
Assume $n>4(C+1)(d+1)\log p$ for some $C>0$ and let $\goodperm$ be a minimum-trace permutation such that 
\begin{align}
\label{eq:prop:mintrace:prob:bound}
\frac{\pl(\tB{\goodperm})}{\tr\tOm{\goodperm}}
\ge \frac{1}{\resdevconst}\sqrt{\frac{50(C+1)(d+1)\log p}{n}}.
\end{align}

\noindent
Then for any $\resdevconst> 0$, $\pr(\errevent(\resdevconst,\lambda;\goodperm))\ge 1 - 2e^{-C(d+1)\log p}$, i.e.
\begin{align*}
\pr\Bigg(
\frac{1}{2n}\norm{\bErr{\goodperm}}_{\frob}^2 - \frac{1}{2n}\norm{\bErr{\estperm}}_{\frob}^2
> \resdevconst\pl(\tB{\goodperm})
\Bigg)
\le 2e^{-C(d+1)\log p}.
\end{align*}
\end{prop}

\noindent
Define two functions by
\begin{align}
\label{eq:def:littleh:bigH}
h_{n}(u)
:= -\frac{u^{2}}{n} + \frac{2u}{\sqrt{n+1}} + \frac{1}{n+1}, \quad
H_{n}(u)
:= \frac{u^{2}}{n} + \frac{2u}{\sqrt{n}}.
\end{align}

\noindent
These functions bound the deviations in the normed residuals $\sampledagerrcol_{j}(\pi)$, and will be used repeatedly in the sequel. We note that
\begin{align}
\label{eq:prop:mintrace:prob:proof:2}
H_{n}(u) + h_{n}(u)
\le \frac{5u}{\sqrt{n}}, 
\quad u\ge n^{-1/2}.
\end{align}

\begin{lemma}
\label{lem:gausserrbound}
Suppose $\samplelinerr\sim\normalN_{n}(0,\sigma^{2}I_{n})$. Then for any $0<u<n/\sqrt{n+1}$,
\begin{align}
\label{eq:lem:gausserrbound}
\sigma^{2}\Big(1-h_{n}(u)\Big)
\le \frac{1}{n}\norm{\samplelinerr}_2^2
&\le \sigma^{2}\Big(1+H_{n}(u)\Big)
\end{align}

\noindent
with probability at least $1-2e^{-u^{2}/2}$.
\end{lemma} 

\begin{proof}
For $z\sim\normalN_{n}(0,I_{n})$, we have the following useful bounds \citep[see, e.g.,][Corollary~1.2]{gordon1988}:
\begin{align*}
\frac{n}{\sqrt{n+1}}
\le \E \norm{z}_{2}
= \sqrt{2}\frac{\Gamma(\frac{n+1}{2})}{\Gamma(\frac{n}{2})}
\le \sqrt{n}.
\end{align*}

\noindent
Gaussian concentration implies that for any $u>0$, 
both
\begin{align*}
    \big\{\Norm{\samplelinerr}_{2} \le \sigma (n/\sqrt{n+1} - u) \big\}, \quad 
    \text{and} \quad
    \big\{\Norm{\samplelinerr}_{2} \ge \sigma (\sqrt{n} + u) \big\}
\end{align*}
hold with probability at most $ e^{-u^{2}/2}$. Thus,
\begin{align}
\label{eq:lem:gausserrbound:proof:1}
\pr\bigg(
\sigma^{2}\Big(\frac{n}{\sqrt{n+1}} - u\Big)^{2}
\le \Norm{\samplelinerr}_{2}^{2} 
\le \sigma^{2}\big(\sqrt{n}+u\big)^{2}
\bigg)
&\ge 1 - 2e^{-u^{2}/2}.
\end{align}

\noindent
Re-writing \eqref{eq:lem:gausserrbound:proof:1} using \eqref{eq:def:littleh:bigH} yields the desired result.
\end{proof}

\begin{lemma}
\label{lem:pi:tracebound}
Suppose $\sampledagmat\iid\normalN_{p}(0,\trueCov)$. Then for any $\pi\in\allperms$ and $0<u<n/\sqrt{n+1}$,
\begin{align}\label{eq:temp:4858}
\frac12\tr\tOm{\pi}\Big(1-h_{n}(u)\Big)
\le \frac{1}{2n}\norm{\bErr{\pi}}_{\frob}^2
&\le \frac12\tr\tOm{\pi}\Big(1+H_{n}(u)\Big)
\end{align}

\noindent
with probability at least $1-2p\binom{p}{d}e^{-u^{2}/2}$.
\end{lemma}

\begin{proof}
Note that for any $\pi\in\allperms$,
\begin{align}
\label{eq:lem:pi:tracebound:1}
\frac{1}{2n}\norm{\bErr{\pi}}_{\frob}^2
= \frac{1}{2n}\sum_{j=1}^{p} \Norm{\sampledagerrcol_{j}(\pi)}_{2}^{2}
= \frac{1}{2n}\sum_{j=1}^{p} \Norm{\sampledagerrcol_{j}(\dagcand_{j}(\pi))}_{2}^{2}.
\end{align}

\noindent
Thus it suffices to bound the deviations in $\Norm{\sampledagerrcol_{j}(S)}_{2}$ for $S\subset[p]_{j}$. Consider the following events
\begin{align*}
    \errevent_j(S) := \bigg\{ \frac{\nhbdvar_{j}^{2}(\dagcand)}{2}\Big(1-h_{n}(u)\Big)
\le \frac{1}{2n}\norm{\sampledagerrcol_{j}(S)}_2^2
\le \frac{\nhbdvar_{j}^{2}(\dagcand)}{2}\Big(1+H_{n}(u) \Big) \bigg\}
\end{align*}
and let $\errevent := \bigcap_{j=1}^{p}\bigcap_{S\subset[p]_{j}} \errevent_j(S)$.
By Lemma~\ref{lem:gausserrbound}, we have $\pr(\errevent_j(S)) \ge 1-2e^{-u^{2}/2}$, for all $S \in [p]_j$. By a monotonicity argument (cf.~\eqref{eq:invariant:nhbderr}), we have $\errevent =  \bigcap_{j=1}^{p} \bigcap_{S \in \nhbdsupp_{j}(\trueCov)} \errevent_j(M_j(S))$.
Applying the union bound and using~\eqref{eq:def:maxdeg:betamin},
\begin{align}
\label{eq:probboound:1}
\pr(\errevent^{c})
&\le 2p\binom{p}{d}e^{-u^{2}/2}.
\end{align}
Summing the inequalities defining $\errevent_j(S_j(\pi))$, over $j$, we conclude that~\eqref{eq:temp:4858} holds on $\errevent$. The proof is complete.
\end{proof}

Consider the (random) collection of permutations 
\begin{align*}
\goodperms = \goodperms(\resdevconst;u)
:= \bigg\{
\pi\in\allperms : \frac12\tr\tOm{\pi}\Big[1&+H_{n}(u)\Big]
\\- \frac12\tr\tOm{\estperm}\Big[1&-h_{n}(u)\Big] 
\le \resdevconst\pl(\tB{\pi})
\bigg\}.
\end{align*}

\begin{lemma}
\label{lem:trevent}
For any $\pi\in\goodperms(\resdevconst;u)$ and $0<u<n/\sqrt{n+1}$, we have
\begin{align*}
\pr\bigg(
\frac{1}{2n}\norm{\bErr{\pi}}_{\frob}^2 - \frac{1}{2n}\norm{\bErr{\estperm}}_{\frob}^2
> \resdevconst\pl(\tB{\pi})
\bigg)
\le 2p\binom{p}{d}e^{-u^{2}/2}.
\end{align*}
\end{lemma}

\begin{proof}
Lemma~\ref{lem:pi:tracebound} implies that
\begin{align*}
\frac{1}{2n}\norm{\bErr{\pi}}_{\frob}^2 - \frac{1}{2n}\norm{\bErr{\estperm}}_{\frob}^2
\le \frac12\tr\tOm{\pi}\Big[1+H_{n}(u)\Big]
- \frac12\tr\tOm{\estperm}\Big[1-h_{n}(u)\Big]
\end{align*}

\noindent
with probability at least $1-2p\binom{p}{d}e^{-u^{2}/2}$. Since $\pi\in\goodperms$, the right-side is bounded above by $\resdevconst\pl{\tB{\pi}}$ by definition, which establishes the claim.
\end{proof}

\begin{lemma}
\label{lem:littleh:pos}
$1-h_{n}(u)>0$ for all $u\ne0$, $n>0$.
\end{lemma}

\begin{proof}
Since $(u+\sqrt{n})^{2}+1>0$, re-writing this inequality yields
\begin{align*}
\frac{u^{2}}{n}+1
> \frac{2u}{\sqrt{n}} + \frac1n
&>  \frac{2u}{\sqrt{n+1}} + \frac1{n+1} \\
\implies
1 + \frac{u^{2}}{n} - \frac{2u}{\sqrt{n+1}} - \frac1{n+1}
&> 0
\end{align*}

\noindent
Comparing with \eqref{eq:def:littleh:bigH} yields the claim.
\end{proof}

\begin{proof}[Proof of Proposition~\ref{prop:mintrace:prob}]
Lemma~\ref{lem:trevent} implies that for a choice of \linebreak[4] $u=\sqrt{2(C+1)(d+1)\log p}$, we have
\begin{align*}
\pr\bigg(
\frac{1}{2n}\norm{\bErr{\pi}}_{\frob}^2 - \frac{1}{2n}\norm{\bErr{\estperm}}_{\frob}^2
> \resdevconst\pl(\tB{\pi})
\bigg)
&\le 2p\binom{p}{d}e^{-(C+1)(d+1)\log p} \\
&\le 2e^{-C(d+1)\log p}
\end{align*}

\noindent
for any $\pi\in\goodperms$. Thus the claim will follow if we can show that $\goodperm\in\goodperms$. Note that
\begin{align*}
\nonumber
\tr\tOm{\goodperm}\Big[1+H_{n}(u)\Big]
&- \tr\tOm{\estperm}\Big[1-h_{n}(u)\Big] \\
&\stackrel{(i)}{\le} \tr\tOm{\goodperm}\Big[H_{n}(u) + h_{n}(u)\Big] \notag \\
&\stackrel{(ii)}{\le}  \tr\tOm{\goodperm}\sqrt{\frac{50(C+1)(d+1)\log p}{n}} \\
&\stackrel{(iii)}{\le} \resdevconst\pl(\tB{\goodperm}),
\end{align*}
where (i) follows from $\tr\tOm{\goodperm}\le\tr\tOm{\estperm}$ and Lemma~\ref{lem:littleh:pos}, (ii) follows by using~\eqref{eq:prop:mintrace:prob:proof:2} with $u=\sqrt{2(C+1)(d+1)\log p}$,
and (iii) follows from assumption~\eqref{eq:prop:mintrace:prob:bound}.
Hence, $\goodperm\in\goodperms$ and the proof is complete.
\end{proof}

\bibliography{highdimdag-arxiv-bib-v3}
\bibliographystyle{abbrvnat}
\label{bib:bib}

\end{document}